\documentclass{article}

\usepackage{amsmath,amssymb,amsthm,bbm,graphicx,geometry,lmodern,mathtools,subcaption}

\usepackage[utf8]{inputenc}

\usepackage[numbers,sort&compress]{natbib}

\usepackage{titling}
\usepackage{bbm}
\usepackage{calc,extarrows}
\usepackage{enumitem}
\usepackage{graphicx}
\usepackage{fullpage}%need to delete
\usepackage{url}
\usepackage[dvipsnames]{xcolor}
\usepackage{tikz-cd}
\usepackage{float}
\usepackage{todonotes}
\usepackage{orcidlink}
\usepackage{url}
\usepackage{comment}
\usepackage{setspace}

\usepackage{tcolorbox} %%%%%%%%Colorbox
\newtcolorbox{tbox}{colback=bisque!15!white,colframe=bisque!75!black}

\usepackage{pgfplots}
\usetikzlibrary{arrows.meta}
\usetikzlibrary{plotmarks}
\pgfplotsset{compat=1.15}
\usepackage{mathrsfs}
\usetikzlibrary{arrows}

\definecolor{bluepigment}{rgb}{0.2, 0.2, 0.6}
\definecolor{amethyst}{rgb}{0.6, 0.4, 0.8}
\definecolor{asparagus}{rgb}{0.53, 0.66, 0.42}
\definecolor{orange-red}{rgb}{1.0, 0.27, 0.0}
\definecolor{coffee}{rgb}{0.88, 0.82 ,0.77}
\definecolor{champagne}{rgb}{0.96, 0.96, 0.86}%{0.97, 0.91, 0.81}
\definecolor{deepchampagne}{rgb}{0.98, 0.84, 0.65}
\definecolor{bittersweet}{rgb}{0.8, 0.58, 0.46}
\definecolor{antiquebrass}{rgb}{0.8, 0.58, 0.46}
\definecolor{burlywood}{rgb}{0.87, 0.72, 0.53}
\definecolor{darkchampagne}{rgb}{0.76, 0.7, 0.5}
\definecolor{bisque}{rgb}{1.0, 0.89, 0.77}
\definecolor{biscuit}{RGB}{244, 164, 96}%{239,204,162}
\definecolor{darkchestnut}{rgb}{0.75, 0.3, 0.3}%{0.6, 0.41, 0.38}

\usepackage{hyperref} % For clickable links and references
\hypersetup{
	colorlinks=true,
	linkcolor=amethyst,
	citecolor=orange-red,
	filecolor=black,      
	urlcolor=black,
}

\newcommand{\N}{\mathbb{N}}
\newcommand{\Z}{\mathbb{Z}}
\newcommand{\R}{\mathbb{R}}

\newcommand{\calC}{\mathcal{C}}

\newcommand{\calF}{\mathcal{F}}
\newcommand{\calE}{\mathcal{E}}

\newcommand{\calH}{\mathcal{H}}
\newcommand{\calI}{\mathcal{I}}
\newcommand{\calQ}{\mathcal{Q}}
\newcommand{\calP}{\mathcal{P}}

\newcommand{\calZ}{\mathcal{Z}}

\newcommand{\x}{\boldsymbol{x}}
\newcommand{\y}{\boldsymbol{y}}

\renewcommand{\o}{\boldsymbol{o}}

\newcommand{\A}{\boldsymbol{\mathsf A}}
\newcommand{\B}{\boldsymbol{\mathsf B}}
\newcommand{\U}{\boldsymbol{\mathsf U}}
\renewcommand{\L}{\boldsymbol{\mathsf L}}
\newcommand{\X}{\boldsymbol{\mathsf X}}
\newcommand{\Y}{\boldsymbol{\mathsf Y}}
\newcommand{\T}{\boldsymbol{\mathsf T}}

\newcommand{\ZZe}{\mathcal Z^2}
\newcommand{\Oi}{\mathcal{O}}
\newcommand{\Vi}{\mathcal{V}}

\newcommand{\Exp}{\mathrm{Exponential}}

\newcommand{\Ber}{\mathrm{Bernoulli}}

\newcommand{\Geo}{\mathrm{Geometric}}
\newcommand{\dd}{\mathrm{d}}

\newcommand{\ind}[1]{\mathbbm{1}{\{#1\}}}

\DeclareMathOperator{\E}{\mathbb{E}}
\renewcommand{\P}{\operatorname{\mathbb{P}}}
\DeclareMathOperator{\Var}{\mathbb{V}\mathrm{ar}}

\newcommand{\e}{\mathrm{e}}

\let\oldepsilon\epsilon
\let\epsilon\varepsilon
\let\varepsilon\oldepsilon
\renewcommand{\phi}{\varphi}
\theoremstyle{plain}

\newtheorem{theorem}{Theorem}[section]
\newtheorem{lemma}[theorem]{Lemma}

\newtheorem{prop}[theorem]{Proposition}

\theoremstyle{definition}

\numberwithin{equation}{section}

\theoremstyle{remark}
\newtheorem{remark}[theorem]{Remark}
\title{Throughput in inhomogeneous planar drainage networks}

\author{
Partha Pratim Ghosh
    \orcidlink{0000-0002-4801-4538}
    \thanksgap{0.05ex}
    \thanks{Technische Universit\"at Braunschweig, Universit\"atsplatz 2, 38106 Braunschweig, Germany}
    \\p.pratim.10.93@gmail.com
\and
Benedikt Jahnel 
    \orcidlink{0000-0002-4212-0065}
    \thanksgap{0.05ex}
    \thanksmark{1}
    \thanksgap{0.2ex} 
    \thanks{Weierstrass Institute for Applied Analysis and Stochastics, Mohrenstr.\ 39, 10117 Berlin, Germany} 
    \\ benedikt.jahnel@tu-braunschweig.de
\and
Yannic Steenbeck
\thanksgap{0.05ex}
    \thanksmark{1}
\\ yannic.steenbeck@tu-braunschweig.de
}

\date{\today}

\begin{document}
\maketitle

%\begin{spacing}{0.9}
\begin{abstract}
\noindent We consider navigation schemes on planar diluted lattices and semi lattices with one discrete and one continuous component. More precisely, nodes that survive inhomogeneous Bernoulli site percolation, or are placed as inhomogeneous Poisson points on shifted copies of $\Z$, forward their individually generated traffic to their respective closest neighbors to the left in the next layer. The resulting drainage network is a tree and we study the amount of traffic that goes through an increasing window at the origin. Our main results show that, properly rescaled, the total traffic, jointly with the total length of the contributing tree part, converges to the area under a time-inhomogeneous Brownian motion until it hits zero. The hitting time corresponds to the limiting maximal path length. 

\smallskip
\noindent\footnotesize{{\textbf{AMS-MSC 2020}: Primary: 60D05, 60K30, 60F05; Secondary: 60K35, 90B20}

}
\end{abstract}
%\end{spacing}
%\tableofcontents

%\newpage
%\listoftodos
%\newpage

%%%%%%%%%%%%%%%%%%%%%%%%%%%%%%%%%%%%%%
%%%%%%%%%%%%%%%%%%%%%%%%%%%%%%%%%%%%%%
%%%%%%%%%%%%%%%%%%%%%%%%%%%%%%%%%%%%%%

\section{Introduction}
\label{Sec:intro}
In this manuscript, we contribute to a line of research concerning drainage networks, see~\cite{coletti2009scaling,ferrari2005two,penrose2010limit}. These are a type of oriented percolation models, where edges between points in a point process are opened in a given direction according to certain navigation rules. The interpretation of these networks is manifold and ranges from opinion dynamics such as the voter model~\cite{arratia1981coalescing,arratia1979coalescing}, geology as for example the formation of river deltas~\cite{RoSaSa16} or communication systems such as device-to-device networks~\cite{hirsch2017traffic,keeler2010stochastic}, and many more~\cite{ferrari2004poisson}. These models can be considered in continuum space, e.g., $\R^d$, where nodes are given by a stationary Poisson point process or on lattices, e.g., $\Z^d$, that may be subject to some dilution as for example Bernoulli cite percolation. Then, each node picks, either deterministically or according to some rule, that may include additional randomness, precisely one successor node in a certain direction and we draw a directed edge towards that node. In the planar case, for example, the navigation may force the directed edges to face towards the west. 

As a result, in infinite volume, we see a family of trees, and, roughly speaking, many of these models have a universal scaling limit, given by the Brownian web, see for example~\cite{fontes2004brownian}. Let us mention here that the spatial dimension plays a big role in the analysis as can be seen for example in~\cite{gangopadhyay2004random}, where the authors establish almost-sure uniqueness of the tree in dimensions $d=2,3$ and almost-sure existence of infinitely many trees in dimensions $d\ge 4$ for a lattice version of the drainage network. Also many more characteristics of these networks have been successfully investigated, such as for example the scaling limits of the length, the width, and the cluster process of the origin's tree, see~\cite{RoSaSa16}. However, to the best of our knowledge, this has been done only in an space-time homogeneous setting, while our results deal with inhomogeneities. 

In this manuscript, we focus on a planar drainage network, which is either discrete or semi-discrete in the sense that nodes are either in $\Z^2$ or in $\Z\times \R$. Inspired by traffic networks in communication systems, we are interested in the amount of traffic that flows through an increasing window, where traffic is accumulated along drainage paths.  
To be specific, for example in the case of the semi lattice $\Z\times \R$, vertex locations are represented by an inhomogeneous but asymptotically flattening Poisson point process with intensity function of the form $\lambda_\ell(\cdot)=\lambda(\cdot/\ell^2)$, where $\lambda(\cdot)$ is smooth and $\ell \to \infty$. A similar flattening of the intensity was also considered in~\cite{hirsch2017traffic}. Under this scaling, point process realizations tend to behave like a homogeneous process within compact regions. Flattening the intensity corresponds to effectively \emph{`zooming out'}, so that local fluctuations of the intensity tend to average out, allowing the global geometry of the environment to emerge. While this scaling spreads the variation of $\lambda$ over a larger area, it preserves the relative differences in intensity. Local features (such as dense urban areas) still matter at small scales, but the flattening reveals how these features interact with the surrounding sparse regions. In order to give an idea, we now present the simplest version of our result, which will later be proven as a special case in a more general framework.

%%%%%%%%%%%%%%%%%%%%%%%%%%%%%%%%%%%%%%
%%%%%%%%%%%%%%%%%%%%%%%%%%%%%%%%%%%%%%
%%%%%%%%%%%%%%%%%%%%%%%%%%%%%%%%%%%%%%

\subsection{Scaling limits for throughputs in planar lattices}
Consider the even planar lattice $\ZZe := \{\x=(x_1,x_2) \in \Z^2 \colon x_1+x_2\in 2\Z\}$ and equip each {\em node} $\x\in \ZZe$ with an independent $\Ber(1/2)$ random variables $\xi_{\x}$ that represents a {\em choice} to go either up or down. Now, in case $\xi_{\x}=1$, we draw a directed edge from $\x=(x_1,x_2)$ to $(x_1-1, x_2+1)$, and if $\xi_{\x}=0$, we draw a {\em directed edge} from $\x=(x_1,x_2)$ to $(x_1-1, x_2-1)$. Next, imagine that each point in $\ZZe$ initially carries a unit amount of {\em traffic} that flows along the directed edges towards the left. For $\x=(x_1,x_2)$, consider the {\em slid}
\begin{align}\label{Def:Slid}
I_\ell(\x):=\{x_1\}\times [x_2-\ell/2,x_2+\ell/2],\qquad \ell> 0,
\end{align}
and let $T_\ell(\x)$ represent the total amount of traffic flowing through the slit $I_\ell(\x)$. Let further $\tau_\ell(\x)$ denote the number of points in the longest path that carries traffic flowing through the slit $I_\ell(\x)$. Denote by $\{B_{ t}\colon t\geq 0\}$ the standard Brownian motion starting at $0$ and by
\[
        \varrho:=\inf\{t\geq 0\colon 1+B_{ t}=0 \},
\]
the first hitting time at $0$ for a standard Brownian motion started at $1$. Then, we have the following scaling limit. 
\begin{theorem}\label{Thm:Pure:convergence-Tl-taul}
For any $\x\in\ZZe$,
\[
\Big( \frac{\tau_\ell(\ell^2\x)}{\ell^2}, \frac{T_\ell(\ell^2\x)}{\ell^3}\Big) \xlongrightarrow{ d } \Big(  \frac{\varrho}{2}, \frac{1}{4}\int_{0}^{\varrho} (1+B_{t}) \,\dd t \Big),\qquad\text{ as }\ell\to\infty. 
\]
\end{theorem}
In words, jointly, the total traffic in the networks that is transported through a growing one-dimensional window together with the maximal length of a transport path, properly rescaled, behaves like a Brownian motion stopping time and the associated area under its path. 
Let us note that the density of the limiting distribution in Theorem~\ref{Thm:Pure:convergence-Tl-taul} can be obtained from~\cite[Equation~(12)]{KeMa05}.

The following result provides a more detailed analysis for the asymptotics of $\tau_\ell$.
\begin{theorem}\label{Thm:Pure:asymp-taul}
For any $\x\in\ZZe$,
\[
\P(\tau_\ell(\x) =n) = \frac{\ell}{2n} \binom{2n}{n+\ell/2} \frac{1}{4^n},\qquad n\in\N.
\]
In particular, for any $f$ satisfying $\lim_{\ell\to\infty}\ell^2/f(\ell)\to0$, 
\[
\lim_{\ell\to\infty}\frac{f(\ell)^{3/2}}{\ell}\P\left(\tau_\ell(\x) =\lfloor f(\ell) \rfloor\right) = \frac{1}{2{ \sqrt \pi}}\qquad \text{and}\qquad
\lim_{\ell\to\infty}\frac{\sqrt{f(\ell)}}{\ell}\P\left(\tau_\ell(\x) \geq f(\ell)\right) = \frac{1}{{4 \sqrt \pi}}.
\]
\end{theorem}

%%%%%%%%%%%%%%%%%%%%%%%%%%%%%%%%%%%%%%
%%%%%%%%%%%%%%%%%%%%%%%%%%%%%%%%%%%%%%
%%%%%%%%%%%%%%%%%%%%%%%%%%%%%%%%%%%%%%

\subsection{Motivation via opinion dynamics}
As indicated earlier, we present versions of the Theorems~\ref{Thm:Pure:convergence-Tl-taul} and~\ref{Thm:Pure:asymp-taul} in two generalized settings. First, instead of $\ZZe$, we base the traffic paths on inhomogeneous independent Bernoulli site percolation fields with an appropriately adjusted navigation rule that connects nearest-neighbor occupied nodes, see Section~\ref{Subsec:diluted_lattice} for details. We briefly sketch here a possible interpretation for the results in this setting. 

Imagine the set $\ZZe$ to represent {\em individuals} holding two possible {\em opinions}, say $\mathcal V,\mathcal W$. The dilution 
variables now independently assign to each individual $y$ and time $x$ the attribute of being an {\em influencer}, with probability $p_\ell(x,y)\in [0,1]$, or not, where $p_\ell(\cdot):=p(\cdot/\ell^2)$ and $\ell\ge 0$ is our scaling parameter. In particular, the horizontal axis represents time and each individual undergoes a possible change of being an influencer in the sense that the vertex $(x+2,y)$ represents the same individual as the vertex $(x,y)$ but one time step further and with a possibly changed influencer status, depending on the independently drawn Bernoulli random variable with the parameter $p_\ell(x+2,y)$. The navigation scheme can now be interpreted as follows. Each individual $y$ (be it an influencer and not) updates its opinion by the opinion of its closest influencer (with a fair random tiebreaker) independent of everything else. Then, we consider a neighborhood in the form of an initial interval of length $\ell\in 2\N+1$, centered around the origin $\o$, that holds one opinion, say $\mathcal V$. Some of the individuals maybe influencers, and they are able to pass along their opinion over time. Then, $\tau_\ell(\o)$ represents the random persistence time of the opinion $\mathcal V$ originating from the $\ell$-neighborhood. The first theorem, in the form of Theorem~\ref{Thm:diluted:convergence-Tl-taul} below, then states that $\tau_\ell(\o)/\ell^2$ converges to a stopping time of a time-changed Brownian motion.
  
In our second generalization, we consider a population represented by $\R$. At each time $x \in \Z$, a random set of individuals, distributed according to a Poisson point process on $\R$ with intensity $\lambda_\ell(x,\cdot)$, become influencers.
The navigation scheme can now be interpreted as follows. Each individual $y$, whether an influencer or not, updates its opinion by adopting the opinion of its closest influencer (with a fair random tiebreaker), independently of everything else. Note that the individuals requiring a tiebreaker are never influencers with probability one, and therefore the tiebreaker does not affect the evolution of the process.
As before, we consider an initial neighborhood in the form of an interval of length $\ell > 0$, centered at the origin $\o$, where all individuals hold a common opinion, denoted by $\mathcal{V}$. Some individuals in this neighborhood may themselves be influencers, and they have the ability to spread this opinion over time.
We then define $\tau_\ell(\o)$ as the random \emph{persistence time} of the opinion $\mathcal{V}$ originating from this $\ell$-neighborhood. Theorem~\ref{Thm:Poisson:convergence-Tl-taul} below states that the rescaled persistence time $\tau_\ell(\o)/\ell^2$ converges to a stopping time of a time-changed Brownian motion.

%%%%%%%%%%%%%%%%%%%%%%%%%%%%%%%%%%%%%%
%%%%%%%%%%%%%%%%%%%%%%%%%%%%%%%%%%%%%%
%%%%%%%%%%%%%%%%%%%%%%%%%%%%%%%%%%%%%%

\subsection{Organization of the manuscript}
The manuscript is further organized as follows. In the next Section~\ref{Sec:settings-and-main-results} we present our setting and main results for two situations, diluted lattices and semi lattices. In Section~\ref{Sec:Proof_Strategy} we present the strategy of the proofs focusing on the semi-lattice case without vertical inhomogeneities. Finally, in Section~\ref{Sec:Proofs} we exhibit all remaining proofs.

%%%%%%%%%%%%%%%%%%%%%%%%%%%%%%%%%%%%%%
%%%%%%%%%%%%%%%%%%%%%%%%%%%%%%%%%%%%%%
%%%%%%%%%%%%%%%%%%%%%%%%%%%%%%%%%%%%%%

\section{Setting and main results}
\label{Sec:settings-and-main-results}
We concentrate on two situations. First, we consider the diluted lattice in which nodes in $\ZZe$ are removed according to a Bernoulli site-percolation process, see Section~\ref{Subsec:diluted_lattice}. Second, we consider nodes in $\Z\times\R$ given by independent Poisson point processes on $\R$ for each horizontal layer $x_1\in \Z$, see Section~\ref{Subsec:Continuum_lattice}. 

%%%%%%%%%%%%%%%%%%%%%%%%%%%%%%%%%%%%%%
%%%%%%%%%%%%%%%%%%%%%%%%%%%%%%%%%%%%%%
%%%%%%%%%%%%%%%%%%%%%%%%%%%%%%%%%%%%%%

\subsection{Diluted lattices}
\label{Subsec:diluted_lattice}
We consider the situation where not all points in $\ZZe$ can generate and carry traffic. More precisely, let $p \colon \R^2 \to (0,1]$ be a real-valued continuous function, which is $\Xi$-Lipschitz continuous in the second argument for some $\Xi > 0$, and satisfies $p(\x) \geq p_{\min}$ for all $\x \in \R^2$, for some $p_{\min} > 0$. We introduce a scaling parameter $\ell > 0$ and let $\{\zeta_{\x} \colon \x \in \ZZe\}$ be a collection of independent $\Ber(p_\ell(\x))$ random variables, where $\x \mapsto p_\ell(\x) = p(\x / \ell^2)$ defines a continuous {\em thinning field} on $\R^2$.
In case $\zeta_{\x}=1$, the node $\x\in\ZZe$ is called {\em occupied} and it can generate and carry traffic. Otherwise, if $\zeta_{\x}=0$, the node $\x$ is called {\em vacant} and is removed from the drainage network. Note that $p(\x)\equiv1$ corresponds to the situation described in the introduction. We denote the set of all occupied nodes as $\Oi$ and the set of all vacant nodes as $\Vi$. For $\x=(x_1,x_2)$ in $\Oi$, the navigation is as follows. We draw a directed edge toward $(x_1-1,y)\in\Oi$ if $|x_2-y|< |x_2-z|$ for all $(x_1-1,z)\in \Oi\setminus \{(x_1-1,y)\}$. In case there are two nodes $(x_1-1,y)\neq (x_1-1,z)$ such that $|x_2-y|=|x_2-z|$ and $|x_2-y|< |x_2-a|$ for all other $(x_1-1,a)\in\Oi\setminus \{(x_1-1,y), (x_1-1,z)\}$, we use the additional independent field of fair Bernoullis $\{\xi_{\x}\colon \x\in \ZZe\}$ to break the tie, just like in the case non-diluted case described in the introduction: If $\xi_{\x}=1$, we pick the upper neighbor and otherwise the lower neighbor, see Figure~\ref{Fig:second_model} for an illustration.
%%%%%%%%%%%%%%%%%%%%%%%%%%%%%%%%%%%%%%%%%%%
%Figure
%%%%%%%%%%%%%%%%%%%%%%%%%%%%%%%%%%%%%%%%%%
\begin{figure}[t]
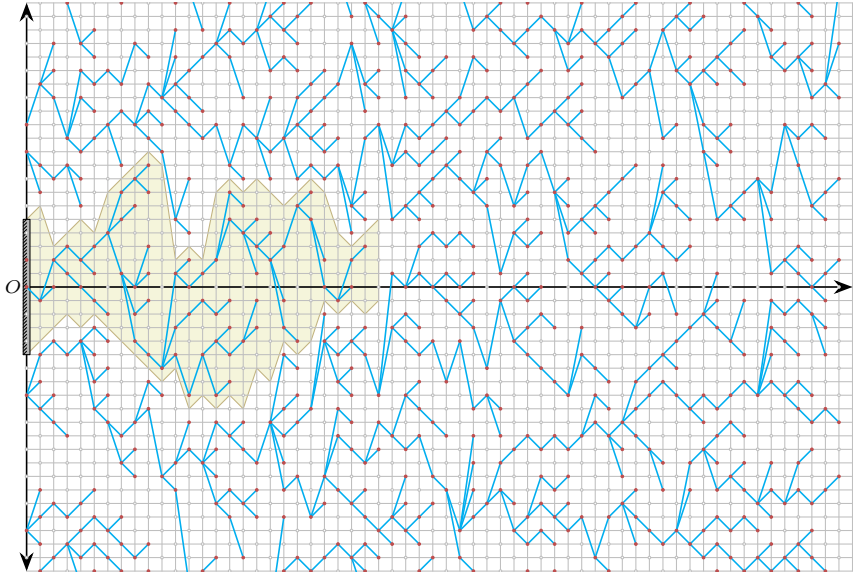

\centering

   \caption{Simulation of traffic flow in the diluted lattice model with $p=1/2$. Occupied vertices are indicated as red dots with their associated unique edges towards the left, in blue. The slid is positioned at the origin as a dashed area and the beige region corresponds to the set of nodes that transmit their traffic through the slid.}
    \label{Fig:second_model}
\end{figure}
%%%%%%%%%%%%%%%%%%%%%%%%%%%%%%%%%%%%%%%%%%
%\begin{figure}[t]
%    \centering
%    \includegraphics{fig1}
%    \caption{Simulation of traffic flow in the diluted lattice model with $p=1/2$. Occupied vertices are indicated as black dots with their associated unique edges towards the left, in blue. The slid is positioned at the origin as a dashed area and the brown region corresponds to the set of nodes that transmit their traffic through the slid.}
%   \label{Fig:second_model}
%\end{figure}
%%%%%%%%%%%%%%%%%%%%%%%%%%%%%%%%%%%%%%%%%%
As in the non-diluted model, traffic flows towards the left along directed edges that connect occupied nodes. Additionally, for $\ell > 0$, we assign to each occupied node $\x$ an amount of {\em traffic} given by $\mu_\ell(\x) = \mu(\x / \ell^2) \ge 0$, which accumulates along traffic paths. Here, $\mu \colon \R^2 \to [0, \infty)$ is a non-negative, real-valued, continuous function, satisfying $0 \le \mu(x, y) \le \mu_{\max} < \infty$ for all $(x, y) \in \R^2$, and is $\Xi$-Lipschitz continuous in its second argument.
Further, recall the {\em slid} $I_\ell(\x)$, as defined in~\eqref{Def:Slid}, the {\em total amount of traffic} $T_\ell(\x)$ flowing through $I_\ell(\x)$ based on $p_\ell$ and $\mu_\ell$, and let $\tau_\ell(\x)$ be the {\em number of points in the longest path} that carries traffic flowing through $I_\ell(\x)$. 

We define for $\x=(x_1,x_2)$,
\[
M_{t}(\x):= \int_{0}^{t} \sqrt{\beta_{p(x_1+s,x_2)}} \,\dd B_s\qquad \text{ and }\qquad
\vartheta(\x):=\inf\{t\geq 0\colon 1+M_{t}(\x)=0 \},
\]
where we recall that $\{B_t\colon t\ge 0\}$ represents a standard Brownian motion started at $0$, and 
\begin{align}\label{Def:Beta_p}
\beta_p:= \frac{p^4 +(2-p)^4}{p^2(2-p)^2}. 
\end{align}
We have the following first main result, where $c\x=(cx_1,cx_2)$.  
\begin{theorem}
\label{Thm:diluted:convergence-Tl-taul_gen}
For every $\x=(x_1,x_2)\in \ZZe$, 
\[
\Big( \frac{\tau_\ell(\ell^2\x)}{\ell^2}, \frac{T_\ell(\ell^2\x)}{\ell^3}\Big) \xlongrightarrow{ d } \Big(  \vartheta(\x), \frac{1}{2}\int_{0}^{\vartheta(\x)} \big(1+ M_{t}(\x)\big) p(x_1+t,x_2)\mu(x_1+t,x_2) \,\dd t \Big),\qquad \text{ as }\ell\to\infty. 
\]
\end{theorem}
In words, the accumulated traffic in an inhomogeneously diluted lattice with varying and homogenizing traffic density can be approximated by an accumulated area under a time-varying Brownian motion. 
The following result describes the precise asymptotics of $\tau_\ell(\x)$, which we only prove in the case where only horizontal inhomogeneities are allowed.  
\begin{theorem}
\label{Thm:diluted:asymp-taul}
Consider the vertically-homogeneous case where $p(x_1,x_2)=p(x_1)$ and $\mu(x_1,x_2)=\mu(x_1)$, for all $(x_1,x_2)\in \mathbb{S}$. Then, there exists $C<\infty$ such that, for any $\alpha > 2$ and all $\ell\ge 1$,
\[
\sup\big\{\P\big(\tau_\ell (\x) \geq \ell\big)\colon \x \in \ZZe\big\} <C \ell^{1/\alpha-1/2}.
\]
\end{theorem}
Before we continue, let us collect simplified expressions for simpler situations. 
\begin{remark}[Simplified expressions]
In the spatially homogeneous situation, where $\mu(\x)\equiv1$ and $p(\x)\equiv p$, note that
\[
M_{t}(\x)= \sqrt{\beta_{p}}B_t\qquad \text{ and }\qquad
\vartheta(\x)=\inf\{t\geq 0\colon 1+\sqrt{\beta_{p}}B_{ t}=0 \}
\]
the associated stopping time. Hence, with \(\varrho:=\inf\{t\geq 0\colon 1+B_{t}=0 \}\) and using that $\sqrt{\beta_{p}}B_{t}\xlongequal{ d } B_{\beta_{p}t}$ and thus $\beta_p\vartheta(\x)=\varrho$, we have the following statement. For any $\x\in\ZZe$ and $p\in [0,1]$,
\begin{align}\label{Thm:diluted:convergence-Tl-taul}
\Big( \frac{\tau_\ell(\ell^2\x)}{\ell^2}, \frac{T_\ell(\ell^2\x)}{\ell^3}\Big) \xlongrightarrow{ d } \Big( \frac{\varrho}{\beta_p}, \frac{p}{2\beta_p}\int_{0}^{\varrho} (1+B_{t}) \,\dd t \Big),\qquad \text{ as }\ell\to\infty.
\end{align}
Again, the density for the limiting distribution can be obtained from~\cite[Equation~(12)]{KeMa05}.
In the case mentioned in the introduction, where $p=1$, we have that $\beta_p=2$, which gives the formulation presented there. 
\end{remark}

%%%%%%%%%%%%%%%%%%%%%%%%%%%%%%%%%%%%%%
%%%%%%%%%%%%%%%%%%%%%%%%%%%%%%%%%%%%%%
%%%%%%%%%%%%%%%%%%%%%%%%%%%%%%%%%%%%%%

\subsection{Semi lattices}
\label{Subsec:Continuum_lattice} 
We now consider the case, where nodes are positioned in the space $\mathbb{S} = \mathbb{Z} \times \mathbb{R}$. For this, let $\lambda\colon \mathbb{R}^2 \to (0,\infty)$ be a real-valued, continuous function, which is $\Xi$-Lipschitz continuous, for some $\Xi>0$, in the second argument and $\lambda_{\min} \leq \lambda(\x) \leq \lambda_{\max}$, for all $\x\in \R^2$, for some $0<\lambda_{\min} \leq \lambda_{\max}<\infty$. For $\ell > 0$ and $x_1 \in \mathbb{Z}$, let $\mathcal{P}_{\ell,x_1}$ denote independent $1$-dimensional  Poisson point processes on $\mathbb{R}$ with inhomogeneous intensity measures given by $\lambda_\ell(x_1,x_2)\dd x_2$, where $\lambda_\ell(x_1,x_2)=\lambda(x_1/\ell^2,x_2/\ell^2)$. We define $\mathcal{P}_\ell := \bigcup_{x_1 \in \mathbb{Z}} \{\{x_1\}\times \mathcal{P}_{\ell,x_1}\}$, and consider the points in $\calP_\ell$ as occupied nodes. The role of the dilution is now played by the void spaces in the vertical inhomogeneous Poisson point processes. Again, we draw a directed edge from $\x=(x_1,x_2)\in \calP_\ell$ to $(x_1-1,y)\in \calP_\ell$ if and only if for all other $(x_1-1,z)\in \calP_\ell$, we have that $|x_2-y|<|x_2-z|$. Due to the assumed continuity of $\lambda$, with probability one, there is a unique such point and the choice variables become obsolete. Again, we may include individual amounts of traffic per vertex via a non-negative, real-valued, continuous function $\mu\colon \mathbb{R}^2 \to [0,\infty)$ with $0 \leq \mu(x, y) \leq \mu_{\max} < \infty$ and where $\mu$ is additionally  $\Xi$-Lipschitz continuous in its second argument. Then, the Poisson point $\x \in\calP_\ell$ generates $\mu_\ell(\x)=\mu(\x/\ell^2)$ amount of traffic. As usual, traffic then flows towards the left along directed edges between nodes in $\calP_\ell$ as illustrated in Figure~\ref{Fig:third_model}.
%%%%%%%%%%%%%%%%%%%%%%%%%%%%%%%%%%%%%%%%%%%
%Figure
%%%%%%%%%%%%%%%%%%%%%%%%%%%%%%%%%%%%%%%%%%
\begin{figure}[t]
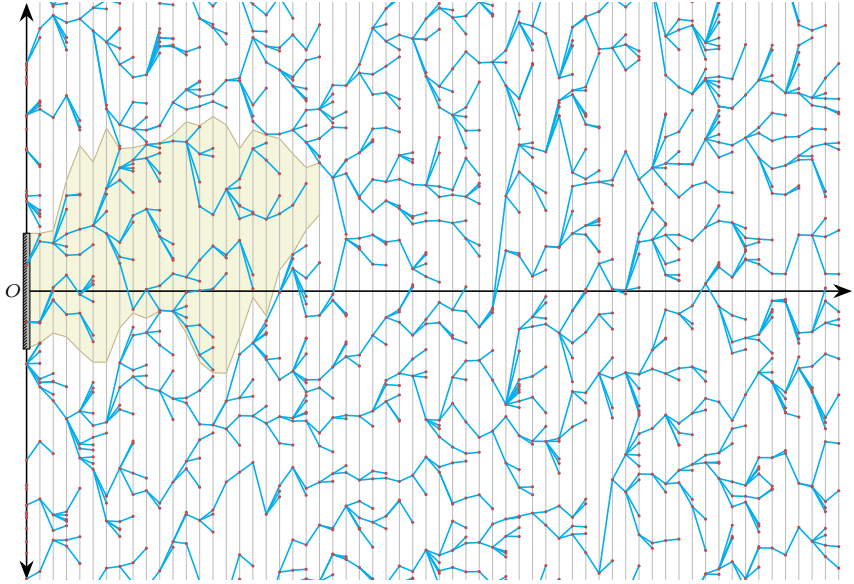

\centering

 \caption{Simulation of traffic flow in the semi-lattice model where the vertical lines are equipped with i.i.d.\ homogeneous Poisson point processes. Again, the beige region corresponds to the set of nodes that transmit their traffic through the slid.}
\label{Fig:third_model}
\end{figure}
%%%%%%%%%%%%%%%%%%%%%%%%%%%%%%%%%%%%%%%%%%%
%\begin{figure}[t]
%    \centering
%    \includegraphics{fig2}
%     \caption{Simulation of traffic flow in the semi-lattice model where the vertical lines are equipped with i.i.d.\ homogeneous Poisson point processes. Again, the beige region corresponds to the set of nodes that transmit their traffic through the slid.}
%\label{Fig:third_model}
%\end{figure}
%%%%%%%%%%%%%%%%%%%%%%%%%%%%%%%%%%%%%%%%%%
As before, we denote by $T_\ell(\x)$ the {\em total amount of traffic} flowing through the slid $I_\ell(\x)$ based on $\lambda_\ell$ and $\mu_\ell$, and let $\tau_\ell(\x)$ be the {\em number of points in the longest path} that carries traffic flowing through $I_\ell(\x)$. 

Define, for $\x=(x_1,x_2)$, the random variables
\[
M_{t}(\x):= \int_{0}^{t} \lambda(x_1+s,x_2)^{-1} \,\dd B_s\qquad \text{ and }\qquad\vartheta(\x):=\inf\{t\geq 0\colon 1+M_{t}(\x)=0\},
\]
corresponding to a Brownian-motion integral with respect to a time change and an associated stopping time. Then, we have the following result. 
\begin{theorem}
\label{Thm:Poisson:convergence-Tl-taul}
For any $\x=(x_1,x_2)\in\mathbb{S}$,
\[
\Big( \frac{\tau_\ell(\ell^2\x)}{\ell^2}, \frac{T_\ell(\ell^2\x)}{\ell^3}\Big) \xlongrightarrow{ d } \Big(  \vartheta(\x), \int_{0}^{\vartheta(\x)} \big(1+ M_{t}(\x)\big) \lambda(x_1+t,x_2)\mu(x_1+t,x_2) \,\dd t \Big), \qquad \text{ as }\ell\to\infty.
\]
\end{theorem}
For the asymptotics of $\tau_\ell(\x)$ we have the following statement, which we prove in the case where only horizontal inhomogeneities are allowed.  
\begin{theorem}
\label{Thm:Poisson:asymp-taul}
Consider the vertically-homogeneous case where $\lambda(x_1,x_2)=\lambda(x_1)$ and $\mu(x_1,x_2)=\mu(x_1)$, for all $(x_1,x_2)\in \mathbb{S}$. Then, there exists $C<\infty$ such that, for any $\alpha > 2$ and all $\ell\ge 1$,
\[
\sup\big\{\P\big(\tau_\ell (\x) \geq \ell\big)\colon \x \in \mathbb{S}\big\} <C \ell^{1/\alpha-1/2}.
\]
\end{theorem}
As in the diluted-lattice case, we next present some alternative and simplified expressions for the limiting random variables. 
\begin{remark}[Alternative and simplified expressions]
\label{Rem:simplified-expressions-diluted-lattice}
By the Dambis--Dubins--Schwarz Theorem (see~\cite[Chapter~V, Theorems~1.6 and~1.7]{RevYor99}), the process $\{M_t(\x)\colon t\geq0\}$ can be represented as a time-changed Brownian motion:  
$M_t(\x) = B_{\Lambda_{\x}(t)}$,
where  
\[
\Lambda_{\x}(t) := \int_0^t \lambda(x_1+s, x_2)^{-2} \, \dd s.
\]  
Using this, we can express the statement of  Theorem~\ref{Thm:Poisson:convergence-Tl-taul} alternatively as   
\[
\Big( \frac{\tau_\ell(\ell^2\x)}{\ell^2}, \frac{T_\ell(\ell^2\x)}{\ell^3}\Big) \xlongrightarrow{ d } \Big( \Lambda_{\x}^{-1} (\varrho), \int_{0}^{\varrho} (1+ B_{t}) \lambda(x_1+\Lambda_{\x}^{-1}(t),x_2)^3\mu(x_1+\Lambda_{\x}^{-1}(t),x_2) \,\dd t \Big),\qquad\text{ as }\ell\to\infty,
\]
where we recall that $\varrho=\inf\{t\geq 0\colon 1+B_{t}=0 \}$.
In particular, when $\mu(\x)=\lambda(\x)^{-3}$, Theorem~\ref{Thm:Poisson:convergence-Tl-taul} says that
\[
\Big( \frac{\tau_\ell(\ell^2\x)}{\ell^2}, \frac{T_\ell(\ell^2\x)}{\ell^3}\Big) \xlongrightarrow{ d } \Big( \Lambda_{\x}^{-1} (\varrho), \int_{0}^{\varrho}(1+ B_{ t}) \,\dd t \Big),\qquad\text{ as }\ell\to\infty.
\]
In the spatially homogeneous case, when $\lambda(\x)\equiv \lambda$ and $\mu(\x)\equiv 1$, Theorem~\ref{Thm:Poisson:convergence-Tl-taul} simplifies to
\begin{align}\label{Thm:SemiHom}
\Big( \frac{\tau_\ell(\ell^2\x)}{\ell^2}, \frac{T_\ell(\ell^2\x)}{\ell^3}\Big) \xlongrightarrow{ d } \Big( \lambda^2 \varrho, \lambda^3\int_{0}^{\varrho}(1+ B_{ t}) \,\dd t \Big),\qquad\text{ as }\ell\to\infty.
\end{align}
In the case where $\lambda(\x)=\frac{1}{x_1+x_2}$ and $\mu(\x)\equiv 1$, Theorem~\ref{Thm:Poisson:convergence-Tl-taul} implies that
\[
\Big( \frac{\tau_\ell(\ell^2\x)}{\ell^2}, \frac{T_\ell(\ell^2\x)}{\ell^3}\Big) \xlongrightarrow{ d }  \Big( (3\varrho)^{1/3}, \int_{0}^{\varrho}\frac{1+ B_{ t}}{3t} \,\dd t \Big),\qquad \text{ as }\ell\to\infty,
\]
and, when $\lambda(\x)=\exp(-x_1-x_2)$ and $\mu(\x)\equiv 1$, Theorem~\ref{Thm:Poisson:convergence-Tl-taul} implies 
\[
\Big( \frac{\tau_\ell(\ell^2\x)}{\ell^2}, \frac{T_\ell(\ell^2\x)}{\ell^3}\Big) \xlongrightarrow{ d }  \Big( \frac{1}{2}\log(1+2\varrho), \int_{0}^{\varrho}\frac{1+ B_{ t}}{(1+2t)^{3/2}} \,\dd t \Big), \qquad \text{ as }\ell\to\infty.
\]
\end{remark}

In the next section, we present the main steps of the proof in case of the semi lattice without vertical inhomogeneities. Transfer results towards the general case and in particular the lattice $\ZZe$ are presented in Section~\ref{Sec:Proofs}. 

%%%%%%%%%%%%%%%%%%%%%%%%%%%%%%%%%%%%%%
%%%%%%%%%%%%%%%%%%%%%%%%%%%%%%%%%%%%%%
%%%%%%%%%%%%%%%%%%%%%%%%%%%%%%%%%%%%%%

\section{Proof strategy}\label{Sec:Proof_Strategy}
In all what follows, we only consider the case where $\x=\o$ is the origin. The case of general $\x$ follows without effort. The key idea is to track some upper and lower-most paths in the traffic network, starting in the growing slid, backwards in time, and indicated in brown in Figure~\ref{Fig:third_model}. Up to a first stopping time, they behave like independent random walks and the occupied area between them converges under our scaling to the area under a Brownian motion. The analysis is complicated by the fact that this first stopping time has to be identified, in the limit, with the stopping time where the paths actually meet for the first time. 

As mentioned, we focus now on the semi-lattice case with only horizontal inhomogeneities, i.e., where $\lambda(x_1,x_2)=\lambda(x_1)$ and $\mu(x_1,x_2)=\mu(x_1)$. 
We start by defining the bounding random walks. 

%%%%%%%%%%%%%%%%%%%%%%%%%%%%%%%%%%%%%%
%%%%%%%%%%%%%%%%%%%%%%%%%%%%%%%%%%%%%%
%%%%%%%%%%%%%%%%%%%%%%%%%%%%%%%%%%%%%%

\subsection{Bounding random walks on $\mathbb{S}$ without vertical inhomogeneities}
Let $\x=\o$. Then, at initial time we set $A_0^\ell := \ell/2$ and $B_0^\ell := -\ell/2$ and define recursively, for each $i \geq 0$,
\begin{align*}
    \overline{U}^\ell_i &:= \inf\{y > A^\ell_i \colon y \in {\calP}_{\ell,i} \}, &
    \overline{L}^\ell_i &:= \inf\{y \geq B^\ell_i \colon y \in {\calP}_{\ell,i} \}, \\
    \underline{U}^\ell_i &:= \sup\{y \leq A^\ell_i \colon y \in {\calP}_{\ell,i} \}, &
    \underline{L}^\ell_i &:= \sup\{y < B^\ell_i \colon y \in {\calP}_{\ell,i} \}, \\
    A^\ell_{i+1} &:= ( \overline{U}^\ell_i +\underline{U}^\ell_i )/2, &
    B^\ell_{i+1} &:= ( \overline{L}^\ell_i + \underline{L}^\ell_i )/2.
\end{align*}
We also define
\[
    \overline{X}^\ell_i := \overline{U}^\ell_i- A^\ell_i, \quad \underline{X}^\ell_i :=  A^\ell_i - \underline{U}^\ell_i, \qquad \text{ and }\qquad \overline{Y}^\ell_i := \overline{L}^\ell_i-  B^\ell_i , \quad \underline{Y}^\ell_i := B^\ell_i - \underline{L}^\ell_i
\]
and note that, with these definitions, we have
\[
    A^\ell_{i+1}-A^\ell_i= (\overline{X}^\ell_i-\underline{X}^\ell_i)/2
    \qquad\text{ and }\qquad
    B^\ell_{i+1}-B^\ell_i= (\overline{Y}^\ell_i-\underline{Y}^\ell_i)/2.
\]
With these definitions, the sequences $\{A^\ell_{i+1} - A^\ell_i\colon i \geq 0\}$ and $\{B^\ell_{i+1} - B^\ell_i\colon i \geq 0\}$ consist of independent random variables. Moreover, note that, for each $i \geq 0$, the variables $\overline{X}^\ell_i$, $\underline{X}^\ell_i$, $\overline{Y}^\ell_i$, and $\underline{Y}^\ell_i$ are  $\Exp(\lambda(i/\ell^2 ))$ random variables.

Next, denote by $\calF_0$ the trivial $\sigma$-algebra and, for each $i \geq 1$, let $\calF_i$ be the $\sigma$-algebra generated by $\{\calP_{\ell, k}\colon 0 \leq k \leq i-1\}$. Then, 
\[
   \tau_{\ell} = \inf \big\{ i \geq 0 \colon \calP_{\ell, i}([B^\ell_i, A^\ell_i]) = 0 \big\}\qquad\text{and}\qquad
   \tau'_{\ell}:= \inf \big\{ i \geq 0 \colon \calP_{\ell, i}([B^\ell_i, A^\ell_i]) \leq 1 \big\},
\]
are stopping times with respect to this filtration. Here, we wrote $\calP(A):=\#( A \cap \calP)$ for the number of points of $\calP$ in a set $A$. Recall that $\tau_\ell$ is the length of the longest path, one of our two key quantities of interest, and that traffic generated at any point below the levels $\{A^\ell_i\}$ will pass through the segment $\{0\} \times (-\infty, \ell]$, while traffic generated at any point above the levels $\{B^\ell_i\}$ will pass through $\{0\} \times [-\ell, \infty)$. Consequently, only traffic generated at Poisson points in
\[
\bigcup_{i \geq 0} \big( \{i\} \times [B^\ell_i, A^\ell_i] \big)
\]
will pass through the slit $I_\ell=I_\ell(\o)$. Furthermore, note that for any $i > \tau_{\ell}$, we have $A^\ell_i = B^\ell_i$. 
Therefore,
\[
T_{\ell}= \sum_{i=0}^{\tau_{\ell} -1} \calP_{\ell, i}([B_i^{\ell}, A_i^{\ell}])\mu(i/\ell^2)\qquad\text{and}\qquad T'_{\ell}:= \sum_{i=0}^{\tau'_{\ell} -1} \calP_{\ell, i}([B_i^{\ell}, A_i^{\ell}])\mu(i/\ell^2)
\]
represent the total traffic accumulated up to the times $\tau_{\ell}$ and $\tau'_{\ell}$, respectively.

Note that, the bounding random walks behave independently only up to time $\tau'_{\ell}$. Consequently, we split the analysis into a part prior to $\tau'_{\ell}$ and a part that deals with what happens between $\tau'_{\ell}$ and $\tau_{\ell}$. 

%%%%%%%%%%%%%%%%%%%%%%%%%%%%%%%%%%%%%%
%%%%%%%%%%%%%%%%%%%%%%%%%%%%%%%%%%%%%%
%%%%%%%%%%%%%%%%%%%%%%%%%%%%%%%%%%%%%%

\subsection{Up to time $\tau'_{\ell}$: bulk contribution}\label{Sec:prior_time_1}
Recall the global bounds $\lambda_{\max}, \mu_{\max}<\infty$. The following result removes the randomness associated to Poisson points in the area between the bounding random walks. 
\begin{prop}
For any $\epsilon, \ell , t>0$, we have that
\[
 \ell^3 \P\Big(   \sup_{0\leq k \leq \lfloor \ell^2 t \rfloor } \frac{1}{\ell^3} 
    \Big | \sum_{i=0}^{k} \calP_{\ell, i}([B_i^{\ell}, A_i^{\ell}])\mu(i/\ell^2)
    - \sum_{i = 0}^{k} \big(A^{\ell}_i- B^{\ell}_i\big) \lambda(i/\ell^2) \mu(i/\ell^2)\Big|>\epsilon\Big)\leq t\epsilon^{-2}\lambda_{\max}\mu_{\max}^2.
\]
\label{Prop:Error_in_area_approximation-homogeneous}
\end{prop}
We present the proof in Section~\ref{Sec:proof_prior_time_1}. 
Now, to analyze the sum
\[
\sum_{i = 0}^{k} (A^{\ell}_i- B^{\ell}_i) \lambda(i/\ell^2) \mu(i/\ell^2),
\]
we couple the processes $\{A^{\ell}_i\colon i \geq 0\}$ and $\{B^{\ell}_i\colon i \geq 0\}$ with two independent processes $\{\hat{A}^{\ell}_i\colon i \geq 0\}$ and $\{\hat{B}^{\ell}_i\colon i \geq 0\}$ as follows. Let $\{\hat{\overline{X}}_i\colon i \geq 0\}$, $\{\hat{\underline{X}}^\ell_i\colon i \geq 0\}$, $\{\hat{\overline{Y}}^\ell_i\colon i \geq 0\}$, and $\{\hat{\underline{Y}}^\ell_i\colon i \geq 0\}$ be four independent sequences of independent random variables. For all $i \geq 0$, the random variables $\hat{\overline{X}}^\ell_i$, $\hat{\underline{X}}^\ell_i$, $\hat{\overline{Y}}^\ell_i$, and $\hat{\underline{Y}}^\ell_i$ are i.i.d.\ $\Exp(\lambda(i/\ell^2, 0))$.
Setting $\hat A^{\ell}_0 = \ell/2$ and $\hat B^{\ell}_0 = -\ell/2$, we define the processes recursively by
\[
    \hat A^\ell_{i+1} := \hat A^\ell_i + ( \hat{\overline{X}}^\ell_i-\hat{\underline{X}}^\ell_i)/2
    \qquad\text{ and }\qquad
    \hat B^\ell_{i+1} := \hat B^\ell_i + (\hat{\overline{Y}}^\ell_i-\hat{\underline{Y}}^\ell_i)/2.
\]
We further define 
\[
    \hat \tau_\ell := \inf \Big\{ i \geq 0 \colon \hat A^\ell_i - \hat{\underline{X}}^\ell_i \leq \hat B^\ell_i + \hat{\overline{Y}}^\ell_i \Big\},
\]
and note that
\begin{align}
    &\Big(\hat \tau_\ell , \Big\{\Big(\hat A^\ell_{i+1}, \hat B^\ell_{i+1}, \hat{\overline{X}}^\ell_i, \hat{\underline{X}}^\ell_i, \hat{\overline{Y}}^\ell_i, \hat{\underline{Y}}^\ell_i\Big)\ind{i< \hat\tau_\ell }  \Big\}_{i\geq 0} \Big)
    \xlongequal{ d } 
    \Big(\tau'_{\ell}, \Big\{\Big( A^\ell_{i+1},  B^\ell_{i+1}, {\overline{X}}^\ell_i, {\underline{X}}^\ell_i, {\overline{Y}}^\ell_i, {\underline{Y}}^\ell_i\Big)\ind{i< \tau'_{\ell} }  \Big\}_{i\geq 0}  \Big).
    \label{Equ:Coupling-homogeneous-to-independent}
\end{align}
This gives us that
\begin{align}
    &\Big(\tau'_{\ell}, \sum_{i = 0}^{\tau'_{\ell}-1} (A^{\ell}_i- B^{\ell}_i) \lambda(i/\ell^2) \mu(i/\ell^2) \Big)\xlongequal{ d }
   \Big(\hat \tau_\ell, \sum_{i = 0}^{\hat \tau_\ell-1} (\hat A^{\ell}_i- \hat B^{\ell}_i) \lambda(i/\ell^2) \mu(i/\ell^2) \Big) 
   = \Big(\hat \tau_\ell, \sum_{i = 0}^{\hat \tau_\ell-1}  D^{\ell}_i \lambda(i/\ell^2) \mu(i/\ell^2) \Big),
\end{align}
where $D^\ell_i:= \hat A^\ell_i-\hat B^\ell_i$ for all $i\geq0$. Observe that $\{D^\ell_{i+1}-D^\ell_i\colon i\geq 0\}$ is a sequence of independent random variables with  $\E[D^\ell_{i+1}-D^\ell_{i}]=0$ and 
\[
    \Var[D^\ell_{i+1}-D^\ell_{i}]=2\Var[\hat A^\ell_{i+1}-\hat A^\ell_{i}] = 2^{-1}\Var[\hat{\overline{X}}^\ell_i-\hat{\underline{X}}^\ell_i]  =\lambda(i/\ell^2)^{-2}.
\]
Since $D^\ell_{i+1}-D^\ell_i=(\hat{\overline{X}}^\ell_i+\hat{\underline{Y}}^\ell_i-\hat{\underline{X}}^\ell_i-\hat{\overline{Y}}^\ell_i)/2 $, its density is given by 
\[
f_{D^\ell_{i+1}-D^\ell_{i}}(x) = \lambda(i/\ell^2)\big(|x|\lambda(i/\ell^2)+1/2\big) \e^{-2|x|\lambda(i/\ell^2) }, \qquad \text{ for $x\in\R$.}
\]
For all $i\geq 1$, we define
\begin{align}
   V_{i}:= \lambda((i-1)/\ell^2)(D^\ell_{i}-D^\ell_{i-1})
   \label{Equ:Def-V}
\end{align}
and note that $\E[V_i]=0$, $\Var[V_i]=1$ and that its density is given by
\[
f_{V_{i}}(x)= (|x|+1/2)\e^{-2|x|},\qquad\text{ for $x\in\R$.}
\]
We now define a sequences of processes $\{W_\ell\colon \ell\geq1\}$ as $W_\ell=\{W_{\ell,t}\colon t\geq0\}$, where
\begin{align}
     W_{\ell,t}:=\ell^{-1}\big( D^\ell_{\lfloor \ell^2t \rfloor} +( \ell^2t- \lfloor \ell^2t \rfloor) ( D^\ell_{\lfloor \ell^2t \rfloor+1} - D^\ell_{\lfloor \ell^2t \rfloor}) \big).
     \label{Equ:Def-Wl}
\end{align}
Note that the $W_\ell$ take values in the measurable space $(\calC[0,\infty), \mathscr{B}( \calC[0,\infty) ))$, where $\calC[0,\infty)$ is the space of all continuous real-valued functions on $[0,\infty)$. We equip $\calC[0,\infty)$ with the metric 
\[
\rho(f,g):= \sum_{k\ge 1} 2^{-k} \max_{0\leq t \leq k} (|f(t)-g(t)|\wedge 1)
\]
and the Borel $\sigma$-algebra $\mathscr{B}( \calC[0,\infty) )$ on $\calC[0,\infty)$  is generated by the open sets defined with respect to the metric $\rho$. The next result establishes a function central-limit theorem for the process of differences. 
\begin{lemma} 
\label{Lem:Wlt-weak-convergence}
As $\ell \to \infty$, the process $W_\ell=\{W_{\ell,t} \colon t \geq 0\}$ converges weakly to 
    \[
    1+M_{\cdot}=\{1+M_t \colon t\geq 0\}:=
    \Big\{1+\int_{0}^t \lambda(s)^{-1}\,\dd B_s\colon t \geq 0\Big\}.
    \]
\end{lemma}
We present the proof in Section~\ref{Sec:proof_prior_time_1}. 
Next, for a function $f\in \calC[0,\infty)$ and $a\in\R$, we define
  \[
  \calH_a(f):= \inf\{ t\geq 0\colon f(t)=a\}
  \qquad\text{ and }\qquad
  \calI_a(f):= \int_{0}^{\calH_a(f)} f(t)\lambda(t)\mu(t)\,\dd t.
  \]
\begin{remark}
\label{Rem:Continuity-of-stopping-time-vert-homo-semi}
 Note that, for any $a \in \mathbb{R}$, the functional \(\calH_a\) is continuous on a set of Wiener measure one (see~\cite[Appendix~M15]{Billingsley99}). The same argument shows that $\mathcal{H}_a$ is almost-surely continuous at $1 + M_{\cdot} = 1 + B_{\Lambda(\cdot)}$, because \(\Lambda\) is a strictly increasing continuous function \([0, \infty) \to [0, \infty)\). Here, as mentioned in Remark~\ref{Rem:simplified-expressions-diluted-lattice}, we use that, by the Dambis--Dubins--Schwarz theorem, $M_{\cdot}$ can be represented as a time-changed Brownian motion $M_{\cdot} = B_{\Lambda(\cdot)}$, where
\[
\Lambda(t) := \int_0^t \lambda(s)^{-2}\, \dd s.
\]  
Moreover, since the map $(\alpha, g) \mapsto (\alpha, \int_{0}^{\alpha} g(t)\, \dd t)$ is continuous from $\R \times \mathcal{C}[0, \infty)$ to $\R^2$,  it follows that the jointly defined functional $(\mathcal{H}_a, \mathcal{I}_a)$ is almost-surely continuous at $1 + M_{\cdot}$.
\end{remark}
  
Lemma~\ref{Lem:Wlt-weak-convergence} and Remark~\ref{Rem:Continuity-of-stopping-time-vert-homo-semi} together allow us to establish the following result, with the full proof given in Section~\ref{Sec:proof_prior_time_1}. 
\begin{lemma}
\label{Lem:Wlt-integration-weak-convergence}
   We have that,
    \[
    \Big(\hat{\tau}_\ell/\ell^2, 
    \int_{0}^{\hat{\tau}_\ell/\ell^2} W_{\ell,t}\lambda(t)\mu(t) \,\dd t \Big) \xlongrightarrow{ d }
\Big(\vartheta,\int_{0}^{\vartheta} \big(1+ M_{t}\big) \lambda(t)\mu(t) \,\dd t \Big),\qquad \text{as }\ell\to\infty.
    \]
\end{lemma}
Now, we have all the tools to present the following key result for the bulk term. 
\begin{prop}
\label{Prop:Conv-in-dist-upto-tau1-vert-homo}
   We have that,
\[
\Big( \tau'_{\ell}/\ell^2, T'_{\ell}/\ell^3 \Big) \xlongrightarrow{ d }
\Big(\vartheta,\int_{0}^{\vartheta} \big(1+ M_{t}\big) \lambda(t)\mu(t) \,\dd t \Big),\qquad \text{as }\ell\to\infty.
\]  
\end{prop}
The proof is again postponed to Section~\ref{Sec:proof_prior_time_1}. We next deal with the remaining contribution between the times $\tau'_{\ell}$ and $\tau_{\ell}$, which turns out to be asymptotically negligible. 

%%%%%%%%%%%%%%%%%%%%%%%%%%%%%%%%%%%%%%
%%%%%%%%%%%%%%%%%%%%%%%%%%%%%%%%%%%%%%
%%%%%%%%%%%%%%%%%%%%%%%%%%%%%%%%%%%%%%

\subsection{Between $\tau'_{\ell}$ and $\tau_{\ell}$: boundary contribution}\label{Sec:prior_time_0}
In order to represent the distance of the bounding random walks at time $\tau'_{\ell}$, we replace $\ell$ in the notation with $(\eta, j, \ell)$ when $A_0=\eta/2$ and $B_0=-\eta/2$, and the intensity of the Poisson point process at $(i, s)$ is given by $\lambda((i + j)/\ell^2)$. In this case, we also write $\calP_{\ell, i}$ as $\calP_{\ell, i}^{(j)}$ and $\calF_i$ as $\calF_i^{(j)}$ and note that, by the vertical homogeneity, only the initial distance $A_0-B_0=\eta$ is relevant and not the specific starting positions. With this notation, we have the  following tail bounds on the distribution of $\tau'_{\ell}$, with the proof given in Section~\ref{Sec:proof_prior_time_0}.
\begin{prop}\label{Prop:tau_1_vert_inhom_decay}
   There exist a constant $\bar{C} \in (0,\infty)$ such that, for all $\eta,t,\ell\geq 1$ and $j \geq 0$,
\[
        \P\big(\tau'_{(\eta, j, \ell)} \geq t \big)
        \leq \bar{C}\eta t^{-1/2}.
\]
    In particular,  for all $t,\ell \geq 1$,
\[
       \P\big(\tau'_{\ell} \geq t\big) \leq \bar{C}\ell t^{-1/2}.
\]
\end{prop}
This result, together with the following bound on the distance between the two stopping times, provides the necessary tools to conclude the negligibility. 
\begin{prop}
\label{Prop:tau0-tau1-decay-rate-vertical-homo-semi}
    There exists a constant $\bar{\mathsf{C}} \in (0, \infty)$ such that, for any $\alpha > 0$ and all  $\ell\geq 1$,
\[
\P(\tau_{\ell} - \tau'_{\ell} > \ell^{\alpha}) \leq \bar{\mathsf{C}} \, \ell^{-\alpha/2} (\log \ell)^8.
\]
\end{prop}
Again, the proof is given in Section~\ref{Sec:proof_prior_time_0}. 

%%%%%%%%%%%%%%%%%%%%%%%%%%%%%%%%%%%%%%
%%%%%%%%%%%%%%%%%%%%%%%%%%%%%%%%%%%%%%
%%%%%%%%%%%%%%%%%%%%%%%%%%%%%%%%%%%%%%

\subsection{Main result for vertically homogeneous semi lattices}
Now we are in the position to establish our first set of main results Theorem~\ref{Thm:Poisson:asymp-taul} and Theorem~\ref{Thm:Poisson:convergence-Tl-taul}, at least in the vertically homogeneous case, also using one technical input from the proof section below. 
\begin{proof}[Proof of Theorem~\ref{Thm:Poisson:convergence-Tl-taul}: vertically homogeneous case]
    In view of Proposition~\ref{Prop:Conv-in-dist-upto-tau1-vert-homo} and using Slutsky's theorem, it is enough to show that
    \begin{align}
     &\big(\tau_{\ell}-\tau'_{\ell} \big) /\ell^2  \xlongrightarrow{ p } 0,\qquad \text{as }\ell\to\infty
     \label{Equ:conv-in-prob-tau0-tau1}
    \end{align}
    and
    \begin{align}
     &\big(T_{\ell}-T'_{\ell} \big)/\ell^3   \xlongrightarrow{ p } 0,\qquad \text{as }\ell\to\infty.
     \label{Equ:conv-in-prob-T0-T1}
    \end{align}
On the one hand,~\eqref{Equ:conv-in-prob-tau0-tau1} follows immediately from Proposition~\ref{Prop:tau0-tau1-decay-rate-vertical-homo-semi}. 
On the other hand, as a consequence of~\eqref{Equ:tail-of-tau0-eta-j-l-3} and Propositions~\ref{Prop:Error_in_area_approximation-homogeneous},~\ref{Prop:Conv-in-dist-upto-tau1-vert-homo} and~\ref{Prop:tau0-tau1-decay-rate-vertical-homo-semi}, the probability of the event
\[
\Big\{ \tau'_{\ell} \leq \ell^3,\, \tau_{\ell} - \tau'_{\ell} \leq \ell,\, \mathscr{E}_{3}^{\ell}, \sup_{0\leq k \leq \lfloor 2\ell^3 \rfloor } \frac{1}{\ell^3} 
    \Big| \sum_{i=0}^{k} \calP_{\ell, i}([B_i^{\ell}, A_i^{\ell}])\mu(i/\ell^2)
    - \sum_{i = 0}^{k} (A^{\ell}_i- B^{\ell}_i) \lambda(i/\ell^2) \mu(i/\ell^2)  \Big| \leq \frac{1}{\sqrt{\ell}} \Big\}  
\]
tends to one as $\ell \to \infty$, and on this event,
\[
T_{\ell} - T'_{\ell} \leq 2\ell^2\sqrt{\ell} +\sum_{i=0}^{\ell-1} (2i+1)(\log \ell)^2 = 2\ell^2\sqrt{\ell} + \ell^2 (\log \ell)^2,
\]
which implies~\eqref{Equ:conv-in-prob-T0-T1}.   
\end{proof}
\begin{proof}[Proof of Theorem~\ref{Thm:Poisson:asymp-taul}]
The statement follows directly from Propositions~\ref{Prop:tau_1_vert_inhom_decay} and~\ref{Prop:tau0-tau1-decay-rate-vertical-homo-semi}.
\end{proof}

As mentioned, the proofs of the statements in Sections~\ref{Sec:prior_time_1} and~\ref{Sec:prior_time_0} can be found below in Sections~\ref{Sec:proof_prior_time_1} and~\ref{Sec:proof_prior_time_0}. Our results for semi lattices in the vertically inhomogenous case are quite a bit harder, since we must now track also the vertical positions of the walkers, leading to significantly less independence in the process. However, the general strategy is the same as in the homogeneous case. This case is dealt with in Section~\ref{Sec:Inhom}. Finally, in Section~\ref{Sec:Proofs_DL} we prove our results for the diluted lattice and in Section~\ref{Sec:Proofs_PL} we finish by considering the pure lattice. 

%\PPG{It is also interesting that we found out 
%\[
%\int_{0}^{\varrho^{(\gamma)}} B^{(\gamma)}_{t} \,\dd %t\,\,\xlongequal{\,d\,}\,\,\gamma^3\int_{0}^{\varrho^%{(1)}} B^{(1)}_{t} \,\dd t,
%\] 
%where $\{B^{(\gamma)}_{ t}\}_{t\geq 0}$ is a standard Brownian motion starting at 
%$\gamma$, and 
%\[
%\varrho^{(\gamma)}=
%\{t\ge
%0;
%B^{(\gamma)}_{t}=0 \}.
%\]}

%%%%%%%%%%%%%%%%%%%%%%%%%%%%%%%%%%%%%%
%%%%%%%%%%%%%%%%%%%%%%%%%%%%%%%%%%%%%%
%%%%%%%%%%%%%%%%%%%%%%%%%%%%%%%%%%%%%%

\section{Proofs}\label{Sec:Proofs}
Recall that we only consider the case where $\x=\o$ is the origin. The general case follows from minor adaptations of the proofs. 

%%%%%%%%%%%%%%%%%%%%%%%%%%%%%%%%%%%%%%
%%%%%%%%%%%%%%%%%%%%%%%%%%%%%%%%%%%%%%
%%%%%%%%%%%%%%%%%%%%%%%%%%%%%%%%%%%%%%

\subsection{Proofs for vertically homogeneous semi lattices}\label{Sec:Proofs_SL_Hom}

%%%%%%%%%%%%%%%%%%%%%%%%%%%%%%%%%%%%%%
%%%%%%%%%%%%%%%%%%%%%%%%%%%%%%%%%%%%%%
%%%%%%%%%%%%%%%%%%%%%%%%%%%%%%%%%%%%%%

\subsubsection{Proofs for Section~\ref{Sec:prior_time_1}}\label{Sec:proof_prior_time_1}
\begin{proof}[Proof of Proposition~\ref{Prop:Error_in_area_approximation-homogeneous}]
Define
\[
\calZ^{\ell}_k:=  \sum_{i = 0}^{k-1}  \calP_{\ell, i}([B_i^{\ell}, A_i^{\ell}])\mu(i/\ell^2)
- \sum_{i = 0}^{k-1} \big(A^{\ell}_i- B^{\ell}_i\big) \lambda(i/\ell^2) \mu(i/\ell^2)
\]
and observe that
\begin{align*}
\E\big[ \calZ^{\ell}_{k+1}- \calZ^{\ell}_k | \calF_k \big] 
= \mu(k/\ell^2) \E\big[\calP_{\ell, k}([B^{\ell}_k, A^{\ell}_k])
-   (A^{\ell}_k- B^{\ell}_k) \lambda(k/\ell^2) | \calF_k \big] = 0.
\end{align*}
This implies that $\{\calZ^{\ell}_k\colon k \geq 0\}$ is a martingale with
\begin{align*}
\E\big[ \big(\calZ^{\ell}_{k+1}\big)^2- \big(\calZ^{\ell}_k \big)^2 | \calF_k \big]
&=\E\big[ 2\calZ^{\ell}_k\big(\calZ^{\ell}_{k+1}-\calZ^{\ell}_k \big)+ \big(\calZ^{\ell}_{k+1}-\calZ^{\ell}_k \big)^2 | \calF_k \big]\\
&=\E\big[ \big(\calZ^{\ell}_{k+1}-\calZ^{\ell}_k \big)^2 | \calF_k \big]\\
&=\mu(k/\ell^2)^2\E\big[ \big({\calP}_k{[B^{\ell}_k, A^{\ell}_k]}
-   (A^{\ell}_k- B^{\ell}_k) \lambda(k/\ell^2) \big)^2 |\calF_k \big].
\end{align*}
Since the conditional distribution of the number of points $\calP_{\ell, k}([B^{\ell}_k, A^{\ell}_k])$, given $\calF_k$, is Poisson with parameter $(A^{\ell}_k- B^{\ell}_k) \lambda(k/\ell^2)$, we obtain that
\begin{align*}
\E\big[ \big(\calZ^{\ell}_{k+1}\big)^2- \big(\calZ^{\ell}_k \big)^2 |\calF_k \big]
&=\mu(k/\ell^2)^2 (A^{\ell}_k- B^{\ell}_k) \lambda(k/\ell^2),
\end{align*}
which implies that
\begin{align*}
\E\big[ \big(\calZ^{\ell}_{k+1}\big)^2- \big(\calZ^{\ell}_k \big)^2  \big]=\mu(k/\ell^2)^2\lambda(k/\ell^2) \E[A^{\ell}_k- B^{\ell}_k]=\ell \mu(k/\ell^2)^2\lambda(k/\ell^2).
\end{align*}
Therefore, using Doob's maximal inequality, we obtain that
\begin{align*}
\P\Big( \sup_{0\leq k \leq \lfloor \ell^2 t \rfloor } \big| \calZ^{\ell}_{k+1} \big| >\epsilon \ell^3 \Big)
&\leq \lim_{\ell\uparrow\infty}\frac{1}{\epsilon^2\ell^6} \E\big[ \big(\calZ^{\ell}_{\lfloor \ell^2 t \rfloor +1}\big)^2 \big]= \lim_{\ell\uparrow\infty}\frac{1}{\epsilon^2\ell^5}\sum_{i = 0}^{\lfloor \ell^2 t \rfloor}\mu(i/\ell^2)^2\lambda(i/\ell^2)
\leq \epsilon^{-2}\ell^{-3}t\lambda_{\max}\mu_{\max}^2,
\end{align*}
as desired. 
\end{proof}

\begin{proof}[Proof of Lemma~\ref{Lem:Wlt-weak-convergence}]
We begin by noting that it suffices to prove that, for any $K  > 0$,
\[
\{W_{\ell,t}\colon 0\leq t\leq K \} \xlongrightarrow{ d } \{1+M_t \colon 0\leq t\leq K\},\qquad\text{ as }\ell \to \infty.
\]
    Let $\{{1/\lambda^{(m)}}\colon m\geq 1\}$ be a sequence of  piece-wise constant approximation of the function $1/\lambda$ in the $\rho$-metric.
    Observe that
\[
        W_{\ell,t}= 1+ \frac{1}{\ell} \Big( \sum_{j=1}^{\lfloor \ell^2t \rfloor} \frac{1}{\lambda((j-1)/\ell^2)} V_j + ( \ell^2t- \lfloor \ell^2t \rfloor) \frac{1}{\lambda(\lfloor \ell^2t \rfloor/\ell^2)} V_{\lfloor \ell^2t \rfloor+1}  \Big).
\]
    Similarly we define 
    \begin{align*}
         W^{(m)}_{\ell,t}:= 1+ \frac{1}{\ell} \Big( \sum_{j=1}^{\lfloor \ell^2t \rfloor} \frac{1}{\lambda^{(m)}((j-1)/\ell^2)} V_j + ( \ell^2t- \lfloor \ell^2t \rfloor) \frac{1}{\lambda^{(m)}(\lfloor \ell^2t \rfloor/\ell^2)} V_{\lfloor \ell^2t \rfloor+1}  \Big).
    \end{align*}
    Now, by Donsker's invariance principle, we obtain that for every $m\geq 1$,
    \begin{align}
    \{W^{(m)}_{\ell,t}\colon 0\leq t\leq K \} \xlongrightarrow{ d }  
    \Big\{1+\int_{0}^t \frac{1}{\lambda^{(m)}(s)}\,\dd B_s\colon 0\leq t\leq K \Big\} \qquad \text{ as $\ell\to\infty$.}
    \label{Equ:Wmlt-as-l-to-infty}
    \end{align}
    Furthermore, by Itô isometry,  we have 
     \begin{align}
    \Big\{1+\int_{0}^t \frac{1}{\lambda^{(m)}(s)}\,\dd B_s\colon 0\leq t\leq K \Big\}
    \xlongrightarrow{ d } 
    \Big\{1+\int_{0}^t \frac{1}{\lambda(s)}\,\dd B_s\colon 0\leq t\leq K \Big\} \qquad
    \text{ as $m\to\infty$.}
    \label{Equ:Wmlt-as-l-to-infty-then-m-to-infty}
    \end{align}
    Considering~\eqref{Equ:Wmlt-as-l-to-infty} and~\eqref{Equ:Wmlt-as-l-to-infty-then-m-to-infty}, along with~\cite[Theorem~3.2]{Billingsley99},  proving Lemma~\ref{Lem:Wlt-weak-convergence} reduces to showing that
    \[
    \lim_{m\uparrow\infty} \limsup_{\ell\uparrow\infty} \P\Big( \sup_{t\in[0,K ]} \big | W_{\ell,t}-  W^{(m)}_{\ell,t} \big| >\epsilon\Big)=0.
    \]
For this, since both $W_{\ell,t}$ and $W^{(m)}_{\ell,t}$ are linear interpolations, we have 
    \[
     \sup_{t\in[0,K ]} \big| W_{\ell,t}-  W^{(m)}_{\ell,t} \big| = \sup_{1\leq i\leq \lceil \ell^2K  \rceil} \Big| \frac{1}{\ell} \sum_{j=1}^{i}  \Big(\frac{1}{\lambda((j-1)/\ell^2)}-  \frac{1}{\lambda^{(m)}((j-1)/\ell^2)}\Big) V_j \Big|.
    \]
    Therefore, by Kolmogorov's maximal inequality, 
    \[
       \P\Big( \sup_{t\in[0,K ]} \big| W_{\ell,t}-  W^{(m)}_{\ell,t} \big| >\epsilon\Big) \leq \frac{1}{\epsilon^2\ell^2}\sum_{j=1}^{ \lceil \ell^2K  \rceil}  \Big(\frac{1}{\lambda((j-1)/\ell^2)}-  \frac{1}{\lambda^{(m)}((j-1)/\ell^2)}\Big)^2
    \]
    and hence,
    \[
    \lim_{m\uparrow\infty} \limsup_{\ell\uparrow\infty} \P\Big( \sup_{t\in[0,K ]} \big| W_{\ell,t}-  W^{(m)}_{\ell,t} \big| >\epsilon\Big)
    \leq \frac{1}{\epsilon^2} \lim_{m\uparrow\infty} \int_{0}^K  \Big( \frac{1}{\lambda(s)}-\frac{1}{\lambda^{(m)}(s)} \Big)^2 \,\dd s =0
    \]
and the proof is completed. 
\end{proof}

\begin{proof}[Proof of Lemma~\ref{Lem:Wlt-integration-weak-convergence}]
    Recall that
\[
 \hat \tau_\ell= \inf\Big\{i\geq 0 \colon D^{\ell}_i < \hat{\underline{X}}^\ell_i + \hat{\overline{Y}}^\ell_i \Big\}.
\]
We define, for any $\alpha\in(0,1)$,
\[
\varsigma^{\ell}_{\alpha \ell}:= \inf\{i\geq 0 \colon D^{\ell}_i\leq \alpha \ell \}
\qquad\text{ and }\qquad
      \varsigma^{\ell}_0:= \inf\{i\geq 0 \colon D^{\ell}_i\leq 0 \},
\]
 and note that, by construction, $W_{\ell}$ is linear in the interval $[i/\ell^2,(i+1)/\ell^2]$, for all $i\geq 0$. Therefore, we have that
\[
(\varsigma^{\ell}_{\alpha \ell}-1)/\ell^2 < \calH_{\alpha}(W_{\ell}) \leq \varsigma^{\ell}_{\alpha \ell}/\ell^2 \qquad\text{ and }\qquad  (\varsigma^{\ell}_0-1)/\ell^2 < \calH_{0}(W_{\ell}) \leq\varsigma^{\ell}_0/\ell^2.
\]
Observe that, by definition of $\hat \tau_\ell$, $D^{\ell}_i>0$ for all $0\leq i< \hat \tau_\ell$, and
\begin{align*}
D^{\ell}_{\hat \tau_\ell} = \hat A^{\ell}_{\hat \tau_\ell} - \hat B^{\ell}_{\hat \tau_\ell} &= \hat A_{\hat \tau_\ell-1} - \hat B_{\hat \tau_\ell-1}+\frac{1}{2}\Big( \hat{\overline{X}}^\ell_{\hat \tau_\ell-1} +\hat{\underline{Y}}^\ell_{\hat \tau_\ell-1} - \hat{\underline{X}}^\ell_{\hat \tau_\ell-1} - \hat{\overline{Y}}^\ell_{\hat \tau_\ell-1} \Big)\\
&= D^{\ell}_{\hat \tau_\ell-1} +\frac{1}{2}\Big( \hat{\overline{X}}^\ell_{\hat \tau_\ell-1} +\hat{\underline{Y}}^\ell_{\hat \tau_\ell-1} - \hat{\underline{X}}^\ell_{\hat \tau_\ell-1} - \hat{\overline{Y}}^\ell_{\hat \tau_\ell-1} \Big)\\
&\geq \hat{\underline{X}}^\ell_{\hat \tau_\ell-1} + \hat{\overline{Y}}^\ell_{\hat \tau_\ell-1} + \frac{1}{2}\Big( \hat{\overline{X}}^\ell_{\hat \tau_\ell-1} +\hat{\underline{Y}}^\ell_{\hat \tau_\ell-1} - \hat{\underline{X}}^\ell_{\hat \tau_\ell-1} - \hat{\overline{Y}}^\ell_{\hat \tau_\ell-1} \Big)\\
&= \frac{1}{2}\Big( \hat{\overline{X}}^\ell_{\hat \tau_\ell-1} +\hat{\underline{Y}}^\ell_{\hat \tau_\ell-1} + \hat{\underline{X}}^\ell_{\hat \tau_\ell-1} + \hat{\overline{Y}}^\ell_{\hat\tau_\ell-1} \Big)>0.
\end{align*}
This implies that $\hat \tau_\ell \leq \varsigma^{\ell}_0-1$, which further implies that
\[
    \hat{\tau}_\ell/\ell^2 \leq \calH_{0}(W_{\ell}).
\]
Since $W_{\ell,t} \geq 0$  for all $0\leq t\leq \calH_{0}(W_\ell)$, we also have that 
\[
    \int_{0}^{\hat{\tau}_\ell/\ell^2} W_{\ell,t}\lambda(t)\mu(t) \,\dd t  \leq   \int_{0}^{\calH_{0}(W_\ell) } W_{\ell,t}\lambda(t)\mu(t) \,\dd t = \calI_{0}(W_\ell).
\]
Therefore, for any $x,y\geq 0$,
\[
  \P\Big( \hat{\tau}_\ell/\ell^2 \leq x, \int_{0}^{\hat{\tau}_\ell/\ell^2} W_{\ell,t} \lambda(t)\mu(t)\,\dd t \leq y  \Big)
   \geq \P\Big( \calH_{0}(W_\ell) \leq x, \calI_{0}(W_\ell)\leq y  \Big).
\]
This, together with Lemma~\ref{Lem:Wlt-weak-convergence} and Remark~\ref{Rem:Continuity-of-stopping-time-vert-homo-semi}, implies that, for any $x,y\geq 0$,
\begin{align}
    \label{Equ:Weak-hat-tau_l-liminf-semi}
\liminf_{\ell\to\infty} \P\Big( \hat{\tau}_\ell/\ell^2 \leq x, \int_{0}^{\hat{\tau}_\ell/\ell^2} W_{\ell,t} \,\dd t \leq y  \Big) \geq \P( \calH_0(1+M_{.})\leq x, \calI_0(1+M_{.})\leq y).
\end{align}
On the other hand, on the event $\{ X_{\hat{\tau}_\ell} + Y_{\hat{\tau}_\ell} \leq \alpha \ell \}$, we have that $\hat{\tau}_\ell \geq \varsigma^{\ell}_{\alpha \ell}$, and consequently,
\[
\int_{0}^{\hat{\tau}_\ell/\ell^2} W_{\ell,t} \lambda(t)\mu(t) \,\dd t \geq \calI_{\alpha}(W_\ell).
\]
Therefore, we have
\[
    \P\Big( \hat{\tau}_\ell/\ell^2 \leq x, \int_{0}^{\hat{\tau}_\ell/\ell^2} W_{\ell,t} \lambda(t)\mu(t) \,\dd t \leq y  \Big)\le\P( \calH_{\alpha}(W_\ell) \leq x, \calI_{\alpha}(W_\ell) \leq y)+ \P\Big( \max_{0\leq i \leq \lceil x\ell^2 \rceil} \big( \hat{\underline{X}}^\ell_{i} + \hat{\overline{Y}}^\ell_{i} \big) > \alpha \ell\Big),
\]
where
\[
    \P\Big( \max_{0\leq i \leq \lceil x\ell^2 \rceil} \big( \hat{\underline{X}}^\ell_{i} + \hat{\overline{Y}}^\ell_{i} \big) > \alpha \ell\Big)
    \le \sum_{i=0}^{\lceil x \ell^2 \rceil} \P\big( \hat{\underline{X}}^\ell_{i} + \hat{\overline{Y}}^\ell_{i} >   \alpha \ell  \big)
    \leq  2(\lceil x \ell^2 \rceil + 1) \e^{- \lambda_{\min} \alpha \ell/2}, 
\]
which tends to zero, as $\ell\to\infty$.
This, together with Lemma~\ref{Lem:Wlt-weak-convergence} and Remark~\ref{Rem:Continuity-of-stopping-time-vert-homo-semi}, implies that
\begin{align}
\label{Equ:Weak-hat-tau_l-limsup-semi}
    \limsup_{\ell\uparrow\infty} \P\Big( \hat{\tau}_\ell/\ell^2 \leq x, \int_{0}^{\hat{\tau}_\ell/\ell^2} W_{\ell,t} \lambda(t)\mu(t) \,\dd t \leq y  \Big) \leq \P\big( \calH_{\alpha}(1+M_{.})\leq x, \calI_{\alpha}(1+M_{.})\leq y\big).
\end{align}
Now, note that, for any $f\in \calC[0,\infty)$ with $f(0)>a$, $\alpha\mapsto\calH_{\alpha}(f)$ is right-continuous at $a$.  Therefore, by combining~\eqref{Equ:Weak-hat-tau_l-liminf-semi} and~\eqref{Equ:Weak-hat-tau_l-limsup-semi} and then taking $\alpha\downarrow 0$, we get
\[
\lim_{\ell\to\infty} \P\Big( \hat{\tau}_\ell/\ell^2 \leq x, \int_{0}^{\hat{\tau}_\ell/\ell^2} W_{\ell,t} \lambda(t)\mu(t)\,\dd t \leq y  \Big) = \P\big( \calH_0(1+M_{.})\leq x, \calI_0(1+M_{.})\leq y  \big).
\]
This proves the result.
\end{proof}

\begin{proof}[Proof of Proposition~\ref{Prop:Conv-in-dist-upto-tau1-vert-homo}]
In view of Proposition~\ref{Prop:Error_in_area_approximation-homogeneous}, Lemma~\ref{Lem:Wlt-integration-weak-convergence} and~\eqref{Equ:Coupling-homogeneous-to-independent}, it is enough to show that, for any $K > 0$, 
\begin{align}
    \sup_{0\leq t\leq K} \Big| \int_0^t W_{\ell,s}\lambda(s)\mu(s)\,\dd s -  \frac{1}{\ell^3} \sum_{i=0}^{\lfloor \ell^2 t\rfloor-1 } D^{\ell}_i\lambda(i/\ell^2)\mu(i/\ell^2) \Big| \xlongrightarrow{ p } 0,\qquad\text{as }\ell \to \infty.
    \label{Equ:Error-in-Integral-approx-coupled-process}
\end{align}
For this, from the definition of $W_{\ell}$ in~\eqref{Equ:Def-Wl}, we have that
\begin{align*}
         \ell \int_0^t W_{\ell,s}\lambda(s)\mu(s)\,\dd s &=\sum_{i=0}^{\lfloor \ell^2 t\rfloor-1} \int_{i/\ell^2}^{(i+1)/\ell^2} \Big( D^{\ell}_i +(\ell^2s-i)(D^{\ell}_{i+1}-D^{\ell}_i) \Big)\lambda(s)\mu(s)\,\dd s \\
        & \qquad +\int_{\lfloor \ell^2 t\rfloor/\ell^2}^t \Big( D^{\ell}_{\lfloor \ell^2 t\rfloor} +(\ell^2s-\lfloor \ell^2 t\rfloor)(D^{\ell}_{\lfloor \ell^2 t\rfloor+1}-D^{\ell}_{\lfloor \ell^2 t\rfloor}) \Big)\lambda(s)\mu(s)\,\dd s
\end{align*}
and therefore we get
    \begin{align}
         \int_0^t W_{\ell,s}\lambda(s)\mu(s)\,\dd s - \frac{1}{\ell^3} \sum_{i=0}^{\lfloor \ell^2 t\rfloor-1 } D^{\ell}_i\lambda(i/\ell^2)\mu(i/\ell^2) =\mathfrak E^{\ell}_{1,t}+\mathfrak E^{\ell}_{2,t}+\mathfrak E^{\ell}_{3,t},
       \label{Equ:sum-of-Eil}
    \end{align}
where
\begin{align*}
 \mathfrak E^{\ell}_{1,t}  &:= \frac{1}{\ell} \sum_{i=0}^{\lfloor \ell^2 t\rfloor-1 } D^{\ell}_i\Big(\int_{i/\ell^2}^{(i+1)/\ell^2} \lambda(s)\mu(s)\,\dd s  - \frac{\lambda(i/\ell^2)\mu(i/\ell^2) }{\ell^2}\Big),\\
 \mathfrak E^{\ell}_{2,t}  &:= \frac{1}{\ell} \sum_{i=0}^{\lfloor \ell^2 t\rfloor-1 } (D^{\ell}_{i+1}-D^{\ell}_i)\int_{i/\ell^2}^{(i+1)/\ell^2} (\ell^2s-i)\lambda(s)\mu(s)\,\dd s, \text{ and}\\
 \mathfrak E^{\ell}_{3,t} &:= \frac{1}{\ell} \int_{\lfloor \ell^2 t\rfloor/\ell^2}^t \Big( D^{\ell}_{\lfloor \ell^2 t\rfloor} +(\ell^2s-\lfloor \ell^2 t\rfloor)(D^{\ell}_{\lfloor \ell^2 t\rfloor+1}-D^{\ell}_{\lfloor \ell^2 t\rfloor}) \Big)\lambda(s)\mu(s)\,\dd s.
\end{align*}
We establish the desired result in several steps.

\medskip
\noindent
\emph{Step (1): first term}.
Since $\lambda$ and $\mu$ are uniformly continuous on the compact interval $[0,K]$, writing 
     \[
     a_{i,\ell}:= \int_{i/\ell^2}^{(i+1)/\ell^2} \lambda(s)\mu(s)\,\dd s  - \frac{\lambda(i/\ell^2)\mu(i/\ell^2) }{\ell^2},
     \]
     we have that
     \begin{align}
         \lim_{\ell\to\infty}\ell^2  \max_{0\leq i \leq \lfloor K\ell^2\rfloor-1} |a_{i,\ell}| =0.
         \label{Equ:l^2*ail}
     \end{align}
Now, by~\eqref{Equ:Def-V}, we obtain that
\begin{align*}
       \mathfrak E^{\ell}_{1,t}
     &= \frac{1}{\ell} \sum_{i=0}^{\lfloor \ell^2 t\rfloor-1 }  a_{i,\ell} D^{\ell}_i 
    = \frac{1}{\ell} \sum_{i=0}^{\lfloor \ell^2 t\rfloor-1 } a_{i,\ell} \Big( \ell+ \sum_{j=1}^i \frac{V_j}{\lambda((j-1)/\ell^2)}\Big) \\
    &= \sum_{i=0}^{\lfloor \ell^2 t\rfloor-1 } a_{i,\ell} + \sum_{j=1}^{\lfloor \ell^2 t\rfloor-1 } \frac{V_j}{\lambda((j-1)/\ell^2)} \sum_{i=j}^{\lfloor \ell^2 t\rfloor-1} \frac{a_{i,\ell}}{\ell} 
    = \sum_{i=0}^{\lfloor \ell^2 t\rfloor-1 } a_{i,\ell} + \sum_{i=1}^{\lfloor \ell^2 t\rfloor-1 } \frac{a_{i,\ell}}{\ell} \sum_{j=1}^{i} \frac{V_j}{\lambda((j-1)/\ell^2)}.
\end{align*}
This implies
\[
    \sup_{0\leq t\leq K}  \big| \mathfrak E^{\ell}_{1,t} \big|
    \leq \lfloor \ell^2 K\rfloor \max_{0\leq i \leq \lfloor \ell^2 K \rfloor-1} |a_{i,\ell}|+ \frac{\lfloor \ell^2 K \rfloor}{\ell}  \max_{0\leq i \leq \lfloor \ell^2 K \rfloor-1} |a_{i,\ell}| 
    \sup_{0\leq i\leq \lfloor \ell^2 K \rfloor-1} \Big|
    \sum_{j=1}^{i} \frac{V_j}{\lambda((j-1)/\ell^2)} 
    \Big|.
\]
Therefore, using $(a+b)^2\leq 2(a^2+b^2)$, together with~\eqref{Equ:l^2*ail} and Doob's maximal inequality, we obtain
\begin{align*}
    \limsup_{\ell \uparrow \infty}&
    \E\Big[\Big(\sup_{0\leq t\leq K}  \big| \mathfrak E^{\ell}_{1,t} \big|\Big)^2 \Big]\\
    &\leq \limsup_{\ell \uparrow \infty} 2\Big(
    \lfloor \ell^2 K\rfloor \max_{0\leq i \leq \lfloor \ell^2 K\rfloor-1} |a_{i,\ell}| \Big)^2
    \Big(1 + \frac{1}{\ell^2} \mathbb{E}\Big[\sup_{0\leq i\leq \lfloor \ell^2 K\rfloor-1} \Big(
    \sum_{j=1}^{i} \frac{V_j}{\lambda((j-1) /\ell^2)} 
    \Big)^2\Big] \Big) \\
    &\leq \limsup_{\ell \uparrow \infty}
    2 \Big(\lfloor \ell^2 K\rfloor \max_{0\leq i \leq \lfloor \ell^2 K\rfloor-1} |a_{i,\ell}| \Big)^2
    \Big(1 + \frac{4}{\ell^2} \mathbb{E}\Big[\Big(
    \sum_{j=1}^{\lfloor \ell^2 K\rfloor-1} \frac{V_j}{\lambda((j-1) /\ell^2)} 
    \Big)^2 \Big] \Big) \\
    &\leq \limsup_{\ell \uparrow \infty}
    2 K^2 \Big(\ell^2 \max_{0\leq i \leq \lfloor \ell^2 K\rfloor-1} |a_{i,\ell}|\Big)^2\Big( 1+ 4\int_0^K \frac{1}{(\lambda(x))^2}\,\dd x \Big)  = 0,
\end{align*}
which implies that
\begin{align}
     \sup_{0\leq t\leq K}  \big| \mathfrak E^{\ell}_{1,t} \big| \xlongrightarrow{ p } 0,\qquad\text{as }\ell \to \infty.
     \label{Equ:E1l-to-0-in-P}
\end{align}

\medskip
\noindent
\emph{Step (2): second term}.
For the second term in~\eqref{Equ:sum-of-Eil}, we observe that
   \begin{align*}
        &\sup_{0 \leq t \leq  K } \big|\mathfrak E^{\ell}_{2,t}\big|
        \leq \frac{1}{\ell} \sum_{i=0}^{\lfloor  \ell^2 K \rfloor-1 } \int_{i/\ell^2}^{(i+1)/\ell^2} (\ell^2s-i)\lambda(s)\mu(s)\,\dd s\frac{1}{\lambda((i+1)/\ell^2)}|V_{i+1}|.
    \end{align*} 
Taking expectations and passing to the limit yields
 \begin{align*}
    \limsup_{\ell \to \infty}
    \E\Big[\sup_{0\leq t\leq K}  \big| \mathfrak E^{\ell}_{2,t} \big| \Big]
    &\leq \limsup_{\ell \uparrow \infty} \frac{1}{\ell} \sum_{i=0}^{\lfloor  \ell^2 K \rfloor-1 } \int_{i/\ell^2}^{(i+1)/\ell^2} (\ell^2s-i) \frac{\lambda(s)\mu(s)}{\lambda((i+1)/\ell^2)}\,\dd s \E[|V_1|] \\
    &\leq \limsup_{\ell \uparrow \infty}
    \frac{1}{\ell}\E[|V_1|] K \frac{\lambda_{\max}\mu_{\max}}{\lambda_{\min}}  = 0,
\end{align*}   
and hence,
\begin{align}
     \sup_{0\leq t\leq K}  \big| \mathfrak E^{\ell}_{2,t} \big| \xlongrightarrow{ p } 0,\qquad\text{as }\ell \to \infty.
     \label{Equ:E2l-to-0-in-P}
\end{align} 

\medskip
\noindent
\emph{Step (3): third term}.
Finally, for the third term in~\eqref{Equ:sum-of-Eil}, we observe that
    \begin{align*}
        \mathfrak E^{\ell}_{3,t}
        &= \frac{D^{\ell}_{\lfloor t\ell^2\rfloor}}{\ell}\int_{\lfloor t\ell^2\rfloor/\ell^2}^t \lambda(s)\mu(s)\,\dd s + \frac{(D^{\ell}_{\lfloor t\ell^2\rfloor+1}-D^{\ell}_{\lfloor t\ell^2\rfloor})}{\ell}\int_{\lfloor t\ell^2\rfloor/\ell^2}^t (\ell^2s-\lfloor t\ell^2\rfloor)\lambda(s)\mu(s)\,\dd s\\
        &= \frac{1}{\ell}\Big( \ell+ \sum_{j=1}^{\lfloor t\ell^2\rfloor} \frac{V_j}{\lambda((j-1)/\ell^2)}\Big)\int_{\lfloor t\ell^2\rfloor/\ell^2}^t \lambda(s)\mu(s)\,\dd s+ \frac{V_{\lfloor t\ell^2\rfloor+1}}{\ell \lambda(\lfloor t\ell^2\rfloor/\ell^2)}\int_{\lfloor t\ell^2\rfloor/\ell^2}^t (\ell^2s-\lfloor t\ell^2\rfloor)\lambda(s)\mu(s)\,\dd s.
    \end{align*}
   This implies that
    \begin{align*}
        \sup_{0 \leq t \leq T}|\mathfrak E^{\ell}_{3,t}| 
        \leq \frac{\lambda_{\max}\mu_{\max}}{\ell^3\lambda_{\min}}\Big(\ell \lambda_{\min}+
        \sum_{j=1}^{\lfloor t\ell^2\rfloor+1} |V_j|\Big).
    \end{align*}
    Taking expectations and applying Markov's inequality, we conclude that
    \begin{align}
        \sup_{0 \leq t \leq T}|\mathfrak E^{\ell}_{3,t}|\xlongrightarrow{ p } 0,\qquad\text{as }\ell \to \infty.
        \label{Equ:E3l-to-0-in-P}
    \end{align}
Now, combining~\eqref{Equ:sum-of-Eil},~\eqref{Equ:E1l-to-0-in-P},~\eqref{Equ:E2l-to-0-in-P}, and~\eqref{Equ:E3l-to-0-in-P} establishes~\eqref{Equ:Error-in-Integral-approx-coupled-process}, thereby proving the result.
\end{proof}

%%%%%%%%%%%%%%%%%%%%%%%%%%%%%%%%%%%%%%
%%%%%%%%%%%%%%%%%%%%%%%%%%%%%%%%%%%%%%
%%%%%%%%%%%%%%%%%%%%%%%%%%%%%%%%%%%%%%

\subsubsection{Proofs for Section~\ref{Sec:prior_time_0}}\label{Sec:proof_prior_time_0}
%In this section, we present the proofs for the results presented in Section~\ref{Sec:prior_time_0}.
\begin{proof}[Proof of Proposition~\ref{Prop:tau_1_vert_inhom_decay}]
  Observe that
  \begin{align}
        \P\big(\tau'_{(\eta, j, \ell)} \geq t \big)
        &= \P\big(\hat\tau_{(\eta, j, \ell)} \geq t \big)
        \le \P\Big(\min_{0\leq k \leq t} D_k^{(\eta, j, \ell)} > 0\Big)\nonumber\\
        &= \P\Big(\min_{0\leq k \leq t} \Big( \eta + \sum_{i = 0}^{k-1} \big(D^{(\eta, j, \ell)}_{i+1} - D^{(\eta, j, \ell)}_{i}\big) \Big) > 0\Big)= \P\Big(\max_{0\leq k \leq t} \sum_{i = 0}^{k-1} (-R_i) < \eta\Big),
        \label{Equ:tail-of-tau1^abjr}
    \end{align} 
    where $R_i:= D^{(\eta, j, \ell)}_{i+1} - D^{(\eta, j, \ell)}_{i} = (\hat{\overline{X}}^{\ell}_i+\hat{\underline{Y}}^{\ell}_i -\hat{\underline{X}}^{\ell}_i-\hat{\overline{Y}}^{\ell}_i)/2$. Since, for all $i\geq 0$, the random variables $\hat{\overline{X}}^{\ell}_i,\hat{\underline{X}}^{\ell}_i ,\hat{\overline{Y}}^{\ell}_i,\hat{\underline{Y}}^{\ell}_i$ are i.i.d.\ $\Exp(\lambda((i+j)/\ell))$, there exist constants $0 < \bar{C}_1 < \bar{C}_2 < \infty$, depending only on $\lambda_{\max}$ and $\lambda_{\min}$, such that, for all $i \geq 0$,
\[
\E[{R_i}^2], \E[|R_i|^3]\in (\bar{C}_1, \bar{C}_2).
\]
   Therefore, by applying~\cite[Lemma~1.7]{arak1975}, the right-hand side of~\eqref{Equ:tail-of-tau1^abjr} is bounded by 
\[
    \bar{C}_3\Big( \eta+ \max_{0\leq k \leq t} \frac{\E[|R_k|^3]}{\E[{R_k}^2]} \Big)\Big(\sum_{k=0}^{\lfloor t \rfloor-1} \E[{R_k}^2]\Big)^{-1/2}
    \leq
    \frac{\bar{C}_3}{\sqrt{\bar{C}_1\lfloor t \rfloor}}\Big(\eta+ \frac{\bar{C}_2}{\bar{C}_1} \Big),
\]
    for some $\bar{C}_3\in(0,\infty)$.
    This proves the result.
\end{proof}
In order to prepare the proof of Proposition~\ref{Prop:tau0-tau1-decay-rate-vertical-homo-semi}, we show the following. 
\begin{lemma}
\label{Lem:tau_0_is_not_tau_1}
There exists a constant $\tilde{C} \in (0, \infty)$ such that, for all sufficiently large $\eta$ and all $j \geq 0$, 
    \begin{align*}
        \P\big(\tau_{(\eta, j,\ell)} \neq \tau'_{(\eta, j,\ell)} \big)
        \leq 1- \tilde C(\log\eta)^{-2}.
    \end{align*}
\end{lemma}
\begin{proof}[Proof of Lemma~\ref{Lem:tau_0_is_not_tau_1}]
  Note that
  \begin{align}
     \P\big(\tau_{(\eta, j,\ell)} \neq \tau'_{(\eta, j,\ell)} \big)
     &\leq \P\Big(\tau_{(\eta, j,\ell)} \neq \tau'_{(\eta, j,\ell)}, \tau'_{(\eta, j,\ell)} \leq \eta^4, \max_{0\leq i\leq \eta^4} \underline{X}^{(\eta,j,\ell)}_i + \overline{Y}^{(\eta,j,\ell)}_i \leq (\log\eta)^2 \Big)\nonumber\\
     &\qquad+ \P\big(\tau'_{(\eta, j,\ell)} > \eta^4 \big) +\P\Big( \max_{0\leq i\leq \eta^4} \underline{X}^{(\eta,j,\ell)}_i + \overline{Y}^{(\eta,j,\ell)}_i > (\log\eta)^2 \Big).
     \label{Equ:tau_0_is_not_tau_1-decompose}
  \end{align}
  Moreover, 
  \begin{align}
  \P\Big( \max_{0\leq i\leq \eta^4} \underline{X}^{(\eta,j,\ell)}_i + \overline{Y}^{(\eta,j,\ell)}_i > (\log\eta)^2 \Big)
      &\leq\sum_{i=0}^{\lfloor \eta^{4}\rfloor}\P\big(\underline{X}^{(\eta,j,\ell)}_i >(\log\eta)^2/2 \big)+ \sum_{i=0}^{\lfloor \eta^{5}\rfloor}\P\big(\overline{Y}^{(\eta,j,\ell)}_i >(\log\eta)^2/2 \big)\nonumber\\
        &= 2(\lfloor \eta^{4}\rfloor+1) \e^{-\lambda_{\min}(\log\eta)^2/2 }=o\big((\log\eta)^{-2} \big),\qquad \text{ as $\eta\to\infty$},
        \label{Equ:tau_0_is_not_tau_1-decompose-1}
  \end{align}
and, by Proposition~\ref{Prop:tau_1_vert_inhom_decay}, for all  $\eta\geq 1$,
  \begin{align}
    \P\big(\tau'_{(\eta, j,\ell)} > \eta^4 \big) \leq  \bar{C}\eta^{-1}. 
    \label{Equ:tau_0_is_not_tau_1-decompose-2}
  \end{align}
We define
\[
N^{(\eta,j,\ell)}_i := \calP^{(j)}_i\big([B^{(\eta,j,\ell)}_i, A^{(\eta,j,\ell)}_i]\big)
\]
and recall that $\calF_0^{(j)}$ is the trivial $\sigma$-algebra, and for all $i \geq 1$, $\calF_i^{(j)}$ is the $\sigma$-algebra generated by $\{\calP_{\ell, k}^{(j)}\colon 0 \leq k \leq i - 1\}$. Therefore, $A^{(\eta,j,\ell)}_i$, $B^{(\eta,j,\ell)}_i$, and $\ind{\tau'_{(\eta, j,\ell)} \geq i}$ are all $\calF_i^{(j)}$-measurable. Hence, we have
\begin{align}
 \P&\Big(\tau_{(\eta, j,\ell)} \neq \tau'_{(\eta, j,\ell)}, \tau'_{(\eta, j,\ell)} \leq \eta^4, \max_{0\leq i\leq \eta^4} \underline{X}^{(\eta,j,\ell)}_i + \overline{Y}^{(\eta,j,\ell)}_i \leq (\log\eta)^2 \Big) \nonumber \\
 &\leq \sum_{i=0}^{\eta^4} 
 \E\Big[
 \ind{\tau'_{(\eta, j,\ell)} =i}\ind{N^{(\eta,j,\ell)}_i=1} \ind{A^{(\eta,j,\ell)}_i-B^{(\eta,j,\ell)}_i \leq (\log\eta)^2}
 \Big]\nonumber\\
 &= \sum_{i=0}^{\eta^4}  \E\Big[ \ind{\tau'_{(\eta, j,\ell)} \geq i}\ind{N^{(\eta,j,\ell)}_i=1}\ind{A^{(\eta,j,\ell)}_i-B^{(\eta,j,\ell)}_i \leq (\log\eta)^2} \Big]\nonumber\\
 &= \sum_{i=0}^{\eta^4} \E\Big[\P\big(N^{(\eta,j,\ell)}_i=1 | \calF_i^{(j)}, \ind{N^{(\eta,j,\ell)}_i\in\{0,1\}} \big) \nonumber\\
 &\qquad\qquad\qquad\qquad
 \ind{\tau'_{(\eta, j,\ell)} \geq i}\ind{N^{(\eta,j,\ell)}_i\in\{0,1\}}\ind{A^{(\eta,j,\ell)}_i-B^{(\eta,j,\ell)}_i \leq (\log\eta)^2}\Big] \nonumber\\
  &= \sum_{i=0}^{\eta^4} \E\Big[\tfrac{\lambda((i+j)/\ell^2,0)(A^{(\eta,j,\ell)}_i-B^{(\eta,j,\ell)}_i)}{1+\lambda((i+j)/\ell^2,0)(A^{(\eta,j,\ell)}_i-B^{(\eta,j,\ell)}_i)}t\ind{\tau'_{(\eta, j,\ell)} \geq i}\ind{N^{(\eta,j,\ell)}_i\in\{0,1\}}\ind{A^{(\eta,j,\ell)}_i-B^{(\eta,j,\ell)}_i \leq (\log\eta)^2}
 \Big] \nonumber\\
  &\leq \sum_{i=0}^{\eta^4} 
 \E\Big[\frac{\lambda_{\max}(\log\eta)^2}{1+\lambda_{\max}(\log\eta)^2}
 \ind{\tau'_{(\eta, j,\ell)} = i}\Big] 
 = \frac{\lambda_{\max}(\log\eta)^2}{1+\lambda_{\max}(\log\eta)^2}= 1-\frac{1}{1+\lambda_{\max}(\log\eta)^2}.
 \label{Equ:tau_0_is_not_tau_1-decompose-3}
\end{align}
Thus, combining~\eqref{Equ:tau_0_is_not_tau_1-decompose},~\eqref{Equ:tau_0_is_not_tau_1-decompose-1},~\eqref{Equ:tau_0_is_not_tau_1-decompose-2}, and~\eqref{Equ:tau_0_is_not_tau_1-decompose-3} yields the result.
\end{proof}

\begin{proof}[Proof of Proposition~\ref{Prop:tau0-tau1-decay-rate-vertical-homo-semi}]
Observe that
 \begin{align}
       \P(\tau_{\ell} - \tau'_{\ell} > \ell^{\alpha})
        &\leq 
    \E\Big[\mathbb{P}\big(\tau_{\ell} - \tau'_{\ell} > \ell^{\alpha} | \mathcal{F}_{\tau_1^{(\ell)}+1} \big)  \ind{A^\ell_{\tau'_{\ell} + 1} - B^\ell_{\tau'_{\ell} + 1} \leq (\log\ell)^2} \Big] 
        + \P\big(A^\ell_{\tau'_{\ell} + 1} - B^\ell_{\tau'_{\ell} + 1} > (\log\ell)^2\big) \nonumber\\
    &= \E\Big[\P\big(\tau_{(\eta,j,\ell)}+1 > \ell^{\alpha} \big)\big|_{\substack{\eta = A^\ell_{\tau'_{\ell} + 1}- B^\ell_{\tau'_{\ell} + 1} ,\,\,j = \tau'_{\ell}+1}}  
     \ind{A^\ell_{\tau'_{\ell} + 1} - B^\ell_{\tau'_{\ell} + 1} \leq (\log\ell)^2} \Big] \nonumber\\
        &\qquad\qquad+ \P\big(\underline{X}^\ell_{\tau'_{\ell}} + \overline{Y}^\ell_{\tau'_{\ell}} > (\log\ell)^2\big) \nonumber \\
   &\leq \E\Big[\P\big(\tau_{((\log\ell)^2,j,\ell)}+1 > \ell^{\alpha} \big)\big|_{\substack{j = \tau'_{\ell}+1}}\Big] 
   + \P\big(\underline{X}^\ell_{\tau'_{\ell}} + \overline{Y}^\ell_{\tau'_{\ell}} > (\log\ell)^2\big).
     \label{Equ:tail-bound-for-tau0-tau1}
    \end{align}
Here, in the last inequality, we use the fact that $\tau'_{(\eta,j,\ell)}$ is increasing in $\eta$. To bound the right-hand side of~\eqref{Equ:tail-bound-for-tau0-tau1}, we now carry out some technical preparation.
Denote $\tau'^{(0)}_{(\eta,j,\ell)} := -1$ and $\tau'^{(1)}_{(\eta,j,\ell)} := \tau'_{(\eta,j,\ell)}$, and, for all $m \geq 2$, define
\[
        \tau'^{(m)}_{(\eta,j,\ell)}
        := \inf\big\{i > \tau'^{(m-1)}_{(\eta,j,\ell)} \colon \calP_{\ell, i}([B_i, A_i])\in \{0,1\}\big\}.
\] 
    Observe that, by the definition of the filtration $\{\mathcal{F}^{(j)}_i\colon i \geq 0\}$, the random times $(\tau_{(\eta,j,\ell)} + 1)$, $(\tau'_{(\eta,j,\ell)} + 1)$, and $(\tau'^{(m)}_{(\eta,j,\ell)} + 1)$ are stopping times with respect to this filtration.
Then, we have
    \begin{align}
        \P\big(\tau_{(\eta,j,\ell)}+1 > \ell^{\alpha}\big) 
        &\leq G_1+G_2+G_3,
        \label{Equ:tail-of-tau0-eta-j-l}
    \end{align}
where
\begin{align*}
G_1&:=\P\Big(\big(\tau_{(\eta,j,\ell)}+1\big) > \ell^{\alpha}, \big( \tau'^{(\lfloor \log \ell\rfloor^4)}_{(\eta,j,\ell)} +1\big)\leq \ell^{\alpha}, \mathscr{E}_{\alpha}^{(\eta,j,\ell)}\Big),\\
    G_2&:= \P\Big(\big(\tau'^{(\lfloor \log \ell\rfloor^4)}_{(\eta,j,\ell)} +1\big)> \ell^{\alpha}, \mathscr{E}_{\alpha}^{(\eta,j,\ell)}\Big),\\
    G_3&:= \P\Big({\mathscr{E}_{\alpha}^{(\eta,j,\ell)}}^{\complement}\Big), and\\
    \mathscr{E}_{\alpha}^{(\eta,j,\ell)}&:=\Big\{ \max_{0 \leq i \leq \ell^{\alpha+1}} \big(\overline{X}^{(\eta,j,\ell)}_i \vee \underline{X}^{(\eta,j,\ell)}_i \vee \overline{Y}^{(\eta,j,\ell)}_i \vee \underline{Y}^{(\eta,j,\ell)}_i\big) \leq (\log\ell)^2/4 \Big\}.
\end{align*}
We also write $\mathscr{E}_{\alpha}^{(\ell)} \equiv \mathscr{E}_{\alpha}^{(\ell,0,\ell)}$. 
We now proceed to bound the right-hand side of~\eqref{Equ:tail-of-tau0-eta-j-l} in several steps.

\medskip
\noindent
\emph{Step (1): Bounding the third term.}
\begin{align}  G_3=\P\Big({\mathscr{E}_{\alpha}^{(\eta,j,\ell)}}^{\complement}\Big)
&\leq 4 \big(\lfloor \ell^{\alpha+1}\rfloor +1\big) \e^{-\lambda_{\min}(\log\ell)^2/4}
        \leq \bar{\mathsf{C}}_3\ell^{\alpha+1} \e^{-\bar{\mathsf{C}}'_3(\log\ell)^2 },\quad\text{ for some }\bar{\mathsf{C}}_3, \bar{\mathsf{C}}'_3\in(0,\infty).
        \label{Equ:tail-of-tau0-eta-j-l-3}
\end{align}
   
\medskip
\noindent
\emph{Step (2): Bounding the second term.}
\begin{align*}
   G_2 &= \P\Big(\big(\tau'^{(\lfloor \log \ell\rfloor^4)}_{(\eta,j,\ell)} +1\big)> \ell^{\alpha}, \mathscr{E}_{\alpha}^{(\eta,j,\ell)}\Big)\\
   &\leq \sum_{m=1}^{\lfloor \log \ell\rfloor^4} \P\Big(\tau'^{(m)}_{(\eta,j,\ell)}-\tau'^{(m-1)}_{(\eta,j,\ell)} > \frac{\ell^{\alpha}}{\lfloor \log \ell\rfloor^4},  \big(\tau'^{(m-1)}_{(\eta,j,\ell)} +1\big)\leq\ell^{\alpha},\mathscr{E}_{\alpha}^{(\eta,j,\ell)} \Big)\\
   & \leq  \P\Big(\tau'^{(1)}_{(\eta,j,\ell)}+1 > \frac{\ell^{\alpha}}{\lfloor \log \ell\rfloor^4} \Big)\\
   &\qquad
   +\sum_{m=2}^{\lfloor \log \ell\rfloor^4} \E\Big[\P\Big(\tau'^{(m)}_{(\eta,j,\ell)}-\tau'^{(m-1)}_{(\eta,j,\ell)} > \frac{\ell^{\alpha}}{\lfloor \log \ell\rfloor^4} \Big| \calF_{\tau'^{(m-1)}_{(\eta,j,\ell)}+1} \Big)\ind{A^{(\eta,j,\ell)}_{\tau'^{(m-1)} + 1} - B^{(\eta,j,\ell)}_{\tau'^{(m-1)} + 1} \leq (\log\ell)^2} \Big]\\
   &\leq  \P\Big(\tau'_{(\eta,j,\ell)}+1 > \frac{\ell^{\alpha}}{\lfloor \log \ell\rfloor^4} \Big)\\
   &\qquad
   +\sum_{m=2}^{\lfloor \log \ell\rfloor^4} \E\Big[\P\Big(\tau_{1}^{(\eta',j,\ell)}+1 > \frac{\ell^{\alpha}}{\lfloor \log \ell\rfloor^4}  \Big)\Big|_{\substack{\eta' = A^{(\eta,j,\ell)}_{\tau'^{(m-1)} + 1} - B^{(\eta,j,\ell)}_{\tau'^{(m-1)} + 1}, \,\,j = \tau'^{(m-1)}_{(\eta,j,\ell)}+1}}\\
   &\qquad\qquad\qquad\qquad
   \qquad\qquad\qquad\qquad\qquad\qquad\qquad\qquad
   \ind{A^{(\eta,j,\ell)}_{\tau'^{(m-1)} + 1} - B^{(\eta,j,\ell)}_{\tau'^{(m-1)} + 1} \leq (\log\ell)^2} \Big].
\end{align*}
Here, to improve readability, we write $A^{(\eta,j,\ell)}_{\tau'^{(m-1)}_{(\eta,j,\ell)}+1}$ as $A^{(\eta,j,\ell)}_{\tau'^{(m-1)} + 1}$ and $B^{(\eta,j,\ell)}_{\tau'^{(m-1)}_{(\eta,j,\ell)}+1}$ as $B^{(\eta,j,\ell)}_{\tau'^{(m-1)} + 1}$. Observing that $\tau'_{(\eta',j,\ell)}$ is increasing in $\eta'$ and invoking Proposition~\ref{Prop:tau_1_vert_inhom_decay}, we deduce that
\begin{align}
  G_2\leq \bar{\mathsf{C}}_2\ell^{-\alpha/2}\big( \eta+(\log\ell)^6 \big)(\log\ell)^2,
  \label{Equ:tail-of-tau0-eta-j-l-2}
\end{align}
    for some $\bar{\mathsf{C}}_2\in(0,\infty)$.

\medskip
\noindent
\emph{Step (3): Bounding the first term.}
Recall that
\(
N^{(\eta,j,\ell)}_i := \calP^{(j)}_{\ell, i}\big([B^{(\eta,j,\ell)}_i, A^{(\eta,j,\ell)}_i]\big)
\). Then,
\begin{align*}
    G_1 &=\P\Big(\tau_{(\eta,j,\ell)}+1 > \ell^{\alpha},  \tau'^{(\lfloor \log \ell\rfloor^4)}_{(\eta,j,\ell)} +1\leq \ell^{\alpha}, \mathscr{E}_{\alpha}^{(\eta,j,\ell)}\Big)\\
    &\leq \P\Big(  \forall 1\leq m \leq \lfloor \log \ell\rfloor^4 , \tau'^{(m)}_{(\eta,j,\ell)} \neq \tau_{(\eta,j,\ell)}; A^{(\eta,j,\ell)}_{\tau'^{(m)} + 1}- B^{(\eta,j,\ell)}_{\tau'^{(m)} + 1} \leq (\log\ell)^2\Big)\\
    &\leq \P\Big( N^{(\eta,j,\ell)}_{\tau'^{(\lfloor \log \ell\rfloor^4)}_{(\eta,j,\ell)} +1}\neq 0; \forall 1\leq m \leq \lfloor \log \ell\rfloor^4 -1, \tau'^{(m)}_{(\eta,j,\ell)} \neq \tau_{(\eta,j,\ell)}; A^{(\eta,j,\ell)}_{\tau'^{(m)} + 1}- B^{(\eta,j,\ell)}_{\tau'^{(m)} + 1} \leq (\log\ell)^2\Big)\\
    &=\E\Big[\P\Big(N^{(\eta,j,\ell)}_{\tau'^{(\lfloor \log \ell\rfloor^4)}_{(\eta,j,\ell)} +1}\neq 0\big| \calF_{\tau'^{(\lfloor \log \ell\rfloor^4-1)}_{(\eta,j,\ell)} +1} \Big)\\
    &\qquad\qquad\ind{\forall 1\leq m \leq \lfloor \log \ell\rfloor^4 -1, \tau'^{(m)}_{(\eta,j,\ell)} \neq \tau_{(\eta,j,\ell)};  A^{(\eta,j,\ell)}_{\tau'^{(m)} + 1}- B^{(\eta,j,\ell)}_{\tau'^{(m)} + 1} \leq (\log\ell)^2}\Big]
    \displaybreak[0]\\
    &=\E\Big[\P\Big(\tau'_{(\eta',j',\ell)}
    \neq \tau_{(\eta',j',\ell)} \Big)\Big|_{\substack{\eta' = A^{(\eta,j,\ell)}_{\tau'^{(\lfloor \log \ell\rfloor^4-1)}+1}- B^{\eta,j,\ell}_{\tau'^{(\lfloor \log \ell\rfloor^4-1)} + 1},\,\, j' = \tau'^{(\lfloor \log \ell\rfloor^4-1)}_{(\eta,j,\ell)} +1}}\\
    &\qquad\qquad \ind{\forall 1\leq m \leq \lfloor \log \ell\rfloor^4 -1, \tau'^{(m)}_{(\eta,j,\ell)} \neq \tau_{(\eta,j,\ell)}; A^{(\eta,j,\ell)}_{\tau'^{(m)} + 1}- B^{(\eta,j,\ell)}_{\tau'^{(m)} + 1} \leq (\log\ell)^2}\Big] \displaybreak[0]\\
    &\leq \Big(1- \tfrac{\tilde C}{(2\log\log\ell)^2}\Big) \P\Big(\forall 1\leq m \leq \lfloor \log \ell\rfloor^4 -1, \tau'^{(m)}_{(\eta,j,\ell)} \neq \tau_{(\eta,j,\ell)}; A^{(\eta,j,\ell)}_{\tau'^{(m)} + 1}- B^{(\eta,j,\ell)}_{\tau'^{(m)} + 1} \leq (\log\ell)^2 \Big).
\end{align*}
Here, in the last inequality we used Lemma~\ref{Lem:tau_0_is_not_tau_1}. Iterating this inequality, we obtain
\begin{align}
    G_1 \leq \Big(1- \tfrac{\tilde C}{(2\log\log\ell)^2}\Big)^{\lfloor \log \ell\rfloor^4} < \e^{-\bar{\mathsf{C}}_1(\log\ell)^2 },\qquad \text{for some }\bar{\mathsf{C}}_1\in(0,\infty).
    \label{Equ:tail-of-tau0-eta-j-l-1}
\end{align}

Combining~\eqref{Equ:tail-of-tau0-eta-j-l},~\eqref{Equ:tail-of-tau0-eta-j-l-3},~\eqref{Equ:tail-of-tau0-eta-j-l-2}, and~\eqref{Equ:tail-of-tau0-eta-j-l-1}, and setting $\eta = (\log \ell)^2$, we conclude that there exists a constant $\bar{\mathsf{C}}_4 \in (0, \infty)$ such that the first term on the right-hand side of~\eqref{Equ:tail-bound-for-tau0-tau1} is bounded from above by
\[
\bar{\mathsf{C}}_4 \, \ell^{-\alpha/2} (\log \ell)^8.
\]
Furthermore, by Proposition~\ref{Prop:tau_1_vert_inhom_decay} and~\eqref{Equ:tail-of-tau0-eta-j-l-3}, there exists a constant $\bar{\mathsf{C}}_5 \in (0, \infty)$ such that the second term is bounded from above by
\[
\P(\tau'_{\ell} > \ell^{2\alpha + 2}) + \P\big(\mathscr{E}_{2\alpha+2}^{(\ell)\, \complement}\big) \leq \bar{\mathsf{C}}_5 \, \ell^{-\alpha}.
\]
Combining these two bounds yields the desired result.  
\end{proof}

%%%%%%%%%%%%%%%%%%%%%%%%%%%%%%%%%%%%%%
%%%%%%%%%%%%%%%%%%%%%%%%%%%%%%%%%%%%%%
%%%%%%%%%%%%%%%%%%%%%%%%%%%%%%%%%%%%%%

\subsection{Proofs for vertically inhomogeneous semi lattices}\label{Sec:Inhom}
In this section we consider the case where additionally the vertical direction may feature inhomogeneities. As mentioned above, the key challenge is to control the spatial dependencies arising from these inhomogeneities. In essence, we overcome these challenges by coupling the homogeneous and inhomogeneous models and using the asymptotic flattening in the intensity $\lambda$ and traffic generation function $\mu$. Again, to ease the exposition, we only consider the case where $\x=\o$, the general case follows then without effort. 

%%%%%%%%%%%%%%%%%%%%%%%%%%%%%%%%%%%%%%
%%%%%%%%%%%%%%%%%%%%%%%%%%%%%%%%%%%%%%
%%%%%%%%%%%%%%%%%%%%%%%%%%%%%%%%%%%%%%

\subsubsection{Bounding random walks on $\mathbb{S}$ with vertical inhomogeneities}\label{sec_prelim_coupling_vert_inhom}
For the vertically inhomogeneous model, we use bold sans-serif symbols for notation. As in the vertically homogeneous model, we set $\A_0^\ell := \ell/2$ and $\B_0^\ell := -\ell/2$, and define recursively, for each $i \geq 0$,
\begin{align*}
    \overline{\U}^\ell_i &:= \inf\{y > \A^\ell_i \colon y \in \boldsymbol{\calP}_{\ell,i} \}, &
    \overline{\L}^\ell_i &:= \inf\{y \geq \B^\ell_i \colon y \in \boldsymbol{\calP}_{\ell,i} \}, \\
    \underline{\U}^\ell_i &:= \sup\{y \leq \A^\ell_i \colon y \in \boldsymbol{\calP}_{\ell,i} \}, &
    \underline{\L}^\ell_i &:= \sup\{y < \B^\ell_i \colon y \in \boldsymbol{\calP}_{\ell,i} \}, \\
    \A^\ell_{i+1} &:= ( \overline{\U}^\ell_i +\underline{\U}^\ell_i )/2, &
    \B^\ell_{i+1} &:= ( \overline{\L}^\ell_i + \underline{\L}^\ell_i )/2.
\end{align*}
Again, we define
\[
    \overline{\X}^\ell_i := \overline{\U}^\ell_i- \A^\ell_i, \qquad \underline{\X}^\ell_i :=  \A^\ell_i - \underline{\U}^\ell_i, \qquad \overline{\Y}^\ell_i := \overline{\L}^\ell_i-  \B^\ell_i , \qquad    \underline{\Y}^\ell_i := \B^\ell_i - \underline{\L}^\ell_i,
\]
and observe that
\[
    \A^\ell_{i+1}-\A^\ell_i= (\overline{\X}^\ell_i-\underline{\X}^\ell_i)/2
    \qquad\text{ and }\qquad
    \B^\ell_{i+1}-\B^\ell_i= (\overline{\Y}^\ell_i-\underline{\Y}^\ell_i)/2.
\]

Note that the processes $\{\A^\ell_i\colon i\geq 0\}$ and $\{\B^\ell_i\colon i\geq 0\}$ are sums of dependent random variables. However, if $\{\boldsymbol{\calP}_{\ell,i}\colon i\geq0\}$ are vertically homogeneous, then the increments of $\{\A^\ell_i\colon i\geq 0\}$ and $\{\B^\ell_i\colon i\geq 0\}$ depend only on the index $i$, and not on the past trajectories of the processes. In this case, each of the two processes becomes a sum of independent random variables, which makes them significantly easier to analyze. To facilitate the study of $\{\A^\ell_i\colon i\geq 0\}$ and $\{\B^\ell_i\colon i\geq 0\}$ in the inhomogeneous setting, we compare them to their vertically homogeneous counterparts via a suitable coupling.

In the vertically homogeneous model, we use standard math fonts for the symbols, and the quantities $\calP_{\ell,i}$, $A^\ell_i$, $B^\ell_i$, $\overline{U}^\ell_i$, $\underline{U}^\ell_i$, $\overline{L}^\ell_i$, $\underline{L}^\ell_i$, $\overline{X}^\ell_i$, $\underline{X}^\ell_i$, $\overline{Y}^\ell_i$, and $\underline{Y}^\ell_i$ are defined analogously.

%%%%%%%%%%%%%%%%%%%%%%%%%%%%%%%%%%%%%%
%%%%%%%%%%%%%%%%%%%%%%%%%%%%%%%%%%%%%%
%%%%%%%%%%%%%%%%%%%%%%%%%%%%%%%%%%%%%%

\subsubsection{Coupling of vertically homogeneous and inhomogeneous models}
Let $\{E_{1,i}\colon i \geq 0\}$, $\{E_{2,i}\colon i \geq 0\}$, $\{E_{3,i}\colon i \geq 0\}$, and $\{E_{4,i}\colon i \geq 0\}$ be four independent sequences of i.i.d.\ $\Exp(\lambda_{\min})$ random variables. We use $E$ to denote a generic random variable with the $\Exp(\lambda_{\min})$ distribution. Then, we define
\[
\tilde\tau_\ell:= \inf\Big\{ i\geq 0 \colon \big(\A^\ell_i - \B^\ell_i \big)\wedge \big( A^\ell_i - B^\ell_i \big) < E_{2,i} +E_{3,i}  \Big\}.
\]
We  couple the processes $\{\boldsymbol{\calP}_{\ell, i}\colon i\geq 0\}$ and $\{\calP_{\ell, i}\colon i\geq 0\}$ as follows.
\begin{itemize}
    \item \textbf{\underline{Auxiliary processes.}} Let $\{\calQ_{j,i}\colon i \geq 0, j=1,2,\ldots, 7\}$ be seven sequences of independent Poisson point processes on $\R$, which are mutually independent and also independent of the sequences of exponential random variables $\{E_{j,i}\colon i \geq 0,  j=1,2,3,4\}$. For each $i \geq 0$, the corresponding point processes have the following intensities.
    \begin{alignat*}{2}
        \nu_{1,i}(s) &:= \lambda(i/\ell^2, s/\ell^2) \wedge \lambda(i/\ell^2, 0) - \lambda_{\min},& \\
        \nu_{2,i}(s) &:= \big( \lambda(i/\ell^2, s/\ell^2) - \lambda(i/\ell^2, 0) \big)_+, &\\
        \nu_{3,i}(s) &:= \big( \lambda(i/\ell^2, 0) - \lambda(i/\ell^2, s/\ell^2) \big)_+, &\\
        \nu_{4,i}(s) &:= \lambda(i/\ell^2, s/\ell^2), 
        &\nu_{5,i}(s) := \lambda(i/\ell^2, 0), \\
        \nu_{6,i}(s) &:= \lambda(i/\ell^2, s/\ell^2),
        &\nu_{7,i}(s) := \lambda(i/\ell^2, 0).
    \end{alignat*}%\vspace{.2cm}
    
    \item \textbf{\underline{Coupling for $i < \tilde\tau_\ell$: region of interaction.}} 
    We define the following intervals.
    \begin{align*}
        I_i^{\A} &:= [\A^\ell_i - E_{2,i},\ \A^\ell_i + E_{1,i}], &
        I_i^A &:= [A^\ell_i - E_{2,i},\ A^\ell_i + E_{1,i}], \\
        I_i^{\B} &:= [\B^\ell_i - E_{4,i},\ \B^\ell_i + E_{3,i}], &
        I_i^B &:= [B^\ell_i - E_{4,i},\ B^\ell_i + E_{3,i}].
    \end{align*}
    For $i<\tilde\tau_\ell$, we couple $\boldsymbol{\calP}_{\ell, i}$ on the set $I_i^{\A} \cup I_i^{\B}$ with $\calP_{\ell, i}$ on the set $I_i^{A} \cup I_i^{B}$ as follows.
        \begin{align*}
        \boldsymbol{\calP}_{\ell, i} \cap (I_i^{\A} \cup I_i^{\B} ) 
        & = \big((\calQ_{1,i} \cup \calQ_{2,i}) \cap (I_i^{\A} \cup I_i^{\B} ) \big)\cup \{\A^\ell_i - E_{2,i}, \A^\ell_i + E_{1,i}, \B^\ell_i - E_{4,i}, \B^\ell_i + E_{3,i}\},
    \end{align*}
    and
    \begin{align*}
        \calP_{\ell, i} \cap (I_i^{A} \cup I_i^{B}) 
        & = \Big(\big((\calQ_{1,i} \cup \calQ_{3,i}) \cap I_i^{\A}  \big) - (\A^\ell_i - A^\ell_i) \Big)\cup \Big(\big((\calQ_{1,i} \cup \calQ_{3,i}) \cap I_i^{\B} \big) - (\B^\ell_i - B^\ell_i) \Big) \\
        &\qquad \cup \{A^\ell_i - E_{2,i}, A^\ell_i + E_{1,i}, B^\ell_i - E_{4,i}, B^\ell_i + E_{3,i}\}.
    \end{align*}%\vspace{.2cm}
    
    \item \textbf{\underline{Coupling for $i < \tilde\tau_\ell$: outside region of interaction.}} For $i<\tilde\tau_\ell$, outside the interaction
regions defined above, the point processes are defined independently as
    \begin{align*}
        \boldsymbol{\calP}_{\ell, i} \cap (I_i^{\A} \cup I_i^{\B} )^{\complement}  = \calQ_{4,i} \cap (I_i^{\A} \cup I_i^{\B} )^{\complement},
        \qquad
        \text{ and }
        \qquad
        \calP_{\ell, i} \cap (I_i^{A} \cup I_i^{B})^{\complement}  = \calQ_{5,i} \cap (I_i^{A} \cup I_i^{B})^{\complement}.
    \end{align*}%\vspace{.2cm}
    
    \item \textbf{\underline{Coupling for $i \geq \tilde\tau_\ell$.}} For $i\geq \tilde\tau_\ell$, we simply set $\boldsymbol{\calP}_{\ell, i} = \calQ_{6,i}$ and $\calP_{\ell, i} =  \calQ_{7,i}$.
\end{itemize}%\vspace{.2cm}

\begin{remark}
\label{Rem:measurability-and-independence-of-tau_C}
 Note that, under this coupling, the random time $\tilde\tau_\ell$ is measurable with respect to the $\sigma$-algebra generated by the point processes $\{\calQ_{j,i}\colon i \geq 0,\, j = 1,\ldots,5\}$ and the exponential random variables $\{E_{j,i}\colon i \geq 0,\, j = 1,\ldots,4\}$. Moreover, these are independent of the point processes $\{\calQ_{j,i}\colon i \geq 0,\, j = 6,7\}$.
\end{remark}

\begin{remark}
\label{Extension-of-initial-coupling}
Also note that the initial coupling is only defined for $i < \tilde\tau_\ell$, since for $i \geq \tilde\tau_\ell$, the regions associated with $\{\A^\ell_i\colon i \geq \tilde\tau_\ell\}$ and $\{\B^\ell_i\colon i \geq \tilde\tau_\ell\}$, or with $\{A^\ell_i\colon i \geq \tilde\tau_\ell\}$ and $\{B^\ell_i\colon i \geq \tilde\tau_\ell\}$, may overlap. This overlap makes it impossible to simultaneously apply the initial coupling to both the regions near $\{\A^\ell_i\colon i \geq \tilde\tau_\ell\}$ and $\{A^\ell_i\colon i \geq \tilde\tau_\ell\}$ and those near $\{\B^\ell_i\colon i \geq \tilde\tau_\ell\}$ and $\{B^\ell_i\colon i \geq \tilde\tau_\ell\}$. However, when needed, the initial coupling can still be extended beyond $i \geq \tilde\tau_\ell$ to either the regions near $\{\A^\ell_i\colon i \geq \tilde\tau_\ell\}$ and $\{A^\ell_i\colon i \geq \tilde\tau_\ell\}$, or the regions near $\{\B^\ell_i\colon i \geq \tilde\tau_\ell\}$ and $\{B^\ell_i\colon i \geq \tilde\tau_\ell\}$, but not both simultaneously.  
\end{remark}

%%%%%%%%%%%%%%%%%%%%%%%%%%%%%%%%%%%%%%%%
%%%%%%%%%%%%%%%%%%%%%%%%%%%%%%%%%%%%%%%%
%%%%%%%%%%%%%%%%%%%%%%%%%%%%%%%%%%%%%%%%

\subsubsection{Proofs for times up to $\tilde\tau_\ell$}
\begin{prop}
\label{Prop:Error_upto_tauCoupling} 
Under the above coupling, for any $t > 0$, there exists a constant $\mathfrak{C}_t < \infty$ such that, for every $\epsilon > 0$, $\alpha \in \mathbb{R}$, and $\ell \geq 1$,
\[
\epsilon^2\ell^{2\alpha-1} 
\P\Big( \sup_{0\leq i\leq \tilde\tau_\ell} \big| \A^\ell_i- A^\ell_i \big| > \epsilon \ell^\alpha, \tilde\tau_\ell \leq \ell^2 t \Big)
\leq \mathfrak{C}_t,
\]
and
\[
\epsilon^2\ell^{2\alpha-1} 
\P\Big( \sup_{0\leq i\leq \tilde\tau_\ell} \big| \B^\ell_i- B^\ell_i \big| > \epsilon \ell^\alpha, \tilde\tau_\ell \leq \ell^2 t \Big)
\leq \mathfrak{C}_t.
\]
\end{prop} 
\begin{proof}[Proof of  Proposition~\ref{Prop:Error_upto_tauCoupling}]
We will prove the first equation, the second follows by a similar argument. To proceed, we extend  the initial coupling, originally defined only for $i < \tilde\tau_\ell$, beyond $\tilde\tau_\ell$ by coupling the regions around $\{\A^\ell_i\colon i \geq \tilde\tau_\ell\}$ with those around $\{A^\ell_i\colon i \geq \tilde\tau_\ell\}$, as justified by Remark~\ref{Extension-of-initial-coupling}.

We now define a non-negative submartingale
   \begin{align*}
        \Phi^\ell_i:= \big( \A^\ell_i- A^\ell_i \big)^2 +\sum_{j=1}^i \Big| \E\Big[ \big( \A^\ell_j- A^\ell_j \big)^2 - \big( \A^\ell_{j-1}- A^\ell_{j-1} \big)^2 | \calF_{j-1} \Big] \Big|,
    \end{align*}
    where $\calF_0$ is trivial and $\calF_i$ is the $\sigma$-algebra generated by $\{\boldsymbol{\calP}_{\ell, k}\colon 0\leq k \leq i-1\}$ and $\{\calP_{\ell, k}\colon 0\leq k \leq i-1\}$ that makes the random variables $\{\A^\ell_k\colon 0\leq k \leq i\}$ and $\{A^\ell_k\colon 0\leq k \leq i\}$ measurable. 
    Now, by applying Doob's maximal inequality, we obtain
\begin{align}
    \P\Big( \sup_{0\leq i\leq \lfloor \ell^2 t \rfloor } \big| \A^\ell_i- A^\ell_i \big| > \epsilon \ell^{\alpha} \Big)
    \leq \P\Big( \sup_{0\leq i\leq \lfloor \ell^2 t \rfloor} \Phi^\ell_i > \epsilon^2 \ell^{2\alpha} \Big)
    \leq \frac{1}{\epsilon^2 \ell^{2\alpha}} \E\big[ \Phi^\ell_{\lfloor \ell^2 t \rfloor} \big].
    \label{Equ:Error_upto_tauCoupling<ell^2t}
\end{align}
 To calculate this, we observe that
    \begin{align*}
       \E&\Big[ \big( \A^\ell_j- A^\ell_j \big)^2 - \big( \A^\ell_{j-1}- A^\ell_{j-1} \big)^2 | \calF_{j-1} \Big] \\
       &= 2\big( \A^\ell_{j-1}- A^\ell_{j-1} \big)
       \E\Big[ \big( \A^\ell_j- \A^\ell_{j-1} \big) - \big( A^\ell_j- A^\ell_{j-1} \big) | \calF_{j-1} \Big]+ \E\Big[ \Big(\big( \A^\ell_j- \A^\ell_{j-1} \big) - \big(  A^\ell_j- A^\ell_{j-1}\big) \Big)^2 | \calF_{j-1} \Big] \\
       &= 2\big( \A^\ell_{j-1}- A^\ell_{j-1} \big)
       \E\big[ \A^\ell_j- \A^\ell_{j-1}  |\calF_{j-1} \big]  + \E\Big[ \Big(\big( \A^\ell_j- \A^\ell_{j-1} \big) - \big( A^\ell_j- A^\ell_{j-1}\big) \Big)^2 | \calF_{j-1} \Big] .
    \end{align*}
   We now need the following statement. Its proof is provided after the proof of Proposition~\ref{Prop:Error_upto_tauCoupling}.
  \begin{lemma}
  \label{Lem:Bound_for_expected_increments}For any $i\geq 0$,
  \[
  \big|\E\big[ \A^\ell_{i+1}-\A^\ell_i| \calF_i\big]\big| \leq \Xi\ell^{-2} \lambda_{\min}^{-3}.
  \]
\end{lemma}
    Using Lemma~\ref{Lem:Bound_for_expected_increments}, we get that
    \begin{align}
       \E&\Big[\Big|\E\big[ \big( \A^\ell_j- A^\ell_j \big)^2 - \big( \A^\ell_{j-1}- A^\ell_{j-1} \big)^2 | \calF_{j-1} \big]\Big|\Big] \nonumber\\
       &\leq \frac{2 \Xi}{\ell^2\lambda_{\min}^3}\E\big[\big| \A^\ell_{j-1}- A^\ell_{j-1} \big|\big]  + \E\Big[ \Big(\big( \A^\ell_j- \A^\ell_{j-1} \big) - \big( A^\ell_j- A^\ell_{j-1}\big) \Big)^2  \Big] \nonumber\\
       &\leq \frac{2\Xi}{\ell^2\lambda_{\min}^3} \sum_{k=1}^{j-1}\E\big[\big| \big(\A^\ell_{k}- \A^\ell_{k-1}\big) - \big(A^\ell_{k}- A^\ell_{k-1}\big) \big|\big]  + \E\Big[ \Big(\big( \A^\ell_j- \A^\ell_{j-1} \big) - \big( A^\ell_j- A^\ell_{j-1}\big) \Big)^2  \Big].
       \label{Equ:Decompose_the_terms_in_expected_Martingale_error}
    \end{align}
Here, note that
\begin{align}
    2 \E\big[\big| \big(\A^\ell_{k+1}-& \A^\ell_{k}\big) - \big(A^\ell_{k+1}- A^\ell_{k}\big) \big|\big]
    = \E\big[\big| \big(\overline{\X}^\ell_{k}- \underline{\X}^\ell_{k}\big) - \big(\overline{X}^\ell_{k}- \underline{X}^\ell_{k}\big) \big|\big] \nonumber\\
    &\leq \E\big[\big| \overline{\X}^\ell_{k}-\overline{X}^\ell_{k}\big|\big]+ \E\big[\big| \underline{\X}^\ell_{k} -  \underline{X}^\ell_{k} \big|\big] \nonumber\\
    &=\E\big[\big( \overline{\X}^\ell_{k}-\overline{X}^\ell_{k}\big)_+\big]+\E\big[\big( \overline{X}^\ell_{k}-\overline{\X}^\ell_{k}\big)_+\big]+ \E\big[\big( \underline{\X}^\ell_{k} -  \underline{X}^\ell_{k} \big)_+\big]+\E\big[\big( \underline{X}^\ell_{k} -  \underline{\X}^\ell_{k} \big)_+\big],
    \label{Equ:Decompose_the_difference_of_increments}
\end{align}
where
\begin{align*}
    \E\big[ \big(\overline{\X}^\ell_{k} - \overline{X}^\ell_{k} \big)_+ \big]
    &= \int_{0}^{\infty}\P\big(\overline{\X}^\ell_{k} - \overline{X}^\ell_{k} >y\big)\,\dd y
    =  \int_{0}^{\infty}\P\big(\overline{X}^\ell_{k}< \overline{\X}^\ell_{k} -y\big)\,\dd y\\\displaybreak[0]
    &= \int_{0}^{\infty} \E\big[\P\big(\overline{X}^\ell_{k}< \overline{\X}^\ell_{k} -y |\overline{\X}^\ell_{k}\big)\big]\,\dd y
    = \E\Big[ \int_{0}^{\infty} \P\big(\overline{X}^\ell_{k}< \overline{\X}^\ell_{k} -y |\overline{\X}^\ell_{k}\big)\,\dd y \Big]\\ \displaybreak[0]
    &= \E\Big[ \int_{0}^{\overline{\X}^\ell_{k}} \P\big(\overline{X}^\ell_{k}< \overline{\X}^\ell_{k} -y | \overline{\X}^\ell_{k}\big)\,\dd y \Big]
    = \E\Big[ \int_{0}^{\overline{\X}^\ell_{k}} \P\big(\overline{X}^\ell_{k}<y |\overline{\X}^\ell_{k}\big)\,\dd y \Big]\\ \displaybreak[0]
   &= \E\Big[ \int_{0}^{\overline{\X}^\ell_{k}} \Big( 1- \exp\Big(-\int_{\A^\ell_{k}}^{\A^\ell_{k}+y}\Big( \lambda(i/\ell^2, 0)-\lambda(i/\ell^2,s/\ell^2)\Big)_+\,\dd s\Big)\Big)\,\dd y \Big]\\ \displaybreak[0]
  &\leq  \E\Big[ \int_{0}^{\overline{\X}^\ell_{k}}  \int_{0}^{y}\big( \lambda(i/\ell^2, 0)-\lambda(i/\ell^2,(\A^\ell_{k}+s)/\ell^2)\big)_+\,\dd s\,\dd y \Big]\\ \displaybreak[0]
  &\leq \frac{\Xi}{\ell^2} \E\Big[ \int_{0}^{\overline{\X}^\ell_{k}}  \int_{0}^{y}\big( |\A^\ell_{k}| +s\big)\,\dd s\,\dd y \Big]
   \leq \frac{\Xi}{\ell^2} \E\Big[ \int_{0}^{E_{1,k}}  \int_{0}^{y}\big( |\A^\ell_{k}| +s\big)\,\dd s\,\dd y \Big]\\ \displaybreak[0]
  &= \frac{\Xi}{\ell^2} \E\Big[ \int_{0}^{E_{1,k}}  \big( |\A^\ell_{k}|y +y^2/2\big)\,\dd y \Big]
   = \frac{\Xi}{\ell^2} \E\Big[  |\A^\ell_{k}|{E_{1,k}}^2/2 +{E_{1,k}}^3/6 \Big]\\\displaybreak[0]
   &=  \frac{\Xi}{\ell^2} \Big( \frac{1}{\lambda_{\min}^2}\E[|\A^\ell_{k}|] + \frac{1}{\lambda_{\min}^3} \Big).
\end{align*}
Observing that
\[
(\A^\ell_{k})^2 = \sum_{r=1}^k (\A^\ell_{r}-\A^\ell_{r-1})^2+ 2\sum_{r=1}^k \A^\ell_{r-1}(\A^\ell_{r}-\A^\ell_{r-1}) +{(\A^\ell_{0})}^2,
\]
and using Lemma~\ref{Lem:Bound_for_expected_increments}, we get that
\begin{align}
     \E\big[|\A^\ell_{k}|\big]^2 \leq \E\big[{(\A^\ell_{k})}^2 \big]
    &=  \sum_{r=1}^k \E[(\A^\ell_{r}-\A^\ell_{r-1})^2] + 2\sum_{r=1}^k \E\big[\A^\ell_{r-1}(\A^\ell_{r}-\A^\ell_{r-1})\big] +\frac{\ell^2}{4} \nonumber\\
    &\leq  \sum_{r=1}^k \E\big[(\A^\ell_{r}-\A^\ell_{r-1})^2\big] + \frac{2\Xi}{\ell^2\lambda_{\min}^3}\sum_{r=1}^k \E\big[|\A^\ell_{r-1}|\big] +\frac{\ell^2}{4} \nonumber\\\displaybreak[0]
    &\leq k \E[E^2] + \frac{2\Xi}{\ell^2\lambda_{\min}^3}\sum_{r=1}^k \Big( \frac{\ell}{2} + (r-1)\E[E] \Big) +\frac{\ell^2}{4} \nonumber\\
    \displaybreak[0]
    &= \frac{2k}{\lambda_{\min}^2}+ \frac{k\Xi}{\ell \lambda_{\min}^3}+ \frac{k(k-1)\Xi}{\ell^2\lambda_{\min}^3} +\frac{\ell^2}{4} \leq C_t\ell^2,
    \label{Equ:Bound_for_Alk^2}
\end{align}
for all $k\leq \lfloor \ell^2t \rfloor$. Here,
\[
C_t:= \frac{2t}{\lambda_{\min}^2}+ \frac{t\Xi}{\lambda_{\min}^3}+ \frac{t^2\Xi}{\lambda_{\min}^3} +\frac{1}{4}.
\]
Therefore, we have that, for all $k\leq \lfloor \ell^2t \rfloor$,
\[
 \E\Big[ \big(\overline{\X}^\ell_{k} - \overline{X}^\ell_{k} \big)_+ \Big] \leq C'_t/\ell,\qquad\text{ where }
C'_t:= \Xi\big(\sqrt{C_t}/\lambda_{\min}^2 +1/\lambda_{\min}^3\big).
\]
By applying similar calculations to the other terms on the right-hand side of~\eqref{Equ:Decompose_the_difference_of_increments}, we obtain that for any $k \leq \lfloor \ell^2t \rfloor$,
\begin{align}
     \E&\big[\big| \big(\A^\ell_{k+1}- \A^\ell_{k}\big) - \big(A^\ell_{k+1}- A^\ell_{k}\big) \big|\big] \leq 2 C'_t /\ell.
    \label{Equ:Decompose_the_difference_of_increments-2}
\end{align}
Similar to~\eqref{Equ:Decompose_the_difference_of_increments}, the second term on the right-hand side of~\eqref{Equ:Decompose_the_terms_in_expected_Martingale_error} can be decomposed as
\begin{align}
     \E&\Big[ \Big(\big( \A^\ell_j- \A^\ell_{j-1} \big) - \big( A^\ell_j- A^\ell_{j-1}\big) \Big)^2  \Big]
    =\frac{1}{4} \E\Big[\Big( \big(\overline{\X}^\ell_{j-1}- \underline{\X}^\ell_{j-1}\big) - \big(\overline{X}^\ell_{j-1}- \underline{X}^\ell_{j-1}\big) \Big)^2\Big] \nonumber\\
    &\leq \frac{1}{2}\E\Big[\big( \overline{\X}^\ell_{j-1}-\overline{X}^\ell_{j-1}\big)^2\Big]+ \frac{1}{2}\E\Big[\big( \underline{\X}^\ell_{j-1} -  \underline{X}^\ell_{j-1} \big)^2\Big] \nonumber\\
    &\leq \E\Big[\big( \overline{\X}^\ell_{j-1}-\overline{X}^\ell_{j-1}\big)_+^2\Big]+\E\Big[\big( \overline{X}^\ell_{j-1}-\overline{\X}^\ell_{j-1}\big)_+^2\Big]+ \E\Big[\big( \underline{\X}^\ell_{j-1} -  \underline{X}^\ell_{j-1} \big)_+^2\Big]+\E\Big[\big( \underline{X}^\ell_{j-1} -  \underline{\X}^\ell_{j-1} \big)_+^2\Big].
    \label{Equ:Decompose_the_difference_of_square_of_increments}
\end{align}
Calculations similar to those done earlier give us that, for all $j \leq \lfloor \ell^2t \rfloor$,
\[
    \E\Big[\big( \overline{\X}^\ell_{j-1}-\overline{X}^\ell_{j-1}\big)_+^2\Big] 
    = \int_0^\infty 2y\P\big(\overline{\X}^\ell_{j-1}-\overline{X}^\ell_{j-1} >y\big)\,\dd y
    \le   \frac{\Xi}{\ell^2} \Big( \frac{4}{\lambda_{\min}^3}\E\big[|\A^\ell_{j-1}|\big] + \frac{6}{\lambda_{\min}^4} \Big) \leq \frac{C''_t}{\ell},
\]
where 
\(
C''_t:= \Xi\big(4\sqrt{C_t}/\lambda_{\min}^3 +6/\lambda_{\min}^4 \big)
\).
By applying similar calculations to the other terms on the right-hand side of~\eqref{Equ:Decompose_the_difference_of_square_of_increments}, we obtain that, for all $j \leq \lfloor \ell^2t \rfloor$,
\begin{align}
     \E&\Big[ \Big(\big( \A^\ell_j- \A^\ell_{j-1} \big) - \big( A^\ell_j- A^\ell_{j-1}\big) \Big)^2  \Big] \leq 4 C''_t /\ell.
    \label{Equ:Decompose_the_difference_of_square_of_increments-2}
\end{align}
Combining~\eqref{Equ:Decompose_the_difference_of_increments-2} and~\eqref{Equ:Decompose_the_difference_of_square_of_increments-2}  together with~\eqref{Equ:Decompose_the_terms_in_expected_Martingale_error},
we get that for all $j\leq \ell^2 t$,
\begin{align}
     \E\Big[\Big|\E\Big[ \big( \A^\ell_j- A^\ell_j \big)^2 - \big( \A^\ell_{j-1}- A^\ell_{j-1} \big)^2 |\calF_{j-1} \Big]\Big|\Big] \leq C'''_t/\ell,
     \label{Equ:Bound_for_the_terms_in_expected_Martingale_error}
\end{align}
    where
    \(
    C'''_t:= 4C'_t t\Xi\lambda_{\min}^{-3} +4C''_t
    \).
Therefore, we get that
\begin{align*}
&\E\big[ \Phi^\ell_{\lfloor \ell^2 t \rfloor} \big]
= \E\Big[\big( \A^\ell_{\lfloor \ell^2t \rfloor}- A^\ell_{\lfloor \ell^2t \rfloor} \big)^2 \Big]
 +\sum_{j=1}^{\lfloor \ell^2t \rfloor} \E\Big[\Big| \E\Big[ \big( \A^\ell_j- A^\ell_j \big)^2 - \big( \A^\ell_{j-1}- A^\ell_{j-1} \big)^2 |\calF_{j-1} \Big] \Big|\Big]\\
&= \E\Big[ \sum_{j=1}^{\lfloor \ell^2t \rfloor}\E\Big[\big( \A^\ell_{j}- A^\ell_{j} \big)^2- \big( \A^\ell_{j-1}- A^\ell_{j-1} \big)^2 |\calF_{j-1}\Big]  \Big]+\sum_{j=1}^{\lfloor \ell^2t \rfloor} \E\Big[\Big| \E\Big[ \big( \A^\ell_j- A^\ell_j \big)^2 - \big( \A^\ell_{j-1}- A^\ell_{j-1} \big)^2 | \calF_{j-1} \Big] \Big|\Big]\\ \displaybreak[0]
&\leq 2\sum_{j=1}^{\lfloor \ell^2t \rfloor} \E\Big[\Big| \E\Big[ \big( \A^\ell_j- A^\ell_j \big)^2 - \big( \A^\ell_{j-1}- A^\ell_{j-1} \big)^2| \calF_{j-1} \Big] \Big|\Big]
\leq 2C'''_t\ell t.
\end{align*}
Combining this with~\eqref{Equ:Error_upto_tauCoupling<ell^2t} yields the desired result with $\mathfrak C_t:= 2C'''_tt$.
\end{proof}

\begin{proof}[Proof of Lemma~\ref{Lem:Bound_for_expected_increments}]
   Note that
     \begin{align*}
         \E\big[ \A^\ell_{i+1} -\A^\ell_i | \calF_i\big]
         = \E\Big[ \frac{1}{2}\big(\overline{\X}^\ell_i -\underline{\X}^\ell_i\big)  | \calF_i\Big]
        =\frac{1}{2}\big( \E\big[ \overline{\X}^\ell_i | \calF_i\big] - \E\big[  \underline{\X}^\ell_i | \calF_i\big] \big).
     \end{align*}
     Now,
     \begin{align*}
         \P\big( \overline{\X}^\ell_i  >x | \calF_i\big) = \exp\Big(-\int_{\A^\ell_i}^{\A^\ell_i+x}\lambda(i/{\ell^2},{s}/{\ell^2})\,\dd s\Big)
         \quad
         \text{and}
         \quad
         \P\big(   \underline{\X}^\ell_i >x | \calF_i\big) = \exp\Big(-\int_{\A^\ell_i-x}^{\A^\ell_i}\lambda({i}/{\ell^2},{s}/{\ell^2})\,\dd s\Big).
     \end{align*}
     Therefore, we have
     \begin{align*}
         2\big| \E\big[ \A^\ell_{i+1} -\A^\ell_i | \calF_i\big] \big|
        &= \Big|
         \int_{0}^{\infty} \Big( \e^{-\int_{\A^\ell_i}^{\A^\ell_i+x}\lambda(i/\ell^2,s/\ell^2)\,\dd s}-  \e^{-\int_{\A^\ell_i-x}^{\A^\ell_i}\lambda(i/\ell^2,s/\ell^2)\,\dd s} \Big)\,\dd x \Big|\\
         &\leq 
         \int_{0}^{\infty} \Big| \e^{-\int_{\A^\ell_i}^{\A^\ell_i+x}\lambda(i/\ell^2,s/\ell^2)\,\dd s}-  \e^{-\int_{\A^\ell_i-x}^{\A^\ell_i}\lambda(i/\ell^2,s/\ell^2)\,\dd s}  \Big| \,\dd x.
     \end{align*}
     Observing that 
     \[
     \min\Big( \int_{\A^\ell_i}^{\A^\ell_i+x}\lambda(i/\ell^2,s/\ell^2)\,\dd s, \int_{\A^\ell_i-x}^{\A^\ell_i}\lambda(i/\ell^2,s/\ell^2)\,\dd s\Big) \geq x\lambda_{\min},
     \]
  and
     \begin{align*}
         \Big| \int_{\A^\ell_i}^{\A^\ell_i+x}\lambda(i/\ell^2,s/\ell^2)\,\dd s - \int_{\A^\ell_i-x}^{\A^\ell_i}\lambda(i/\ell^2,s/\ell^2)\,\dd s\Big|
        & = \Big| \int_{0}^{x}\lambda\big(i/\ell^2,(\A^\ell_i+s)/{\ell^2}\big)\,\dd s - \int_{0}^{x}\lambda\big(i/\ell^2,(\A^\ell_i-s)/{\ell^2}\big)\,\dd s\Big|\\
        & \leq \int_{0}^{x} \Big| \lambda\big(i/\ell^2,(\A^\ell_i+s)/{\ell^2}\big)-\lambda\big(i/\ell^2,(\A^\ell_i-s)/{\ell^2}\big)\Big|\,\dd s,
     \end{align*}
 and using the inequality $|\e^{-a}-\e^{-b}|\leq |a-b|\e^{-(a\wedge b)}$, for any $a,b>0$, we get that
 \begin{align*}
     2\big| \E\big[ \A^\ell_{i+1} -\A^\ell_i | \calF_i\big] \big|
        &\leq \int_{0}^{\infty} \int_{0}^{x}
        \Big| \lambda\big(i/\ell^2,(\A^\ell_i+s)/{\ell^2}\big)-\lambda\big(i/\ell^2,(\A^\ell_i-s)/{\ell^2}\big)\Big|\e^{-x\lambda_{\min}}\,\dd s\,\dd x\\
        &= \int_{0}^{\infty} \int_{s}^{\infty}
        \Big| \lambda\big(i/\ell^2,(\A^\ell_i+s)/{\ell^2}\big)-\lambda\big(i/\ell^2,(\A^\ell_i-s)/{\ell^2}\big)\Big|\e^{-x\lambda_{\min}}\,\dd x \,\dd s\\
        &= \frac{1}{\lambda_{\min}}\int_{0}^{\infty} \Big| \lambda\big(i/\ell^2,(\A^\ell_i+s)/{\ell^2}\big)-\lambda\big(i/\ell^2,(\A^\ell_i-s)/{\ell^2}\big)\Big|\e^{-s\lambda_{\min}} \,\dd s.
 \end{align*}
 Now, using $\Xi$-Lipschitz continuity, we obtain that
 \begin{align*}
  2\big| \E\big[ \A^\ell_{i+1} -\A^\ell_i | \calF_i\big] \big|
  \leq 
  \frac{\Xi}{\lambda_{\min}}\int_{0}^{\infty} \frac{2s}{\ell^2}\e^{-s\lambda_{\min}} \,\dd s
  = \frac{2\Xi}{\ell^2\lambda_{\min}^3}.
 \end{align*}
This proves the desired result.
\end{proof}

To approximate the total traffic prior to $\tilde\tau_\ell$, we establish the following result.

\begin{prop}
For any $t > 0$, there exists a constant $\mathsf{C}_t < \infty$ such that for every $\epsilon > 0$,
\[
\limsup_{\ell\to\infty} \epsilon\ell \P\Big(   \sup_{0\leq k \leq \lfloor \ell^2 t \rfloor } \frac{1}{\ell^3} 
    \Big| \sum_{i = 0}^{k} \int_{\B^{\ell}_i}^{\A^{\ell}_i}  \mu(i/\ell^2, s/\ell^2) \, \boldsymbol{\calP}_{\ell, i} (\dd s) 
    - \sum_{i = 0}^{k} (\A^{\ell}_i- \B^{\ell}_i) \lambda(i/\ell^2,0) \mu(i/\ell^2,0)  \Big| >\epsilon\Big) < \mathsf C_t.
\]
\label{Prop:Error_in_area_approximation}
\end{prop}
\begin{proof}[Proof of Proposition~\ref{Prop:Error_in_area_approximation}]
    Observe that
    \begin{align}
          \sup_{0\leq k \leq \lfloor \ell^2 t \rfloor } \frac{1}{\ell^3} 
    \Big| \sum_{i = 0}^{k} \int_{\B^{\ell}_i}^{\A^{\ell}_i} \mu(i/\ell^2, s/\ell^2) \, \boldsymbol{\calP}_{\ell, i} (\dd s) 
    - \sum_{i = 0}^{k} (\A^{\ell}_i- \B^{\ell}_i) \lambda(i/\ell^2,0) \mu(i/\ell^2,0)  \Big|
    \leq \calE^{(\ell)}_{1,t}+ \calE^{(\ell)}_{2,t}+ \calE^{(\ell)}_{3,t},
    \label{Equ:Steps-docomposing_error_terms}
    \end{align}
    where
\begin{align*}
\calE^{(\ell)}_{1,t}&:= \sup_{0\leq k \leq \lfloor \ell^2 t \rfloor } \frac{1}{\ell^3}
\Big| \sum_{i = 0}^{k} \int_{\B^{\ell}_i}^{\A^{\ell}_i}  \mu(i/\ell^2, s/\ell^2) \, \boldsymbol{\calP}_{\ell, i}(\dd s) 
    - \sum_{i = 0}^{k} \mu(i/\ell^2,0)\boldsymbol{\calP}_{\ell, i}([\B^{\ell}_i, \A^{\ell}_i]) \Big|,
\\
\calE^{(\ell)}_{2,t}&:= \sup_{0\leq k \leq \lfloor \ell^2 t \rfloor } \frac{1}{\ell^3}
\Big|  \sum_{i = 0}^{k} \mu(i/\ell^2,0)  \boldsymbol{\calP}_{\ell, i}([\B^{\ell}_i, \A^{\ell}_i] )
- \sum_{i = 0}^{k} \mu(i/\ell^2,0) \int_{\B^{\ell}_i}^{\A^{\ell}_i}  \lambda(i/\ell^2, s/\ell^2) \,\dd s \Big|, \text{ and}
\\
\calE^{(\ell)}_{3,t}&:= \sup_{0\leq k \leq \lfloor \ell^2 t \rfloor } \frac{1}{\ell^3}
\Big|\sum_{i = 0}^{k} \mu(i/\ell^2,0) \int_{\B^{\ell}_i}^{\A^{\ell}_i}  \lambda(i/\ell^2, s/\ell^2) \,\dd s
        - \sum_{i = 0}^{k}  \Big(\A^{\ell}_i - \B^{\ell}_i  \Big) \lambda(i/\ell^2,0) \mu(i/\ell^2, 0)  \Big|.
\end{align*}
We establish the desired result in several steps.

\medskip
\noindent
\emph{Step (1): first term}.
Using the Lipschitz continuity of $\mu$ in its second argument, we have
\begin{align*}
\calE^{(\ell)}_{1,t}
&= \sup_{0\leq k \leq \lfloor \ell^2 t \rfloor } \frac{1}{\ell^3}
\Big| \sum_{i = 0}^{k} \int_{\B^{\ell}_i}^{\A^{\ell}_i}  \big( \mu(i/\ell^2, s/\ell^2) - \mu(i/\ell^2,0) \big) \, \boldsymbol{\calP}_{\ell, i}(\dd s)  \Big|
\\
&\leq \sup_{0\leq k \leq \lfloor \ell^2 t \rfloor } \frac{1}{\ell^3}
\sum_{i = 0}^{k} \int_{\B^{\ell}_i}^{\A^{\ell}_i}  \big| \mu(i/\ell^2, s/\ell^2) - \mu(i/\ell^2,0)\big| \, \boldsymbol{\calP}_{\ell, i}(\dd s)
\\
&\leq \sup_{0\leq k \leq \lfloor \ell^2 t \rfloor } \frac{\Xi}{\ell^5}
\sum_{i = 0}^{k} \int_{\B^{\ell}_i}^{\A^{\ell}_i}  |s| \, \boldsymbol{\calP}_{\ell, i}(\dd s)
=    \frac{\Xi}{\ell^5}
\sum_{i = 0}^{\lfloor \ell^2 t \rfloor} \int_{\B^{\ell}_i}^{\A^{\ell}_i}  |s| \, \boldsymbol{\calP}_{\ell, i}(\dd s).
\end{align*}
Hence, for all $\ell \geq 1$,
\begin{align*}
 \ell \E\big[\calE^{(\ell)}_{1,t}\big]
&\leq  \frac{\Xi}{\ell^4}
\sum_{i = 0}^{\lfloor \ell^2 t \rfloor} \E\Big[ \int_{\B^{\ell}_i}^{\A^{\ell}_i}  |s| \, \boldsymbol{\calP}_{\ell, i}(\dd s) \Big]= \frac{\Xi}{\ell^4}
\sum_{i = 0}^{\lfloor \ell^2 t \rfloor} \E\Big[\E\Big[ \int_{\B^{\ell}_i}^{\A^{\ell}_i}  |s| \, \boldsymbol{\calP}_{\ell, i}(\dd s) | \calF_i \Big]\Big] \nonumber\\
&= \frac{\Xi}{\ell^4}
\sum_{i = 0}^{\lfloor \ell^2 t \rfloor} \E\Big[ \int_{\B^{\ell}_i}^{\A^{\ell}_i}  |s| \lambda(i/\ell^2, s/\ell^2)\, \dd s \Big]\le  \frac{\Xi \lambda_{\max}}{2\ell^4}
\sum_{i = 0}^{\lfloor \ell^2 t \rfloor} \E\big[ {(\A^{\ell}_i)}^2 + {(\B^{\ell}_i)}^2 \big] \nonumber\\
&\leq \frac{\Xi \lambda_{\max}}{2\ell^4} \lfloor \ell^2 t \rfloor 2C_t\ell^2
\leq \Xi \lambda_{\max}C_t t.
\end{align*}
In the second-to-last inequality, we used the bound for $\E[ (\A^{\ell}_i)^2 ]$ as given in \eqref{Equ:Bound_for_Alk^2}, and a similar argument yields the same bound for $\E[ (\B^{\ell}_i)^2 ]$. The above inequality, when combined with Markov's inequality, implies that, for all $\ell \geq 1$,
\begin{align}
    \epsilon\ell \P\big(\calE^{(\ell)}_{1,t} >\epsilon/3 \big) \leq 3\Xi \lambda_{\max}C_t t.
    \label{Equ:Step-1_bound_for_error_prob}
\end{align}

\medskip
\noindent
\emph{Step (2): second term}.
Let us define
\[
\Psi^{\ell}_k:=  \sum_{i = 0}^{k-1} \mu(i/\ell^2,0) \boldsymbol{\calP}_{\ell, i}([\B^{\ell}_i, \A^{\ell}_i])
- \sum_{i = 0}^{k-1} \mu(i/\ell^2,0) \int_{\B^{\ell}_i}^{\A^{\ell}_i}   \lambda(i/\ell^2, s/\ell^2) \,\dd s, 
\]
and observe that
\begin{align*}
\E\big[ \Psi^{\ell}_{k+1}- \Psi^{\ell}_k | \calF_k \big] 
= \mu(k/\ell^2,0) \E\Big[  \boldsymbol{\calP}_{\ell, k}([\B^{\ell}_k, \A^{\ell}_k])
-   \int_{\B^{\ell}_k}^{\A^{\ell}_k}   \lambda(k/\ell^2, s/\ell^2) \,\dd s | \calF_k \Big] = 0.
\end{align*}
This implies that $\{\Psi^{\ell}_k\colon k \geq 0\}$ is a martingale, and consequently, $\big\{\big(\Psi^{\ell}_k\big)^2\colon k \geq 0\}$ is a submartingale. Therefore, using Doob's maximal inequality, we obtain that
\begin{align}
\P\big(\calE^{(\ell)}_{2,t} >\epsilon/3 \big)
= \P\Big( \sup_{0\leq k \leq \lfloor \ell^2 t \rfloor } \big| \Psi^{\ell}_{k+1} \big| >\epsilon \ell^3/3 \Big)
\leq \frac{9}{\epsilon^2\ell^6} \E\Big[ \big(\Psi^{\ell}_{\lfloor \ell^2 t \rfloor +1}\big)^2 \Big].
\label{Equ:Step-2_Doob_bound}
\end{align}
Now,
\begin{align*}
\E\Big[ \big(\Psi^{\ell}_{k+1}\big)^2- \big(\Psi^{\ell}_k \big)^2 | \calF_k \Big]
&=\E\Big[ 2\Psi^{\ell}_k\big(\Psi^{\ell}_{k+1}-\Psi^{\ell}_k \big)+ \big(\Psi^{\ell}_{k+1}-\Psi^{\ell}_k \big)^2 | \calF_k \Big]=\E\Big[ \big(\Psi^{\ell}_{k+1}-\Psi^{\ell}_k \big)^2 | \calF_k \Big]\\
&=\mu(k/\ell^2,0)^2\E\Big[ \big(  \boldsymbol{\calP}_{\ell, k}([\B^{\ell}_k, \A^{\ell}_k])
-   \int_{\B^{\ell}_k}^{\A^{\ell}_k}   \lambda(k/\ell^2, s/\ell^2) \,\dd s \big)^2 | \calF_k \Big].
\end{align*}
Since the conditional distribution of the number of points $\boldsymbol{\calP}_{\ell, k}([\B^{\ell}_k, \A^{\ell}_k])$ given $\calF_k$ is Poisson with parameter $\int_{\B^{\ell}_k}^{\A^{\ell}_k}   \lambda(k/\ell^2, s/\ell^2) \,\dd s$, we obtain that
\begin{align*}
\E\Big[ \big(\Psi^{\ell}_{k+1}\big)^2- \big(\Psi^{\ell}_k \big)^2 | \calF_k \Big]
&=\mu(k/\ell^2,0)^2 \int_{\B^{\ell}_k}^{\A^{\ell}_k}   \lambda(k/\ell^2, s/\ell^2) \,\dd s,
\end{align*}
which implies that
\begin{align*}
\E\Big[ \big(\Psi^{\ell}_{k+1}\big)^2- \big(\Psi^{\ell}_k \big)^2  \Big]
&=\mu(k/\ell^2,0)^2 \E\Big[\int_{\B^{\ell}_k}^{\A^{\ell}_k}   \lambda(k/\ell^2, s/\ell^2) \,\dd s \Big]\leq \lambda_{\max} \mu(k/\ell^2,0)^2 \E\Big[\A^{\ell}_k-\B^{\ell}_k \Big]\\
&\leq \lambda_{\max} \mu(k/\ell^2,0)^2 2\sqrt{C_t}\ell.
\end{align*}
As in Step (1), we have used the bounds for $\E[(\A^{\ell}_k)^2 ]$ and $\E[(\B^{\ell}_k)^2 ]$, specifically, the bound for $\E[(\A^{\ell}_k)^2 ]$ as given in~\eqref{Equ:Bound_for_Alk^2}, and a similar argument yields the same bound for $\E[(\B^{\ell}_k)^2 ]$. Now, incorporating the above inequality into~\eqref{Equ:Step-2_Doob_bound} gives us
\begin{align}
\limsup_{\ell\to\infty}\epsilon^2\ell^3 \P\big(\calE^{(\ell)}_{2,t} >\epsilon/3 \big)
&\leq \limsup_{\ell\to\infty} 18\sqrt{C_t}\lambda_{\max} \Big(\frac{1}{\ell^2}\sum_{i = 0}^{\lfloor \ell^2 t \rfloor}\mu(i/\ell^2,0)^2 \Big)= 18\sqrt{C_t}\lambda_{\max} \int_0^t \mu(x,0)^2\, \dd x.
\label{Equ:Step-2_bound_for_error_prob}
\end{align}

\medskip
\noindent
\emph{Step (3): third term}.
Using the Lipschitz continuity of $\lambda$ in its second argument, we obtain
\begin{align*}
\calE^{(\ell)}_{3,t} 
&= \sup_{0\leq k \leq \lfloor \ell^2 t \rfloor } \frac{1}{\ell^3}
\Big|\sum_{i = 0}^{k} \mu(i/\ell^2,0) \int_{\B^{\ell}_i}^{\A^{\ell}_i}  \big(\lambda(i/\ell^2, s/\ell^2) - \lambda(i/\ell^2,0) \big) \,\dd s 
          \Big|\\
& \leq \sup_{0\leq k \leq \lfloor \ell^2 t \rfloor } \frac{1}{\ell^3}
\sum_{i = 0}^{k} \mu(i/\ell^2,0) \int_{\B^{\ell}_i}^{\A^{\ell}_i}  \big|\lambda(i/\ell^2, s/\ell^2) - \lambda(i/\ell^2,0) \big| \,\dd s \\
& \leq \sup_{0\leq k \leq \lfloor \ell^2 t \rfloor } \frac{\Xi}{\ell^5}
\sum_{i = 0}^{k} \mu(i/\ell^2,0) \int_{\B^{\ell}_i}^{\A^{\ell}_i}  |s| \,\dd s \\
& \leq \sup_{0\leq k \leq \lfloor \ell^2 t \rfloor } \frac{\Xi}{2\ell^5}
\sum_{i = 0}^{k} \mu(i/\ell^2,0) \big((\A^{\ell}_i)^2+ (\B^{\ell}_i)^2 \big)  =  \frac{\Xi}{2\ell^5}
\sum_{i = 0}^{\lfloor \ell^2 t \rfloor} \mu(i/\ell^2,0) \big( (\A^{\ell}_i)^2+ (\B^{\ell}_i)^2 \big).
\end{align*}
Therefore, using the bounds for $\E[(\A^{\ell}_k)^2 ]$ and $\E[(\B^{\ell}_k)^2 ]$, we get
\begin{align*}
\limsup_{\ell\to\infty} \ell \E\big[\calE^{(\ell)}_{3,t}\big] 
&\leq \limsup_{\ell\to\infty} \frac{\Xi}{2\ell^4}
\sum_{i = 0}^{\lfloor \ell^2 t \rfloor} \mu(i/\ell^2,0) \E\big[ (\A^{\ell}_i)^2+ (\B^{\ell}_i)^2 \big] \leq \limsup_{\ell\to\infty} \frac{\Xi}{\ell^4}
\sum_{i = 0}^{\lfloor \ell^2 t \rfloor} \mu(i/\ell^2,0)C_t\ell^2  \nonumber\\
&= \limsup_{\ell\to\infty} \Xi C_t
\Big(\frac{1}{\ell^2}\sum_{i = 0}^{\lfloor \ell^2 t \rfloor} \mu(i/\ell^2,0) \Big)=\Xi C_t\int_0^t \mu(x, 0)\,\dd x .
\end{align*}
This, when combined with Markov's inequality, leads to
\begin{align}
\limsup_{\ell\to\infty}\epsilon\ell \P\big(\calE^{(\ell)}_{3,t} >\epsilon/3 \big) \leq 3\Xi C_t\int_0^t \mu(x, 0)\,\dd x.
\label{Equ:Step-3_bound_for_error_prob}
\end{align}
Now, combining~\eqref{Equ:Steps-docomposing_error_terms},~\eqref{Equ:Step-1_bound_for_error_prob},~\eqref{Equ:Step-2_bound_for_error_prob} and~\eqref{Equ:Step-3_bound_for_error_prob} yields the result with
\[
\mathsf C_t:= 3\Xi \lambda_{\max}C_t t + 3\Xi C_t\int_0^t \mu(x, 0)\,\dd x,
\]
as desired.
\end{proof}

We now state the following result, which is needed to prove the subsequent proposition.
\begin{lemma}
Under our assumption
\[
  \Big(  \frac{\tilde\tau_\ell}{\ell^2}, \frac{1}{\ell^3} 
    \sum_{i = 0}^{\tilde\tau_\ell -1} (A^{\ell}_i- B^{\ell}_i) \lambda(i/\ell^2,0) \mu(i/\ell^2,0)  
     \Big)\xlongrightarrow{d}
\Big( \vartheta, \int_{0}^{\vartheta} (1+M_x)\lambda(x,0)\mu(x,0)\,\dd x \Big), \qquad\text{ as } \ell\to\infty.
\]
\label{Lem:convergence-in-dist-upto-tauC-vertically-homogeneous}
\end{lemma}
\begin{proof}[Proof of Lemma~\ref{Lem:convergence-in-dist-upto-tauC-vertically-homogeneous}]
An argument similar to that in the proof of Proposition~\ref{Prop:Conv-in-dist-upto-tau1-vert-homo}, combined with Proposition~\ref{Prop:Error_upto_tauCoupling}, gives us the required result.
\end{proof}

\begin{prop}
Under our assumption
\[
\Big(  \frac{\tilde\tau_\ell}{\ell^2}, \frac{1}{\ell^3} 
    \sum_{i = 0}^{\tilde\tau_\ell -1 } \int_{\B^{\ell}_i}^{\A^{\ell}_i} \mu(i/\ell^2, s/\ell^2) \, \boldsymbol{\calP}_i(\dd s) 
     \Big)\xlongrightarrow{d}
\Big( \vartheta, \int_{0}^{\vartheta} (1+M_x)\lambda(x,0)\mu(x,0)\,\dd x \Big), \qquad\text{ as } \ell\to\infty.
\]
\label{Prop:convergence-in-dist-upto-tauC}
\end{prop}
\begin{proof}[Proof of Proposition~\ref{Prop:convergence-in-dist-upto-tauC}]
From Proposition~\ref{Prop:Error_in_area_approximation}, we have that for any $t,\epsilon>0$,
\begin{align}
\lim_{\ell\to\infty} \P  \Big(  \frac{\tilde\tau_\ell}{\ell^2} \leq t, \frac{1}{\ell^3} \Big| 
    \sum_{i = 0}^{\tilde\tau_\ell -1 } \int_{\B^{\ell}_i}^{\A^{\ell}_i}  \mu(i/\ell^2, s/\ell^2) \, \boldsymbol{\calP}_{\ell, i}(\dd s)  - \sum_{i = 0}^{\tilde\tau_\ell -1} (\A^{\ell}_i- \B^{\ell}_i) \lambda(i/\ell^2,0) \mu(i/\ell^2,0)  \Big|>\epsilon  \Big)
=0.
\label{Equ:Error_in_area_approximation-for-convergence-in-dist-upto-tauC}
\end{align}
Also observe that, on the event
\[
\Big\{  \tilde\tau_\ell \leq \ell^2 t, \sup_{0\leq i\leq \tilde\tau_\ell} \big| \A^\ell_i- A^\ell_i \big| \leq \epsilon \ell^\alpha, \sup_{0\leq i\leq \tilde\tau_\ell} \big| \B^\ell_i- B^\ell_i \big| \leq \epsilon \ell^\alpha \Big\},
\]
we have that
\begin{align*}
    &\Big|\sum_{i = 0}^{\tilde\tau_\ell -1}  (\A^{\ell}_i- \B^{\ell}_i) \lambda(i/\ell^2,0) \mu(i/\ell^2,0) - \sum_{i = 0}^{\tilde\tau_\ell -1} (A^{\ell}_i- B^{\ell}_i) \lambda(i/\ell^2,0) \mu(i/\ell^2,0) \Big|\\
&\leq \sum_{i = 0}^{\tilde\tau_\ell -1} \big( \big|\A^{\ell}_i - A^{\ell}_i \big|+\big|\B^{\ell}_i - B^{\ell}_i\big| \big) \lambda(i/\ell^2,0) \mu(i/\ell^2,0)\\
& \leq 2\epsilon \ell^\alpha \sum_{i = 0}^{\tilde\tau_\ell -1} \lambda(i/\ell^2,0) \mu(i/\ell^2,0)\leq 2\epsilon \ell^{(2+\alpha)}\Big(\frac{1}{\ell^2} \sum_{i = 0}^{\ell^2t} \lambda(i/\ell^2,0) \mu(i/\ell^2,0) \Big) 
\leq \epsilon\ell^3,
\end{align*}
for $\alpha\in(1/2,1)$ and all large $\ell$. Therefore, by Proposition~\ref{Prop:Error_upto_tauCoupling}, for any $\alpha\in(1/2,1)$ and any $t,\epsilon>0$,
\begin{align}
   \lim_{\ell\to\infty} \P  \Big(  \frac{\tilde\tau_\ell}{\ell^2} \leq t, \frac{1}{\ell^3} \Big|\sum_{i = 0}^{\tilde\tau_\ell -1} (\A^{\ell}_i- \B^{\ell}_i) \lambda(i/\ell^2,0) \mu(i/\ell^2,0) - \sum_{i = 0}^{\tilde\tau_\ell -1} (A^{\ell}_i- B^{\ell}_i) \lambda(i/\ell^2,0) \mu(i/\ell^2,0)  \Big|>\epsilon  \Big) =0.
   \label{Error_upto_tauCoupling-for-convergence-in-dist-upto-tauC}
\end{align}
Hence, combining~\eqref{Equ:Error_in_area_approximation-for-convergence-in-dist-upto-tauC} and~\eqref{Error_upto_tauCoupling-for-convergence-in-dist-upto-tauC} with Lemma~\ref{Lem:convergence-in-dist-upto-tauC-vertically-homogeneous} yields the desired result.
\end{proof}

%%%%%%%%%%%%%%%%%%%%%%%%%%%%%%%%%%%%%%%%%%%%
%%%%%%%%%%%%%%%%%%%%%%%%%%%%%%%%%%%%%%%%%%%%
%%%%%%%%%%%%%%%%%%%%%%%%%%%%%%%%%%%%%%%%%%%%

\subsubsection{Proofs for times after $\tilde\tau_\ell$}
After \(\tilde\tau_\ell\), we are in the following typical situation.
\begin{lemma}\label{Lem:starting_situation_after_tau_C}
For all \(j \in \{0, 1\}\) and \(\alpha \in \big(1/2, 2/3\big) \), we have that
    \begin{align*}
        \P\Big( \A^{\ell}_{\tilde\tau_\ell + j} - \B^{\ell}_{\tilde\tau_\ell + j} 
        \geq \ell^{\alpha} \Big) 
        \rightarrow
        0, \qquad\text{ as } \ell\to\infty.
    \end{align*} 
\end{lemma}
\begin{proof}
For $j=0$, the result follows from Proposition~\ref{Prop:Error_upto_tauCoupling} and the tightness of $\{\tilde{\tau}_\ell/\ell^2\colon \ell>0\}$, as implied by Proposition~\ref{Prop:convergence-in-dist-upto-tauC}. The case $j=1$ is implied by the additional bound that, for all $x, \ell \geq 1$,
\begin{align}
\P\Big(\max_{0 \leq i \leq \lfloor x\ell^2 \rfloor -1}  \overline{\X}^\ell_i \vee \underline{\X}^\ell_i \vee \overline{\Y}^\ell_i \vee \underline{\Y}^\ell_i > (\log \ell)^2 \Big)
\leq
4\lfloor x\ell^2 \rfloor \e^{-\lambda_{\min} (\log \ell)^2},
    \label{Equ:sup_of_A_B_bound}
\end{align}
which tends to zero as $\ell \to \infty$.
\end{proof}

We will now employ the following coupling to control the behavior in the vertically inhomogeneous case after \(\tilde\tau_\ell\).

Let \(\{\widetilde{\mathcal{Q}}^{(a, b, j, \ell)}_{1, i}\colon i \geq 0\}, \{\widetilde{\mathcal{Q}}^{(a, b, j, \ell)}_{2, i}\colon i \geq 0\}, \{\widetilde{\mathcal{Q}}^{(a, b, j, \ell)}_{3, i}\colon i \geq 0\}\) be independent Poisson point processes, also independent from all constructions from before, with intensities
\begin{align*}
    \widetilde{\nu}^{(a, b, j, \ell)}_{1, i}(s) 
    &:= \lambda((j+i)/\ell^2, s/\ell^2) \land \lambda((j+i)/\ell^2, (a+b)/(2\ell^2)), \\
    \widetilde{\nu}^{(a, b, j, \ell)}_{2, i}(s) 
    &:= \big(\lambda((j+i)/\ell^2, s/\ell^2) - \lambda((j+i)/\ell^2, (a+b)/(2\ell^2))\big)_+, \text{ and} \\
    \widetilde{\nu}^{(a, b, j, \ell)}_{3, i}(s) 
    &:= \big(\lambda((j+i)/\ell^2, (a+b)/(2\ell^2)) - \lambda((j+i)/\ell^2, s/\ell^2)\big)_+ .
\end{align*}
We couple the processes \(\{\widetilde{\boldsymbol{\calP}}^{(a, b, j, \ell)}_i\colon i\geq 0\}\) and \(\{\widetilde{\calP}^{(a, b, j, \ell)}_i\colon i\geq 0\}\) as follows
\begin{align*}
    \widetilde{\boldsymbol{\calP}}^{(a, b, j, \ell)}_i &:= \widetilde{\mathcal{Q}}^{(a, b, j, \ell)}_{1, i} \cup \widetilde{\mathcal{Q}}^{(a, b, j, \ell)}_{2, i}\qquad\text{ and }\qquad
    \widetilde{\calP}^{(a, b, j, \ell)}_i := \widetilde{\mathcal{Q}}^{(a, b, j, \ell)}_{1, i} \cup \widetilde{\mathcal{Q}}^{(a, b, j, \ell)}_{3, i}
\end{align*} and define the random variables $\A^{(a, b, j, \ell)}_i$, $\B^{(a, b, j, \ell)}_i$, $\overline{\X}^{(a, b, j, \ell)}_i$, $\underline{\X}^{(a, b, j, \ell)}_i$, $\overline{\Y}^{(a, b, j, \ell)}_i$, and $\underline{\Y}^{(a, b, j, \ell)}_i$ as well as their vertically homogeneous counterparts $A^{(a, b, j, \ell)}_i$, $B^{(a, b, j, \ell)}_i$, $\overline{X}^{(a, b, j, \ell)}_i$, $\underline{X}^{(a, b, j, \ell)}_i$, $\overline{Y}^{(a, b, j, \ell)}_i$, and $\underline{Y}^{(a, b, j, \ell)}_i$ in this context analogous to Section~\ref{sec_prelim_coupling_vert_inhom}. 
Note that we start the vertically homogeneous and inhomogeneous processes both at heights \(\A_0^{(a, b, j, \ell)} := a, A_0^{(a, b, j, \ell)} := a\) and \(\B_0^{(a, b, j, \ell)} := b, B_0^{(a, b, j, \ell)} := b\).
Moreover we define $A^{(a, b, j, \ell)}_{1, i}$, $B^{(a, b, j, \ell)}_{1, i}$, $\overline{X}^{(a, b, j, \ell)}_{1, i}$, $\underline{X}^{(a, b, j, \ell)}_{1, i}$, $\overline{Y}^{(a, b, j, \ell)}_{1, i}$, and $\underline{Y}^{(a, b, j, \ell)}_{1, i}$ analogously by letting \(\{\widetilde{\mathcal{Q}}^{(a, b, j, \ell)}_{1, i}\colon i \geq 0\}\) play the role of \(\{\widetilde{\boldsymbol{\calP}}^{(a, b, j, \ell)}_i\colon i\geq 0\}\), respectively \(\{\widetilde{\calP}^{(a, b, j, \ell)}_i\colon i\geq 0\}\).

This coupling will be used via the following result.
\begin{lemma}\label{le_coupling_after_tau_C_success_probability}
    Let \(\alpha \in (1/2, 2/3)\). Then, for all \(1 \leq a-b \leq \ell^{\alpha}\) (uniformly in \(a, b, j\)), we have that
    \begin{align*}
        \P\big(\boldsymbol{\mathsf \tau}_{(a, b, j, \ell)} \neq \tau_{(a, b, j, \ell)}\big) 
       \rightarrow
        0, \qquad\text{ as } \ell\to\infty.
    \end{align*} 
\end{lemma}
\begin{proof}
    In order to see closeness, we observe that \(\{\widetilde{\mathcal{Q}}^{(a, b, j, \ell)}_{2, i}\colon i \geq 0\}, \{\widetilde{\mathcal{Q}}^{(a, b, j, \ell)}_{3, i}\colon i \geq 0\}\) will not present any contributing point with high enough probability. For this, note first that
\begin{align}\label{Equ:sup_is_controlled_by_CLT_behaviour_for_A1i}
        &\P\Big(\sup_{i \leq t\ell^{2\alpha}} (|A^{(a, b, j, \ell)}_{1, i} - A^{(a, b, j, \ell)}_{1, 0}|\lor|B^{(a, b, j, \ell)}_{1, i} - B^{(a, b, j, \ell)}_{1, 0}|) \geq C \ell^{\alpha}\Big) \nonumber\\
        &\leq 
        \P\Big(\sup_{i \leq t\ell^{2\alpha}} \vert A^{(a, b, j, \ell)}_{1, i} - A^{(a, b, j, \ell)}_{i} \vert \geq (C/4) \ell^{\alpha}\Big) 
        + \P\Big(\sup_{i \leq t\ell^{2\alpha}} \vert B^{(a, b, j, \ell)}_{1, i} - B^{(a, b, j, \ell)}_{i} \vert \geq (C/4) \ell^{\alpha}\Big) \nonumber\\
        &+ \P\Big(\sup_{i \leq t\ell^{2\alpha}} (\vert A^{(a, b, j, \ell)}_{i} - A^{(a, b, j, \ell)}_{0} \vert \geq (C/4) \ell^{\alpha}\Big) 
        + \P\Big(\sup_{i \leq t\ell^{2\alpha}} (\vert B^{(a, b, j, \ell)}_{i} - B^{(a, b, j, \ell)}_{0} \vert \geq (C/4) \ell^{\alpha}\Big)\leq \frac{C_t}{C}
    \end{align}
    by an analogous bound to~\eqref{Equ:Error_upto_tauCoupling<ell^2t} and by Doob's martingale inequality, with some \(t\)-dependent constant \(0 < C_t < \infty\).
    Further, we have
    \begin{align*}
        &\mathbb{P}\Big(\big(\widetilde{\mathcal{Q}}^{(a, b, j, \ell)}_{2, i} \cup \widetilde{\mathcal{Q}}^{(a, b, j, \ell)}_{3, i}\big)\Big( \big[A^{(a, b, j, \ell)}_{1, i} - \underline{\X}^{(a, b, j, \ell)}_{1, i}, A^{(a, b, j, \ell)}_{1, i}\big] \cup \big[A^{(a, b, j, \ell)}_{1, i}, A^{(a, b, j, \ell)}_{1, i} + \overline{\X}^{(a, b, j, \ell)}_{1, i}\big] \\
        &\qquad \cup  \big[B^{(a, b, j, \ell)}_{1, i} - \underline{\Y}^{(a, b, j, \ell)}_{1, i}, B^{(a, b, j, \ell)}_{1, i}\big] \cup \big[B^{(a, b, j, \ell)}_{1, i}, B^{(a, b, j, \ell)}_{1, i} + \overline{\Y}^{(a, b, j, \ell)}_{1, i}\big] \Big) = \emptyset\,\, \forall i \leq t\ell^{2 \alpha}, \\
        &\qquad \qquad\sup_{i \leq t\ell^{2\alpha}} \big(|A^{(a, b, j, \ell)}_{1, i} - A^{(a, b, j, \ell)}_{1, 0}|\lor|B^{(a, b, j, \ell)}_{1, i} - B^{(a, b, j, \ell)}_{1, 0}|\big) \leq C \ell^{\alpha}, \\
        &\qquad \qquad\sup_{i \leq t\ell^{2\alpha}} \big(\underline{\X}^{(a, b, j, \ell)}_{1, i} \lor \overline{\X}^{(a, b, j, \ell)}_{1, i} \lor \underline{\Y}_{1, i} \lor \overline{\Y}^{(a, b, j, \ell)}_{1, i}\big) \leq (\log\ell)^2\Big) \\
        &= \mathbb{E}\Big[\prod_{i = 0}^{t \ell^{2\alpha}} p_i(A_{1, i}, B_{1, i},  \underline{\X}_{1, i}, \overline{\X}_{1, i}, \underline{\Y}_{1, i}, \overline{\Y}_{1, i})
        \mathbbm{1}\Big\{\sup_{i \leq t\ell^{2\alpha}} \big(|A^{(a, b, j, \ell)}_{1, i} - A^{(a, b, j, \ell)}_{1, 0}|\lor|B^{(a, b, j, \ell)}_{1, i} - B^{(a, b, j, \ell)}_{1, 0}|\big) \leq C \ell^{\alpha}\Big\} \\
        &\qquad\qquad \mathbbm{1}\Big\{\sup_{i \leq t\ell^{2\alpha}} \big(\underline{\X}^{(a, b, j, \ell)}_{1, i} \lor \overline{\X}^{(a, b, j, \ell)}_{1, i} \lor \underline{\Y}_{1, i} \lor \overline{\Y}^{(a, b, j, \ell)}_{1, i}\big) \leq (\log\ell)^2\Big\} \Big],
    \end{align*} with
    \begin{align*}
        p_i\big(\widetilde{a}, \widetilde{b}, \underline{\x}, \overline{\x}, \underline{\y}, \overline{\y}\big)
        &= \mathbb{P}\Big(\big(\widetilde{\mathcal{Q}}^{(a, b, j, \ell)}_{2, i} \cup \widetilde{\mathcal{Q}}^{(a, b, j, \ell)}_{3, i}\big)\big( [\widetilde{a} - \underline{\x}, \widetilde{a}] \cup ([\widetilde{a}, \widetilde{a} + \overline{\x}]) \cup  [\widetilde{b} - \underline{\y}, \widetilde{b}] \cup ([\widetilde{b}, \widetilde{b} + \overline{\y}]) \big) = \emptyset\Big) \\
        &\geq 1 - \frac{\Xi}{\ell^2} \Big(\Big|\widetilde{a} - \frac{a+b}{2}\Big|(\underline{\x} + \overline{\x}) + \Big|\widetilde{b} - \frac{a+b}{2}\Big|(\underline{\y} + \overline{\y}) +  \frac{\underline{\x}^2 + \overline{\x}^2 + \underline{\y}^2 + \overline{\y}^2}{2}  \Big),
    \end{align*} 
    since
    \begin{align*}
        \mathbb{P}\Big(\big(\widetilde{\mathcal{Q}}^{(a, b, j, \ell)}_{2, i} \cup \widetilde{\mathcal{Q}}^{(a, b, j, \ell)}_{3, i}\big)\big( [\widetilde{a} - \underline{\x}, \widetilde{a}] \big) = \emptyset\Big)
        &= \exp\Big(- \int_{\widetilde{a} - \underline{\x}}^{\widetilde{a}} \Big|\lambda\Big(\frac{j+i}{\ell^2}, \frac{s}{\ell^2}\Big) - \lambda\Big(\frac{j+i}{\ell^2}, \frac{a+b}{2\ell^2}\Big)\Big|\mathrm{d}s \Big)  \\
        &\geq \exp\Big(-\frac{\Xi}{\ell^2} \Big(\Big|\widetilde{a} - \frac{a+b}{2}\Big|\underline{\x} + \frac{\underline{\x}^2}{2} \Big) \Big)
    \end{align*} 
    and analogous bounds. It follows that
    \begin{align*}
        \mathbb{P}\big(\boldsymbol{\mathsf{\tau}}_{(a, b, j, \ell)}& = \tau_{(a, b, j, \ell)}\big)
        \\
        &\geq  \mathbb{P}\Big(\big(\widetilde{\mathcal{Q}}^{(a, b, j, \ell)}_{2, i} \cup \widetilde{\mathcal{Q}}^{(a, b, j, \ell)}_{3, i}\big)\big( \big[A^{(a, b, j, \ell)}_{1, i} - \underline{\X}^{(a, b, j, \ell)}_{1, i}, A^{(a, b, j, \ell)}_{1, i}\big] \cup \big[A^{(a, b, j, \ell)}_{1, i}, A^{(a, b, j, \ell)}_{1, i} + \overline{\X}^{(a, b, j, \ell)}_{1, i}\big] \\
        &\qquad\qquad \cup  \big[B^{(a, b, j, \ell)}_{1, i} - \underline{\Y}^{(a, b, j, \ell)}_{1, i}, B^{(a, b, j, \ell)}_{1, i}\big] \cup \big[B^{(a, b, j, \ell)}_{1, i}, B^{(a, b, j, \ell)}_{1, i} + \overline{\Y}^{(a, b, j, \ell)}_{1, i}\big] \big) = \emptyset\, \, \forall i \leq t\ell^{2 \alpha}, \\
        &\qquad\qquad  \sup_{i \leq t\ell^{2\alpha}} \big(\big|A^{(a, b, j, \ell)}_{1, i} - A^{(a, b, j, \ell)}_{1, 0}\big|\lor\big|B^{(a, b, j, \ell)}_{1, i} - B^{(a, b, j, \ell)}_{1, 0}\big|\big) \leq C \ell^{\alpha}, \\
        &\qquad\qquad \sup_{i \leq t\ell^{2\alpha}} \big(\underline{\X}^{(a, b, j, \ell)}_{1, i} \lor \overline{\X}^{(a, b, j, \ell)}_{1, i} \lor \underline{\Y}_{1, i} \lor \overline{\Y}^{(a, b, j, \ell)}_{1, i}\big) \leq (\log\ell)^2\Big) \\
        &\qquad - \P\big(\tau_{(a, b, j, \ell)} > t\ell^{2 \alpha}\big)\\
        &\geq K_{\ell, t, C}  \big(1 -C/\ell^{2-\alpha - \epsilon}\big)^{t \ell^{2 \alpha}}- \P\big(\tau_{(a, b, j, \ell)} > t\ell^{2\alpha}\big),
    \end{align*} 
    for \(\epsilon \in \big(0,2/3 - \alpha \big)\) and 
    \begin{align*}
        K_{\ell, t, C}
        &:=
        \P\Big(\sup_{i \leq t\ell^{2\alpha}} \big(\big|A^{(a, b, j, \ell)}_{1, i} - A^{(a, b, j, \ell)}_{1, 0}\big|\lor\big|B^{(a, b, j, \ell)}_{1, i} - B^{(a, b, j, \ell)}_{1, 0}\big|\big) \leq C \ell^{\alpha}\Big) \\
        &\qquad - \P\Big(\sup_{i \leq t\ell^{2\alpha}} \big(\underline{\X}^{(a, b, j, \ell)}_{1, i} \lor \overline{\X}^{(a, b, j, \ell)}_{1, i} \lor \underline{\Y}^{(a, b, j, \ell)}_{1, i} \lor \overline{\Y}^{(a, b, j, \ell)}_{1, i}\big) > (\log\ell)^2 \Big).
    \end{align*}
    Hence,
    \begin{align*}
        \lim_{\ell \to \infty} \mathbb{P}\big(\boldsymbol{\mathsf{\tau}}_{(a, b, j, \ell)} = \tau_{(a, b, j, \ell)}\big)
        &\geq \limsup_{t \to \infty} \limsup_{C \to \infty} \limsup_{\ell \to \infty}\big\{  K_{\ell, t, C}  \big(1 - C/\ell^{2-\alpha - \epsilon}\big)^{t \ell^{2 \alpha}}- \P\big(\tau_{(a, b, j, \ell)} > t\ell^{2\alpha}\big) \big\} \\
        &\geq \limsup_{t \to \infty} \limsup_{C \to \infty} \limsup_{\ell \to \infty}\big\{ 1 - C_t/C-  \P\big(\tau_{(a, b, j, \ell)} > t \ell^{2\alpha}\big) \big\}= 1.
    \end{align*} For the second inequality we used \eqref{Equ:sup_is_controlled_by_CLT_behaviour_for_A1i} and an easy bound analogous to~\eqref{Equ:sup_of_A_B_bound}; the last equality uses tightness of \(\big\{\tau_{(a, b, j, \ell)}/\ell^{2 \alpha}\colon 1\le \ell, 1 \leq a-b \leq \ell^{\alpha}\big\}\), which follows by the exact same argument as we used for Theorem~\ref{Thm:Poisson:asymp-taul}, using the analogues of Propositions~\ref{Prop:tau_1_vert_inhom_decay} and 
\ref{Prop:tau0-tau1-decay-rate-vertical-homo-semi} in the current setting.
\end{proof}

\begin{lemma}\label{le:tau_C_vs_tau_0_inhom}
    Let \(\alpha \in(1/2,2/3)\), Then, for all \(g \colon [0, \infty) \to [0, \infty)\) with \(\ell^{2\alpha}/g(\ell) \rightarrow 0\), as $\ell \to \infty$, we have that
    \begin{align*}
        \frac{1}{g(\ell)}\big|\tilde\tau_\ell - \boldsymbol{\mathsf \tau}_{\ell} \big|
        \xlongrightarrow{p} 0,\qquad\text{as }\ell\to\infty. 
    \end{align*} 
\end{lemma}
\begin{proof}[Proof of Lemma~\ref{le:tau_C_vs_tau_0_inhom}]
    We have \(\tilde\tau_\ell \leq \boldsymbol{\mathsf \tau}_\ell\).
    On the other hand,
    \begin{align*}
        \P\big( \boldsymbol{\mathsf \tau}_{\ell} &- \tilde\tau_\ell > g(\ell) \big)\\
        &\leq
        \mathbb{E}\Big[\P\big(\boldsymbol{\mathsf \tau}_{\ell} - \tilde\tau_\ell > g(\ell) | \mathcal{F}_{\tilde\tau_\ell + 1} \big) \mathbbm{1}\big\{\A^{\ell}_{\tilde\tau_\ell + 1} - \B^{\ell}_{\tilde\tau_\ell + 1} 
        < \ell^{\alpha} \big\} \Big] + \P\big( \A^{\ell}_{\tilde\tau_\ell + 1} - \B^{\ell}_{\tilde\tau_\ell + 1} 
        \geq \ell^{\alpha} \big),
    \end{align*} 
    where the second summand on the right-hand side goes to zero, by Lemma~\ref{Lem:starting_situation_after_tau_C}, and for the first summand it holds too that
    \begin{align*}
        &\mathbb{E}\Big[\P\big(\boldsymbol{\mathsf \tau}_{\ell} - \tilde\tau_\ell > g(\ell) | \mathcal{F}_{\tilde\tau_\ell + 1} \big) \mathbbm{1}\big\{\A^{\ell}_{\tilde\tau_\ell + 1} - \B^{\ell}_{\tilde\tau_\ell + 1} 
        < \ell^{\alpha} \big\} \Big]  \\
        &= \mathbb{E}\Big[\P\big(\boldsymbol{\mathsf \tau}_{(a, b, j, \ell)}> g(\ell) - 1\big)_{a = \A^{\ell}_{\tilde\tau_\ell + 1}, b = \B^{\ell}_{\tilde\tau_\ell + 1} , j = \tilde\tau_\ell + 1}  \mathbbm{1}\big\{\A^{\ell}_{\tilde\tau_\ell + 1} - \B^{\ell}_{\tilde\tau_\ell + 1} 
        < \ell^{\alpha} \big\} \Big] \\
        &\leq \mathbb{E}\Big[\Big( \P\big(\tau_{(a, b, j, \ell)}> g(\ell) - 1\big) + \P\big(\boldsymbol{\mathsf \tau}_{(a, b, j, \ell)} \neq \tau_{(a, b, j, \ell)}\big) \Big)_{a = \A^{\ell}_{\tilde\tau_\ell + 1}, b = \B^{\ell}_{\tilde\tau_\ell + 1} , j = \tilde\tau_\ell + 1} \mathbbm{1}\big\{\A^{\ell}_{\tilde\tau_\ell + 1} - \B^{\ell}_{\tilde\tau_\ell + 1} < \ell^\alpha\big\} \Big],
    \end{align*} 
    which converges to zero as $\ell \to \infty$, due to Proposition~\ref{Prop:tau_1_vert_inhom_decay} and Lemma~\ref{le_coupling_after_tau_C_success_probability}.
\end{proof}

\begin{lemma}\label{le:small} We have that
    \begin{align*}
        \frac{1}{\ell^3} \big|\tilde\T_{\ell} - \T_{\ell} \big|
        \xlongrightarrow{p} 0,\qquad\text{as }\ell\to\infty.
    \end{align*}
\end{lemma}
\begin{proof}[Proof of Lemma~\ref{le:small}]
    Let \(\alpha \in(1/2, 2/3)\) and also set \(g(\ell) = \ell^{2\alpha}\).
    We have
    \begin{align*}
        \P\big(\big|\tilde\T_{\ell} - \T_{\ell} \big\vert > \epsilon \ell^3 \big)
        &\leq \P\big(\big|\tilde\T_{\ell} - \T_{\ell} \big\vert > \epsilon \ell^3, \vert\tilde\tau_\ell - \boldsymbol{\mathsf \tau}_{\ell}\vert \leq g(\ell), \A^{\ell}_{\tilde\tau_\ell + 1} - \B^{\ell}_{\tilde\tau_\ell + 1} 
        < \ell^{\alpha}   \big) \\
        &\qquad + \P(\big| \tilde\tau_\ell - \boldsymbol{\mathsf \tau}_{\ell}\big\vert > g(\ell)) + \P(\A^{\ell}_{\tilde\tau_\ell + 1} - \B^{\ell}_{\tilde\tau_\ell + 1} 
        \geq \ell^{\alpha})
    \end{align*} and the second and third summand on the right-hand side go to zero by Lemma~\ref{le:tau_C_vs_tau_0_inhom} and Lemma~\ref{Lem:starting_situation_after_tau_C}.
    But
    \begin{align*}
        &\P\big(\big|\tilde\T_{\ell} - \T_{\ell} \big\vert > 2\epsilon \ell^3, \big| \tilde\tau_\ell - \boldsymbol{\mathsf \tau}_{\ell}\big\vert \leq g(\ell), \A^{\ell}_{\tilde\tau_\ell + 1} - \B^{\ell}_{\tilde\tau_\ell + 1} 
        < \ell^{\alpha}   \big) \\
        &\leq \P\Big(\sum_{i = \tilde\tau_\ell + 1}^{\tilde\tau_\ell + g(\ell)} \int_{\B^{\ell}_i}^{\A^{\ell}_i}  \mu_{\max} \, \boldsymbol{\calP}_{\ell,i} (\dd s)  > \epsilon \ell^3, \A^{\ell}_{\tilde\tau_\ell + 1} - \B^{\ell}_{\tilde\tau_\ell + 1} 
        < \ell^{\alpha} \Big) + \P\Big(\int_{\B^{\ell}_{\tilde\tau_\ell}}^{\A^{\ell}_{\tilde\tau_\ell}}  \mu_{\max} \, \boldsymbol{\calP}_{\ell,\tilde\tau_\ell} (\dd s)  > \epsilon \ell^3 \Big),
    \end{align*} 
    where
    \begin{align*}
    \P\Big(\int_{\B^{\ell}_{\tilde\tau_\ell}}^{\A^{\ell}_{\tilde\tau_\ell}}  \mu_{\max} \, \boldsymbol{\calP}_{\ell, \tilde\tau_\ell} (\dd s)  > \epsilon \ell^3 \Big)
    &= \mathbb{E}\Big[ \P\Big(\int_{b}^{a}  \mu_{\max} \, {\calQ}_{i, 6} (\dd s)  > \epsilon \ell^3 \Big)_{a = \A^{\ell}_{\tilde\tau_\ell}, b = \B^{\ell}_{\tilde\tau_\ell}, i = \tilde\tau_\ell} \Big] \\
    &\leq \mathbb{E}\Big[\tfrac{\lambda_{\max} \mu_{\max} (\A^{\ell}_{\tilde\tau_\ell} - \B^{\ell}_{\tilde\tau_\ell})}{\epsilon \ell^3} \land 1 \Big]\leq \frac{\lambda_{\max} \mu_{\max} \ell^\alpha}{\epsilon \ell^3} + \P(\A^{\ell}_{\tilde\tau_\ell} - \B^{\ell}_{\tilde\tau_\ell} 
    > \ell^\alpha),
    \end{align*}
    which tends to zero as $\ell\to\infty$,
    by Lemma~\ref{Lem:starting_situation_after_tau_C} and
    \begin{align*}
        &\mathbb{E}\Big[\P\Big(\sum_{i = \tilde\tau_\ell + 1}^{\tilde\tau_\ell + g(\ell)} \int_{\B^{\ell}_i}^{\A^{\ell}_i}  \mu(i/\ell^2, s/\ell^2) \, \boldsymbol{\calP}_{\ell, i} (\dd s)  > \epsilon \ell^3 | \mathcal{F}_{\tilde\tau_\ell + 1}\Big) \mathbbm{1}\big\{\A^{\ell}_{\tilde\tau_\ell + 1} - \B^{\ell}_{\tilde\tau_\ell + 1} < \ell^{\alpha}\big\} \Big] \\
        &= \mathbb{E}\Big[\P\Big(\sum_{i = 0}^{g(\ell)-1} \int_{\B^{(a, b, j, \ell)}_i}^{\A^{(a, b, j, \ell)}_i}  \mu(i/\ell^2, s/\ell^2) \, \boldsymbol{\calP}_{\ell, i} (\dd s)  > \epsilon \ell^3\Big)_{a = \A^{\ell}_{\tilde\tau_\ell + 1}, b = \B^{\ell}_{\tilde\tau_\ell + 1}, j = \tilde\tau_\ell + 1} \mathbbm{1}\big\{\A^{\ell}_{\tilde\tau_\ell + 1} - \B^{\ell}_{\tilde\tau_\ell + 1} < \ell^{\alpha}\big\} \Big] \\
        &\leq \mathbb{E}\Big[ \mathbbm{1}\big\{\A^{\ell}_{\tilde\tau_\ell + 1} - \B^{\ell}_{\tilde\tau_\ell + 1} < \ell^{\alpha}\big\}\Big( \P\Big(\sum_{i = 0}^{g(\ell)-1} \int_{b - g(\ell)(\log\ell)^2}^{a + g(\ell)(\log\ell)^2 }  \mu_{\max} \, \boldsymbol{\calP}_{\ell, i} (\dd s)  > \epsilon \ell^3\Big) \\
        &\quad + \P\Big(\max_{1 \leq i \leq g(\ell)} \big(\A^{(a, b, j, \ell)}_i - \A^{(a, b, j, \ell)}_{i-1}\big) \vee \big(-\B^{(a, b, j, \ell)}_i + \B^{(a, b, j, \ell)}_{i-1}\big)  > (\log\ell)^2\Big) \Big)_{a = \A^{\ell}_{\tilde\tau_\ell + 1}, b = \B^{\ell}_{\tilde\tau_\ell + 1}, j = \tilde\tau_\ell + 1} \Big] \\
        &\leq \frac{g(\ell) \mu_{\max} \lambda_{\max} (\ell^{\alpha} + 2 g(\ell) (\log\ell)^2)}{\epsilon \ell^3} + \mathbb{E}\Big[ \mathbbm{1}\big\{\A^{\ell}_{\tilde\tau_\ell + 1} - \B^{\ell}_{\tilde\tau_\ell + 1} < \ell^{\alpha}\big\}\\
        &\qquad\Big(\P\Big(\max_{1 \leq i \leq g(\ell)} \big(\A^{(a, b, j, \ell)}_i - \A^{(a, b, j, \ell)}_{i-1}\big) \vee \big(-\B^{(a, b, j, \ell)}_i + \B^{(a, b, j, \ell)}_{i-1}\big)  > (\log\ell)^2\Big) \Big)_{a = \A^{\ell}_{\tilde\tau_\ell + 1}, b = \B^{\ell}_{\tilde\tau_\ell + 1}, j = \tilde\tau_\ell + 1} \Big],
    \end{align*} 
   where we used Markov's inequality to obtain the first term on the right-hand side. Now, the first summand on the right-hand side is bounded by $(3/\epsilon) (\log \ell)^2 \ell^{4\alpha - 3}$, which tends to zero as $\ell \to \infty$, and the second summand on the right-hand side is bounded by $2\ell^{2\alpha} e^{-2\lambda_{\min}(\log \ell)^2}$, which also tends to zero as $\ell \to \infty$. This completes the proof.
\end{proof}

\subsection{Main result for vertically inhomogeneous semi lattices}
\begin{proof}[Proof of Theorem~\ref{Thm:Poisson:convergence-Tl-taul}: vertically inhomogeneous case]
We already have
\begin{align*}
    \Big(  \frac{\tilde\tau_\ell}{\ell^2}, \frac{\tilde\T_{\ell}}{\ell^3} 
     \Big)\xlongrightarrow{d}
\Big( \vartheta, \int_{0}^{\vartheta} (1+M_x)\lambda(x,0)\mu(x,0)\,\dd x \Big),\qquad\text{ as }\ell\to\infty,
\end{align*}  
by Proposition~\ref{Prop:convergence-in-dist-upto-tauC}.
But \(\ell^{-2}\big|\tilde\tau_\ell - \boldsymbol{\mathsf \tau}_{\ell} \big| \xlongrightarrow{p} 0\) by Lemma~\ref{le:tau_C_vs_tau_0_inhom} and \(\ell^{-3} \big|\tilde\T_{\ell} - \T_{\ell} \big| \xlongrightarrow{p} 0\) by Lemma~\ref{le:small}, so that
\begin{align*}
    \Big( \frac{\boldsymbol{\mathsf \tau}_{\ell}}{\ell^2}, \frac{ \T_{\ell}}{\ell^3}\Big) \xlongrightarrow{ d } \Big(  \vartheta, \int_{0}^{\vartheta} \big(1+ M_{x}\big) \lambda(x,0)\mu(x,0) \,\dd x \Big),\qquad\text{ as }\ell\to\infty,
\end{align*} is implied by Slutsky's Theorem.
\end{proof}

%%%%%%%%%%%%%%%%%%%%%%%%%%%%%%%%%%%%%
%%%%%%%%%%%%%%%%%%%%%%%%%%%%%%%%%%%%%
%%%%%%%%%%%%%%%%%%%%%%%%%%%%%%%%%%%%%

\subsection{Proofs for diluted lattices}\label{Sec:Proofs_DL}
As in the case of the semi lattice, we first address the vertically homogeneous diluted lattices; the results can then be extended to the vertically inhomogeneous case using similar approximation techniques as those employed for the semi lattices. Without loss of generality, we assume that $\ell/2$ is an  odd positive integer.

As in the case of semi lattices, our starting point here is again the following definition of upper and lower bounding paths. We set $A_0^\ell := \ell/2$ and $B_0^\ell := -\ell/2$, and define them recursively for each $i \geq 0$ as follows,
\begin{align*}
    \overline{U}^\ell_i &:= \inf\{y > A^\ell_i \colon (i,y)\in\Oi \}, &
    \overline{L}^\ell_i &:= \inf\{y \geq B^\ell_i \colon (i,y)\in\Oi \}, \\
    \underline{U}^\ell_i &:= \sup\{y \leq A^\ell_i \colon (i,y)\in\Oi \}, &
    \underline{L}^\ell_i &:= \sup\{y < B^\ell_i \colon (i,y)\in\Oi \}, 
\end{align*}
and
\[
    \overline{X}^\ell_i := (\overline{U}^\ell_i- A^\ell_i+1)/2,\quad \underline{X}^\ell_i :=  (A^\ell_i - \underline{U}^\ell_i+1)/2,\quad
    \overline{Y}^\ell_i := (\overline{L}^\ell_i-  B^\ell_i+1)/2 , \quad 
    \underline{Y}^\ell_i := (B^\ell_i - \underline{L}^\ell_i+1)/2.
\]
and then we update the bounding paths according to the following rules.
\begin{align*}
    A^\ell_{i+1}:&= A^\ell_i+ \begin{cases}
        \overline{X}^\ell_i - \underline{X}^\ell_i  & \text{ if } \overline{X}^\ell_i - \underline{X}^\ell_i \text{ is odd};\\[.25cm]
        \overline{X}^\ell_i - \underline{X}^\ell_i +1 & \text{ if } \overline{X}^\ell_i - \underline{X}^\ell_i \text{ is even and }\xi_{( i+1, A^\ell_i+ \overline{X}^\ell_i - \underline{X}^\ell_i)}=0;\\[.25cm]
        \overline{X}^\ell_i - \underline{X}^\ell_i -1 & \text{ if } \overline{X}^\ell_i - \underline{X}^\ell_i \text{ is even and }  \xi_{( i+1, A^\ell_i+ \overline{X}^\ell_i - \underline{X}^\ell_i)}=1;
    \end{cases}\\[.25cm]
     B^\ell_{i+1}:&= B^\ell_i+ \begin{cases}
        \overline{Y}^\ell_i - \underline{Y}^\ell_i  & \text{ if } \overline{Y}^\ell_i - \underline{Y}^\ell_i \text{ is odd};\\[.25cm]
        \overline{Y}^\ell_i - \underline{Y}^\ell_i +1 & \text{ if } \overline{Y}^\ell_i - \underline{Y}^\ell_i \text{ is even and }\xi_{( i+1, B^\ell_i+ \overline{Y}^\ell_i - \underline{Y}^\ell_i)}=0;\\[.25cm]
        \overline{Y}^\ell_i - \underline{Y}^\ell_i -1 & \text{ if } \overline{Y}^\ell_i - \underline{Y}^\ell_i \text{ is even and }  \xi_{( i+1, B^\ell_i+ \overline{Y}^\ell_i - \underline{Y}^\ell_i)}=1.
    \end{cases}
\end{align*}

Note that the sequences $\{A^\ell_{i+1} - A^\ell_i\colon i \geq 0\}$ and $\{B^\ell_{i+1} - B^\ell_i\colon i \geq 0\}$ consist of independent random variables. For each $i \geq 0$, the variables $\overline{X}^\ell_i$, $\underline{X}^\ell_i$, $\overline{Y}^\ell_i$, and $\underline{Y}^\ell_i$ are  $\Geo(p(i/\ell^2))$ random variables.

As earlier, we let 
\[
  \tau_{\ell} = \inf \big\{ i \geq 0 \colon \#\big( (\{i\} \times [B^\ell_i, A^\ell_i]) \cap \Oi \big) = 0 \big\}\quad\text{ and }\quad  \tau'_{\ell} := \inf \big\{ i \geq 0 \colon \#\big( (\{i\} \times [B^\ell_i, A^\ell_i]) \cap \Oi \big) \leq 1 \big\}.
\]
As before, define the total traffic accumulated prior to times $\tau_\ell$ and $\tau'_\ell$, respectively via
\[
T_\ell = \sum_{i=0}^{\tau_0^\ell - 1} \#\big( (\{i\} \times [B^\ell_i, A^\ell_i]) \cap \Oi \big)\mu(i/\ell^2)\quad\text{ and }\quad
T'_\ell := \sum_{i=0}^{\tau_1^\ell - 1} \#\big( (\{i\} \times [B^\ell_i, A^\ell_i]) \cap \Oi \big)\mu(i/\ell^2).
\]
As earlier, to analyze these quantities, we couple the processes $\{A^{\ell}_i\colon i \geq 0\}$ and $\{B^{\ell}_i\colon i \geq 0\}$ with two independent processes $\{\hat{A}^{\ell}_i\colon i \geq 0\}$ and $\{\hat{B}^{\ell}_i\colon i \geq 0\}$ as follows. 
Let $\{\hat{\overline{X}}^\ell_i\colon i \geq 0\}$, $\{\hat{\underline{X}}^\ell_i\colon i \geq 0\}$, $\{\hat{\overline{Y}}^\ell_i\colon i \geq 0\}$, and $\{\hat{\underline{Y}}^\ell_i\colon i \geq 0\}$ be four independent sequences of independent random variables. For all $i \geq 0$, the random variables $\hat{\overline{X}}^\ell_i$, $\hat{\underline{X}}^\ell_i$, $\hat{\overline{Y}}^\ell_i$, and $\hat{\underline{Y}}^\ell_i$ are i.i.d.\ $\Geo(p(i/\ell^2))$. In addition, let $\{\hat{\xi}_{\x}\colon \x\in\ZZe\}$ and $\{\hat{\xi}'_{\x}\colon \x\in\ZZe\}$ be two independent sequences consisting of i.i.d.\ $\Ber(1/2)$ random variables each, with these six sequences being independent of each other. Then, setting ${\hat A}^\ell_0=\ell/2$, ${\hat B}^\ell_0=\ell/2$, we define  for all $i\geq 0$,
\begin{align*}
    {\hat A}^\ell_{i+1}:&= {\hat A}^\ell_i+ \begin{cases}
        \hat{{\overline X}}^\ell_i - \hat{{\underline X}}^\ell_i  & \text{ if } \hat{{\overline X}}^\ell_i - \hat{{\underline X}}^\ell_i \text{ is odd};\\[.25cm]
        \hat{{\overline X}}^\ell_i - \hat{{\underline X}}^\ell_i +1 & \text{ if } \hat{{\overline X}}^\ell_i - \hat{{\underline X}}^\ell_i \text{ is even and }\xi_{( i+1, {\hat A}^\ell_i+ \hat{{\overline X}}^\ell_i - \hat{{\underline X}}^\ell_i )}=0;\\[.25cm]
        \hat{{\overline X}}^\ell_i - \hat{{\underline X}}^\ell_i -1 & \text{ if } \hat{{\overline X}}^\ell_i - \hat{{\underline X}}^\ell_i \text{ is even and }  \xi_{( i+1, {\hat A}^\ell_i+ \hat{{\overline X}}^\ell_i - \hat{{\underline X}}^\ell_i )}=1;
    \end{cases}\\[.25cm]
     {\hat B}^\ell_{i+1}:&= {\hat B}^\ell_i+ \begin{cases}
        \hat{{\overline Y}}^\ell_i - \hat{{\underline Y}}^\ell_i  & \text{ if } \hat{{\overline Y}}^\ell_i - \hat{{\underline Y}}^\ell_i \text{ is odd};\\[.25cm]
        \hat{{\overline Y}}^\ell_i - \hat{{\underline Y}}^\ell_i +1 & \text{ if } \hat{{\overline Y}}^\ell_i - \hat{{\underline Y}}^\ell_i \text{ is even and }\xi_{( i+1, {\hat B}^\ell_i+ \hat{{\overline Y}}^\ell_i - \hat{{\underline Y}}^\ell_i )}=0;\\[.25cm]
        \hat{{\overline Y}}^\ell_i - \hat{{\underline Y}}^\ell_i -1 & \text{ if } \hat{{\overline Y}}^\ell_i - \hat{{\underline Y}}^\ell_i \text{ is even and }  \xi_{( i+1, {\hat B}^\ell_i+ \hat{{\overline Y}}^\ell_i - \hat{{\underline Y}}^\ell_i )}=1.
    \end{cases}
\end{align*}
We also define
\begin{align*}
    \hat \tau_\ell:= \inf\Big\{i\geq 0\colon {\hat A}^\ell_i- 2\hat{{\underline X}}^\ell_i< {\hat B}^\ell_i+2\hat{{\overline Y}}^\ell_i \Big\}.
\end{align*}
Then, analogous to the semi-lattice case, we have
\begin{align}
    &\Big(\hat \tau_\ell , \big\{\big(\hat A^\ell_{i+1}, \hat B^\ell_{i+1}, \hat{\overline{X}}^\ell_i, \hat{\underline{X}}^\ell_i, \hat{\overline{Y}}^\ell_i, \hat{\underline{Y}}^\ell_i\big)\ind{i< \hat \tau_\ell }\colon i\geq 0 \big\} \Big) \nonumber\\
    &\qquad\qquad\xlongequal{ d } 
    \Big(\tau'_{\ell}, \big\{\big( A^\ell_{i+1},  B^\ell_{i+1}, {\overline{X}}^\ell_i, {\underline{X}}^\ell_i, {\overline{Y}}^\ell_i, {\underline{Y}}^\ell_i\big)\ind{i< \tau'_{\ell} }\colon i\geq 0  \big\}  \Big).
\end{align}

As before, we define a sequence of processes $\{W_\ell\colon\ell \geq 1\}$, where each $W_\ell = \{W_{\ell,t} \colon t \geq 0\}$ is given by
\begin{align}
    W_{\ell,t} := \frac{1}{\ell} \big( D^\ell_{\lfloor \ell^2 t \rfloor} + ( \ell^2 t - \lfloor \ell^2 t \rfloor) \big( D^\ell_{\lfloor \ell^2 t \rfloor + 1} - D^\ell_{\lfloor \ell^2 t \rfloor} \big) \big),
\end{align}
and where, for all $i \geq 0$, we define $D^\ell_i := \hat{A}^\ell_i - \hat{B}^\ell_i$. Then, we have the following result.
\begin{lemma}
\label{Lem:Variance-of-increments-of-D-in-Diluted-lattice}
$\Var[D^\ell_{i+1}-D^\ell_i]= \beta_{p(i/\ell^2,0)}$,
where $\beta_q := \frac{q^4 +(2-q)^4 }{q^2(2-q)^2}$.
\end{lemma}
\begin{proof}[Proof of Lemma~\ref{Lem:Variance-of-increments-of-D-in-Diluted-lattice}]
To simplify notation, we denote $p(i/\ell^2)$ by $p$.
Since the increments of $\{{\hat A}^\ell_i\colon i \geq 0\}$ and $\{{\hat B}^\ell_i\colon i \geq 0\}$ are independent, and, for all $i \geq 0$, the random variables ${\hat A}^\ell_{i+1} - {\hat A}^\ell_i$ and ${\hat B}^\ell_{i+1} - {\hat B}^\ell_i$ are identically distributed, it follows that
  $ \Var[D^{\ell}_{i+1}-D^{\ell}_{i}]= 2 \Var[{\hat A}^\ell_{i+1}- {\hat A}^\ell_i]$. Observe that, for any $k\geq 0$,
\begin{align*}
    \P\big(\hat{{\overline X}}^\ell_i-\hat{{\underline X}}^\ell_i= -k\big) &=\P\big(\hat{{\overline X}}^\ell_i-\hat{{\underline X}}^\ell_i = k\big)= \sum_{j=1}^{\infty} \P\big(\hat{{\underline X}}^\ell_i=j\big)\P\big(\hat{{\overline X}}^\ell_i=j+k\big)\\
    &=  \sum_{j=1}^{\infty} (1-p)^{j-1}p(1-p)^{j+k-1}p= (1-p)^{k}p^2\sum_{j=1}^{\infty} (1-p)^{2j-2}
    =\frac{(1-p)^{k}p^2}{1-(1-p)^2}=\frac{(1-p)^{k}p}{2-p}.
\end{align*}
Now, for all $k\geq 0$,
\begin{align*}
    \P\big(|{\hat A}^\ell_{i+1}-{\hat A}^\ell_i|+1=2(k+1)\big)
    &=  \P({\hat A}^\ell_{i+1}-{\hat A}^\ell_i=2k+1)+ \P({\hat A}^\ell_{i+1}-{\hat A}^\ell_i=-2k-1)\\
    &= 2\P({\hat A}^\ell_{i+1}-{\hat A}^\ell_i=2k+1)\\
    &= 2\P(\hat{{\overline X}}^\ell_i-\hat{{\underline X}}^\ell_i= 2k+1)+\P(\hat{{\overline X}}^\ell_i-\hat{{\underline X}}^\ell_i= 2k)+\P(\hat{{\overline X}}^\ell_i-\hat{{\underline X}}^\ell_i= 2k+2)\\
   &= \frac{2(1-p)^{2k+1}p}{2-p}+\frac{(1-p)^{2k}p}{2-p}+\frac{(1-p)^{2k+2}p}{2-p}\\
   &= \frac{(1-p)^{2k}p}{2-p} \Big( 2(1-p)+1 +(1-p)^2 \Big)\\
   &= \frac{(1-p)^{2k}p}{2-p} (2-p)^2 =  (1-p)^{2k}p(2-p)= (1-p)^{2k} \big( 1- (1-p)^2 \big),
\end{align*}
which means that $(|{\hat A}^\ell_{i+1}-{\hat A}^\ell_i|+1)/2  \sim \Geo(1- (1-p)^2)$. So, we have
\[
\Var\Big[ \frac{|{\hat A}^\ell_{i+1}-{\hat A}^\ell_i|+1}{2} \Big] = \frac{(1-p)^2}{(1- (1-p)^2)^2}\quad\text{ and }\quad
\E\Big[ \frac{|{\hat A}^\ell_{i+1}-{\hat A}^\ell_i|+1}{2} \Big] = \frac{1}{1- (1-p)^2}.
\]
Therefore, we get
\begin{align*}
    \Var[{\hat A}^\ell_{i+1}-{\hat A}^\ell_i]
    &= \E\big[ ({\hat A}^\ell_{i+1}-{\hat A}^\ell_i)^2 \big]
    =\Var\big[|{\hat A}^\ell_{i+1}-{\hat A}^\ell_i|\big] + \big(\E\big[|{\hat A}^\ell_{i+1}-{\hat A}^\ell_i|\big] \big)^2\\
    &=\frac{4(1-p)^2}{(1- (1-p)^2)^2}+ \Big( \frac{1+ (1-p)^2}{1- (1-p)^2}\Big)^2
    =\frac{(1-p)^4 +6(1-p)^2 +1}{(1- (1-p)^2)^2}\\
    &=\frac{(1-(1-p))^4 +(1+(1-p))^4}{2p^2(2-p)^2}=\frac{p^4 +(2-p)^4 }{2p^2(2-p)^2}.
\end{align*}
But this implies that
\[
\Var[D^\ell_{i+1}-D^\ell_i]= 2\Var[{\hat A}^\ell_{i+1} -{\hat A}^\ell_i]
=\frac{p^4 +(2-p)^4 }{p^2(2-p)^2} =\beta_p,
\] 
as desired.
\end{proof}

We now state the following result, whose proof follows from arguments similar to those used in the proof of Lemma~\ref{Lem:Wlt-weak-convergence}, in combination with Lemma~\ref{Lem:Variance-of-increments-of-D-in-Diluted-lattice}.
\begin{lemma} 
As $\ell \to \infty$, the process $W_\ell=\{W_{\ell,t} \colon t \geq 0\}$ converges weakly to 
    \[
    1+M_{\cdot}=\big\{1+M_t \colon t\geq 0\big\}:=
    \Big\{1+\int_{0}^t \sqrt{\beta_{p(s)}} \,\dd B_s\colon t \geq 0\Big\}.
    \]
\end{lemma}

\begin{proof}[Proof of Theorem~\ref{Thm:diluted:convergence-Tl-taul_gen}]
In the diluted lattice case, given $A^\ell_i$ and $B^\ell_i$, the expected number of occupied vertices between them is $(A^\ell_i - B^\ell_i)\, p(i/\ell^2)/2$, in contrast to $(A^\ell_i - B^\ell_i)\, \lambda(i/\ell^2)$ in the semi-lattice case. This discrepancy introduces an additional factor of $1/2$ in the limiting expression appearing in Theorem~\ref{Thm:diluted:convergence-Tl-taul_gen}. The remainder of the proof follows by an argument analogous to that used in the proof of Theorem~\ref{Thm:Poisson:convergence-Tl-taul}.
\end{proof}

\begin{proof}[Proof of Theorem~\ref{Thm:diluted:asymp-taul}]
The proof of Theorem~\ref{Thm:diluted:asymp-taul} follows by an argument analogous to the one used in the proof of Theorem~\ref{Thm:Poisson:asymp-taul}.
\end{proof}

%%%%%%%%%%%%%%%%%%%%%%%%%%%%%%%%%%%%%
%%%%%%%%%%%%%%%%%%%%%%%%%%%%%%%%%%%%%
%%%%%%%%%%%%%%%%%%%%%%%%%%%%%%%%%%%%%

\subsection{Proofs for pure lattices}\label{Sec:Proofs_PL}
As in the diluted lattice case, here too we assume, without loss of generality, that $\ell/2$ is a positive odd integer.
In this model, we have a stronger version of coupling  as follows. 
Let $\{\hat{\xi}_{\x}\colon \x\in\ZZe\}$ and $\{\hat{\xi}'_{\x}\colon \x\in\ZZe\}$ be two independent sequences consisting of  i.i.d.\ $\Ber(1/2)$ random variables each. Then, setting ${\hat A}^\ell_0=\ell/2$, ${\hat B}^\ell_0=\ell/2$, we define  for all $i\geq 0$,
\begin{align*}
    \hat{A}^{\ell}_{i+1}:&= \hat{A}^{\ell}_i+ \begin{cases}
        1 & \text{ if  $\hat \xi_{(i+1, \hat{A}^{\ell}_i )}=0$};\\
       -1 & \text{ if $\hat \xi_{(i+1, \hat{A}^{\ell}_i )}=1$};
    \end{cases}
    %\\[.25cm]
   \qquad \qquad \hat{B}^{\ell}_{i+1}:= \hat{B}^{\ell}_i+ \begin{cases}
        1 & \text{ if $\hat \xi'_{(i+1, \hat{B}^{\ell}_i) }=0$};\\
         -1 & \text{ if  $\hat \xi'_{(i+1, \hat{B}^{\ell}_i) }=1$};
    \end{cases}
\end{align*}
and 
\begin{align*}
    \hat \tau_\ell:= \inf\{i\geq 0 \colon \hat{A}^{\ell}_i-  \hat{B}^{\ell}_i=0 \}.
\end{align*}
Then we have  $\tau_{\ell} \xlongequal{ d }  \hat \tau_\ell $.

\begin{proof}[Proof of Theorem~\ref{Thm:Pure:asymp-taul}]
Writing $D_i:=\hat{A}^{\ell}_i - \hat{A}^{\ell}_i$, we see that $\{D_i\colon i\geq0\}$ is a random walk with $D_0=\ell$, whose increments are  $2$, $0$, or $-2$.
So we have
\begin{align*}
\P(\tau_{\ell} = n) = \P(\hat \tau_\ell=n)
&= \P( D_n=0 \text{ and } D_i>0\text{ for all $1\leq i\leq n-1$} | D_0=\ell)\\
&= \P( D_n=\ell \text{ and } D_i>0\text{ for all $1\leq i\leq n-1$} | D_0=0)\\
&=  \P( D_{n}=\ell | D_0=0)\P(  D_i >0 \text{ for all $1\leq i\leq n-1$} | D_0=0, D_{n}=\ell).
\end{align*}
Now, since the increments of $\{D_i/2\colon i \geq 0\}$ are integer-valued with a maximum value of $1$, it follows from~\cite[Theorem~2]{AdRe08} that
\[
\P(  D_i >0 \text{ for all $1\leq i\leq n-1$} | D_0=0, D_{n}=\ell)=\ell/(2n).
\]
We also have 
\begin{align*}
 \P(D_{n}=\ell | D_0=0) &= \text{coefficient of $x^\ell$ in the expansion of } \Big( \frac{x^2}{4}+\frac{1}{2}+ \frac{1}{4x^2} \Big)^n\\
&= \text{coefficient of $x^{\ell/2}$ in the expansion of } \Big( \frac{x}{4}+\frac{1}{2}+ \frac{1}{4x} \Big)^n\\
&= \text{coefficient of $x^{n+\ell/2}$ in the expansion of } \Big( x^2+2x+ 1 \Big)^n\frac{1}{4^n}\\
&= \text{coefficient of $x^{n+\ell/2}$ in the expansion of } \Big( 1+x \Big)^{2n}\frac{1}{4^n}= {\binom{2n}{n+\ell/2}}\frac{1}{4^n}.
\end{align*}
This implies,
\[
\P(\tau_{\ell} = n) = \frac{\ell}{2n} {\binom{2n}{n+\ell/2}}\frac{1}{4^n}.
\]
Now, by Stirling's approximation,  for any $f$ satisfying $\ell^2/f(\ell)\to0$ as $\ell\to \infty$ and for every $\epsilon>0$, there exists $\ell_{(\epsilon,f)}>0$ such that, for all $\ell\geq \ell_{(\epsilon,f)}$ and for all $n\geq f(\ell)$, 
\[
\frac{n^{3/2}}{\ell}\P(\tau_\ell =n) \in \Big(\frac{1}{2{\sqrt \pi}} -\epsilon, \frac{1}{2{\sqrt \pi}} +\epsilon\Big).
\]
This further implies that
\[
\lim_{\ell\to\infty}\frac{\sqrt{f(\ell)}}{\ell}\P\big(\tau_\ell \geq f(\ell)\big) = \frac{1}{{4 \sqrt \pi}}.
\]
This proves the result.
\end{proof}

%%%%%%%%%%%%%%%%%%%%%%%%%%%%%%%%%%%%%%
%%%%%%%%%%%%%%%%%%%%%%%%%%%%%%%%%%%%%%
%%%%%%%%%%%%%%%%%%%%%%%%%%%%%%%%%%%%%%

\section*{Acknowledgement}
The authors would like to thank Sebastian Andres, Daniel Kamecke and Kumarjit Saha  for many fruitful discussions.
This research was supported by the Leibniz Association within the Leibniz Junior Research Group on {\em Probabilistic Methods for Dynamic Communication Networks} as part of the Leibniz Competition (grant no.\ J105/2020) and the Berlin Cluster of Excellence {\em MATH+} through the project {\em EF45-3} on {\em Data Transmission in Dynamical Random Networks}.

%%%%%%%%%%%%%%%%%%%%%%%
%%%%%Bibliography%%%%%%
%%%%%%%%%%%%%%%%%%%%%%%

%\bibliographystyle{plain}	
%\bibliography{bib.bib}

\begin{thebibliography}{10}

\bibitem{AdRe08}
L.~Addario-Berry and B.A. Reed.
\newblock Ballot theorems, old and new.
\newblock In {\em Horizons of Combinatorics}, volume~17 of {\em Bolyai Soc.
  Math. Stud.}, pages 9--35. Springer, Berlin, 2008.

\bibitem{arak1975}
T.V. Arak.
\newblock On the distribution of the maximum of successive partial sums of
  independent random variables.
\newblock {\em Theory Probab. Appl.}, 19(2):245--266, 1975.
\newblock English translation of \textit{Teor. Verojatnost. i Primenen.} 19
  (1974).

\bibitem{arratia1979coalescing}
R.A. Arratia.
\newblock {\em Coalescing Brownian Motions on the Line}.
\newblock The University of Wisconsin-Madison, 1979.

\bibitem{arratia1981coalescing}
R.A. Arratia.
\newblock Coalescing brownian motions and the voter model on {Z}.
\newblock {\em Unpublished partial manuscript. Available from
  rarratia@math.usc.edu}, 1981.

\bibitem{Billingsley99}
P.~Billingsley.
\newblock {\em Convergence of Probability Measures}.
\newblock Wiley Series in Probability and Statistics: Probability and
  Statistics. John Wiley \& Sons, Inc., New York, second edition, 1999.

\bibitem{coletti2009scaling}
C.~Coletti, L.~Fontes, and E.~Dias.
\newblock Scaling limit for a drainage network model.
\newblock {\em Journ. Appl. Probab.}, 46(4):1184--1197, 2009.

\bibitem{ferrari2005two}
P.~Ferrari, L.~Fontes, and X.-Y. Wu.
\newblock Two-dimensional {P}oisson trees converge to the {B}rownian web.
\newblock {\em Ann. inst. Henri Poincare (B) Probab. Stat.}, 41(5):851--858,
  2005.

\bibitem{ferrari2004poisson}
P.~Ferrari, C.~Landim, and H.~Thorisson.
\newblock Poisson trees, succession lines and coalescing random walks.
\newblock In {\em Ann. inst. Henri Poincare (B) Probab. Stat.}, volume~40,
  pages 141--152, 2004.

\bibitem{fontes2004brownian}
L.~Fontes, M.~Isopi, C.M. Newman, and K.~Ravishankar.
\newblock The {B}rownian web: characterization and convergence.
\newblock {\em Ann. Probab.}, 32(4):2857--2883, 2004.

\bibitem{gangopadhyay2004random}
S.~Gangopadhyay, R.~Roy, and A.~Sarkar.
\newblock Random oriented trees: a model of drainage networks.
\newblock {\em Ann. Appl. Probab.}, 14(3):1242--1266, 2004.

\bibitem{hirsch2017traffic}
C.~Hirsch, B.~Jahnel, P.~Keeler, and R.I.A. Patterson.
\newblock Traffic flow densities in large transport networks.
\newblock {\em Adv. in Appl. Probab.}, 49(4):1091--1115, 2017.

\bibitem{KeMa05}
M.J. Kearney and S.N. Majumdar.
\newblock On the area under a continuous time {B}rownian motion till its
  first-passage time.
\newblock {\em J. Phys. A}, 38(19):4097--4104, 2005.

\bibitem{keeler2010stochastic}
P.~Keeler and P.~Taylor.
\newblock A stochastic analysis of a greedy routing scheme in sensor networks.
\newblock {\em SIAM J. Appl. Math.}, 70(7):2214--2238, 2010.

\bibitem{penrose2010limit}
M.D. Penrose and A.R. Wade.
\newblock Limit theorems for random spatial drainage networks.
\newblock {\em Adv. in Appl. Probab.}, 42(3):659--688, 2010.

\bibitem{RevYor99}
D.~Revuz and M.~Yor.
\newblock {\em Continuous {M}artingales and {B}rownian {M}otion}, volume 293 of
  {\em Grundlehren der mathematischen Wissenschaften}.
\newblock Springer-Verlag, Berlin, third edition, 1999.

\bibitem{RoSaSa16}
R.~Roy, K.~Saha, and A.~Sarkar.
\newblock Hack's law in a drainage network model: a {B}rownian web approach.
\newblock {\em Ann. Appl. Probab.}, 26(3):1807--1836, 2016.

\end{thebibliography}

\end{document}